\documentclass[10pt,leqno]{article}

\usepackage[%
  tmargin=1in,bmargin=1in,%
  lmargin=1.3in,rmargin=1.3in,%
  marginparsep=0.2in, marginparwidth=1.0in%
]{geometry}
\AtBeginDocument{%
  \setlength{\abovedisplayskip}{1.5ex plus 0.3ex minus 0.3ex}%
  \setlength{\abovedisplayshortskip}{1.0ex plus 0.2ex minus 0.2ex}%
  \setlength{\belowdisplayskip}{1.5ex plus 0.3ex minus 0.3ex}%
  \setlength{\belowdisplayshortskip}{1.0ex plus 0.2ex minus 0.2ex}%
}

\usepackage{datetime2}
\usepackage{fancyhdr}
\usepackage[hyphens]{url}
\usepackage[style=numeric,backend=bibtex8,maxbibnames=100]{biblatex}
\usepackage[newparttoc]{titlesec}
\usepackage[dotinlabels]{titletoc}
\usepackage{appendix}
\usepackage[letterspace=50]{microtype}
\usepackage[bottom]{footmisc}
\usepackage[inline]{enumitem}
\usepackage{xspace}
\usepackage{calc}
\usepackage{etoolbox}
\usepackage{ifthen}
\usepackage{tikz}
\usepackage{tikz-cd}
\usepackage{textcase}

\definecolor{cite}{rgb}{0.45,0.55,0.6}
\definecolor{link}{rgb}{0.45,0.55,0.6}
\definecolor{url}{rgb}{0.4,0.4,0.4}
\usepackage[%
  hyperfootnotes=false,
  colorlinks,%
  linkcolor=link,%
  citecolor=cite,%
  urlcolor=url,%
]{hyperref}
\DeclareCiteCommand{\citebare}
  {\usebibmacro{prenote}}
  {\usebibmacro{citeindex}%
   \usebibmacro{cite}}
  {\multicitedelim}
  {\usebibmacro{postnote}}

\renewbibmacro{in:}{}
\DeclareFieldFormat*{title}{\mkbibquote{#1}}
\DeclareFieldFormat*{booktitle}{\mkbibitalic{#1}}
\DeclareFieldFormat*{journaltitle}{#1\isdot}
\DeclareFieldFormat*{volume}{#1}
\DeclareFieldFormat*[online,misc,incollection,inproceedings]{date}{\mkbibparens{#1}}
\DeclareFieldInputHandler{pages}{}

\newtoggle{bbx:series}
\newtoggle{bbx:volume}

\DeclareBiblatexOption{global,type,entry}[boolean]{series}[true]{%
  \settoggle{bbx:series}{#1}}
\DeclareBiblatexOption{global,type,entry}[boolean]{volume}[true]{%
  \settoggle{bbx:volume}{#1}}

\newbibmacro*{addendum+pubstate+date}{%
  \printfield{addendum}%
  \newunit\newblock
  \printfield{pubstate}
  \usebibmacro{date}}

\newbibmacro*{doi+eprint+url+date}{%
  \iftoggle{bbx:doi}
  {\printfield{doi}}
  {}%
  \newunit\newblock
  \iftoggle{bbx:eprint}
  {\usebibmacro{eprint}}
  {}%
  \iftoggle{bbx:url}
  {\usebibmacro{url+urldate}}
  {}
  \usebibmacro{date}}

\DeclareBibliographyDriver{incollection}{%
  \usebibmacro{bibindex}%
  \usebibmacro{begentry}%
  \usebibmacro{author}%
  \newunit
  \usebibmacro{title}%
  \newunit
  \printfield{booktitle}%
  \newunit
  \iftoggle{bbx:series}{\printfield{series}}%
  {}
  \iftoggle{bbx:volume}{\printfield{volume}}%
  {}
  \usebibmacro{date}%
  \usebibmacro{finentry}}

\DeclareBibliographyDriver{inproceedings}{%
  \usebibmacro{bibindex}%
  \usebibmacro{begentry}%
  \usebibmacro{author}%
  \newunit
  \usebibmacro{title}%
  \newunit
  \printfield{booktitle}%
  \newunit
  \iftoggle{bbx:series}{\printfield{series}}%
  {}
  \iftoggle{bbx:volume}{\printfield{volume}}%
  {}
  \usebibmacro{date}%
  \usebibmacro{finentry}}

\DeclareBibliographyDriver{online}{%
  \usebibmacro{bibindex}%
  \usebibmacro{begentry}%
  \usebibmacro{author}%
  \newunit
  \usebibmacro{title}%
  \newunit
  \usebibmacro{doi+eprint+url+date}%
  \usebibmacro{finentry}}

\DeclareBibliographyDriver{misc}{%
  \usebibmacro{bibindex}%
  \usebibmacro{begentry}%
  \usebibmacro{author}%
  \newunit
  \usebibmacro{title}%
  \newunit
  \printfield{howpublished}%
  \newunit
  \printfield{type}%
  \newunit
  \printfield{version}%
  \newunit
  \printfield{note}%
  \newunit
  \usebibmacro{doi+eprint+url}%
  \newunit
  \usebibmacro{addendum+pubstate+date}%
  \usebibmacro{finentry}}

\DeclareFieldFormat{eprint:ktheory}{%
  K-theory preprint archives: \href{https://faculty.math.illinois.edu/K-theory/#1}{ #1}%
}

\usepackage{amsmath,amsthm}

\usepackage[T1]{fontenc}
\usepackage{lmodern}
\usepackage{MnSymbol}
\usepackage[bbgreekl]{mathbbol}
\DeclareSymbolFontAlphabet{\mathbbl}{bbold}

\usepackage[%
  cal=euler,  calscaled=1.0,%
  bb=boondox, bbscaled=1.0,%
  frak=euler, frakscaled=1.0,%
]{mathalfa}  

\makeatletter
\g@addto@macro\bfseries{\boldmath}
\makeatother

\usepackage[textsize=footnotesize,backgroundcolor=orange!50]{todonotes}
\usepackage[compress,capitalize]{cleveref}

\let\theoldbibliography\thebibliography
\renewcommand{\thebibliography}[1]{%
  \theoldbibliography{#1}%
  \setlength{\parskip}{0ex}
  \setlength{\itemsep}{0.5ex plus 0.2ex minus 0.2ex}
  \small
}
\apptocmd{\thebibliography}{\raggedright}{}{}

\renewcommand{\title}[1]{\newcommand{\thetitle}{#1}}
\renewcommand{\author}[1]{\newcommand{\theauthor}{#1}}

\renewcommand{\maketitle}{%
  \begin{center}
    {\linespread{1.15}%
      \bfseries\MakeTextUppercase%
      \thetitle\par} \bigskip
    \footnotesize
    {\MakeUppercase \theauthor}
  \end{center}
  \bigskip
  \thispagestyle{fancy}
}

\renewenvironment{abstract}{\noindent\begin{center}\begin{minipage}{0.8\linewidth}\small{\scshape Abstract.}}{\end{minipage}\end{center}\bigskip}

\newlength{\tagsep}
\makeatletter
\def\fullwidthdisplay{\displayindent\z@ \displaywidth\columnwidth}
\edef\@tempa{\noexpand\fullwidthdisplay\the\everydisplay}
\everydisplay\expandafter{\@tempa}
\makeatother

\titleformat{\part}{\centering\titlerule\vspace{1.0ex}\bfseries}{Part \thepart.}{1.5\tagsep}{}[\vspace{1.0ex}\titlerule]
\titleformat{\section}{\centering}{\textsection\thesection.}{1.5\tagsep}{\scshape}
\titleformat{\subsection}[runin]{}{\fontseries{b}\selectfont\textsection\bfseries\thesubsection.}{1.5\tagsep}{\bfseries}[.]
\titleformat{\subsubsection}[runin]{}{\textsection\thesubsubsection.}{\tagsep}{\itshape}[.]

\titlespacing*{\part}{0pt}{0ex}{4ex}
\titlespacing*{\section}{0pt}{4ex}{\medskipamount}
\titlespacing*{\subsection}{0pt}{\bigskipamount}{0.5em}
\titlespacing*{\subsubsection}{0pt}{\medskipamount}{0.5em}

\titlecontents{part}[0em]{\vspace{1.25ex}}{Part \thecontentslabel.\hspace{0.5em}}{}{}
\dottedcontents{section}[1.3em]{\vspace{0.25ex}}{1.3em}{0.7em}
\dottedcontents{subsection}[3.4em]{}{2.0em}{0.7em}

\newcommand{\crefeqfmt}[1]{
  \crefformat{#1}{(##2##1##3)}
  \Crefformat{#1}{(##2##1##3)}
  \crefrangeformat{#1}{(##3##1##4--##5##2##6)}
  \Crefrangeformat{#1}{(##3##1##4--##5##2##6)}
  \crefmultiformat{#1}{(##2##1##3}{, ##2##1##3)}{, ##2##1##3}{, ##2##1##3)}
  \Crefmultiformat{#1}{(##2##1##3}{, ##2##1##3)}{, ##2##1##3}{, ##2##1##3)}
  \crefrangemultiformat{#1}{(##3##1##4--##5##2##6}{, ##3##1##4--##5##2##6)}{, ##3##1##4--##5##2##6}{, ##3##1##4--##5##2##6)}
  \Crefrangemultiformat{#1}{(##3##1##4--##5##2##6}{, ##3##1##4--##5##2##6)}{, ##3##1##4--##5##2##6}{, ##3##1##4--##5##2##6)}
}
\newcommand{\crefsecfmt}[1]{%
  \crefformat{#1}{\textsection##2##1##3}
  \Crefformat{#1}{\textsection##2##1##3}
  \crefrangeformat{#1}{\textsection\textsection##3##1##4--##5##2##6}
  \Crefrangeformat{#1}{\textsection\textsection##3##1##4--##5##2##6}
  \crefmultiformat{#1}{\textsection\textsection##2##1##3}{--##2##1##3}{, ##2##1##3}{ and~##2##1##3}
  \Crefmultiformat{#1}{\textsection\textsection##2##1##3}{--##2##1##3}{, ##2##1##3}{ and~##2##1##3}
  \crefrangemultiformat{#1}{\textsection\textsection##3##1##4--##5##2##6}{ and~##3##1##4--##5##2##6}{, ##3##1##4--##5##2##6}{ and~##3##1##4--##5##2##6}
  \Crefrangemultiformat{#1}{\textsection\textsectionXS##3##1##4--##5##2##6}{ and~##3##1##4--##5##2##6}{, ##3##1##4--##5##2##6}{ and~##3##1##4--##5##2##6}
}

\crefeqfmt{equation}
\crefeqfmt{enumi}
\crefeqfmt{enumii}
\crefsecfmt{section}
\crefsecfmt{subsection}
\crefsecfmt{appendix}
\crefname{part}{Part}{Parts}
\crefname{chapter}{Chapter}{Chapters}
\crefname{figure}{Figure}{Figures}

\makeatletter

\newcommand{\blocknumfont}{\bfseries}
\newcommand{\blockheadfont}{\bfseries}
\newcommand{\blocknotefont}{\normalfont}
\newcommand{\blockspecialfont}{\itshape}
\newcommand{\blockhorizspace}{0.4em}
\newcommand{\blocknotespace}{0.4em}
\newcommand{\blocknumsep}{}
\newcommand{\blocksep}{.}
\newcommand{\blockvertspace}{\medskipamount}

\newtheoremstyle{block}%
  {\blockvertspace}
  {\blockvertspace}
  {}
  {}
  {} 
  {}
  {0em}
  {\thmname{{\blockheadfont#1}}%
    \@ifnotempty{#1}{\@ifnotempty{#2}{ }}%
    \thmnumber{{\blocknumfont #2\blocknumsep}}%
    \@ifnotempty{#1#2}{{\blockheadfont\blocksep}}%
    \thmnote{\hspace{\blocknotespace}[{\blocknotefont#3}]}%
    \hspace{\blockhorizspace}%
  }

\newtheoremstyle{blockspecial}%
  {\blockvertspace}
  {\blockvertspace}
  {\blockspecialfont}
  {}
  {} 
  {}
  {0em}
  {\thmname{{\blockheadfont#1}}%
    \@ifnotempty{#1}{\@ifnotempty{#2}{ }}%
    \thmnumber{{\blocknumfont #2\blocknumsep}}%
    \@ifnotempty{#1#2}{{\blockheadfont\blocksep}}%
    \thmnote{\hspace{\blocknotespace}[{\blocknotefont#3}]}%
    \hspace{\blockhorizspace}%
  }

\newtheoremstyle{blocknamed}%
  {\blockvertspace}
  {\blockvertspace}
  {}
  {}
  {} 
  {}
  {0em}
  {\@ifempty{#3}{\thmname{\blockheadfont#1}}{\thmname{\blockheadfont#3}}%
    \@ifnotempty{#1#3}{\blockheadfont\blocksep}\hspace{\blockhorizspace}%
  }
  
\theoremstyle{blocknamed}
\newtheorem*{pf}{Proof}

\renewenvironment{proof}[1][]{%
  \pushQED{\qed}%
  \begin{pf}[#1]
}{%
  \popQED\endtrivlist\@endpefalse%
  \end{pf}
}
  
\newcommand{\defblock}[2]{%
  \theoremstyle{block}
  \newtheorem{#1}[block]{#2}%
  \Crefname{#1}{#2}{{#2}s}
  \newtheorem*{#1*}{#2}%
}

\newcommand{\defblockspecial}[2]{%
  \theoremstyle{blockspecial}
  \newtheorem{#1}[block]{#2}%
  \Crefname{#1}{#2}{{#2}s}
  \newtheorem*{#1*}{#2}%
}

\makeatother

\setlist{%
  parsep=0ex, listparindent=\parindent,%
  itemsep=0.75ex, topsep=0.75ex,%
  leftmargin=2.3em,%
}

\setlist[enumerate, 1]{%
  label={\upshape(\arabic*)},%
  ref={\arabic*},%
  widest=9,
}
\setlist[enumerate, 2]{%
  leftmargin=*,%
  label={\upshape(\roman*)},%
  ref=\roman*,%
  widest=iii,
}
\setlist[itemize, 1]{%
  label={\textbf{--}},%
}
\setlist[itemize, 2]{%
  label=--,%
}

\usetikzlibrary{calc,decorations.pathmorphing,shapes,arrows}
\tikzcdset{
  arrow style=tikz,
  diagrams={>={stealth}},
}
  
\newcommand{\from}{\leftarrow}
\newcommand{\fromto}{\rightleftarrows}
\newcommand{\lblto}[1]{\xrightarrow{#1}}

\newcommand{\isoto}{\lblto{\raisebox{-0.3ex}[0pt]{$\scriptstyle \sim$}}}

\newcommand{\iso}{\simeq}
\newcommand{\lbliso}[1]{\overset{#1}{\simeq}}

\defblock{claim}{Claim}
\defblock{construction}{Construction}
\defblockspecial{corollary}{Corollary}
\defblock{definition}{Definition}
\defblock{example}{Example}
\defblock{fact}{Fact}
\defblock{hypothesis}{Hypothesis}
\defblockspecial{lemma}{Lemma}
\defblock{notation}{Notation}
\defblock{nothing}{}
\defblockspecial{proposition}{Proposition}
\defblock{question}{Question}
\defblock{remark}{Remark}
\defblockspecial{specification}{Specification}
\defblockspecial{theorem}{Theorem}
\defblock{variant}{Variant}

\numberwithin{equation}{subsection}
\numberwithin{block}{subsection}
\makeatletter
    \let\c@equation\c@block
\makeatother

\bibliography{ref}

\definecolor{arpon}{rgb}{0.8,0.5,0.8}

\let\amp=&

\newcommand{\cat}[1]{\mathcal{#1}}
\newcommand{\num}[1]{\mathbb{#1}}
\newcommand{\shf}[1]{\mathcal{#1}}

\newcommand{\A}{\num{A}}
\newcommand{\B}{\mathrm{B}}

\newcommand{\C}{\mathrm{C}}
\newcommand{\Conj}{\mathrm{Conj}}

\newcommand{\CAlg}{\mathrm{CAlg}}
\newcommand{\CP}{\num{CP}}
\newcommand{\Cp}{\mathrm{C}_p}
\newcommand{\CycSpt}{\mathrm{CycSpt}}
\newcommand{\E}{\num{E}}
\newcommand{\F}{\num{F}}
\newcommand{\FilSpt}{\mathrm{FilSpt}}
\newcommand{\Fun}{\operatorname{Fun}}

\newcommand{\HC}{\mathrm{HC}}
\newcommand{\HP}{\mathrm{HP}}
\newcommand{\HT}{\mathrm{HT}}
\newcommand{\J}{\mathrm{J}}
\newcommand{\K}{\mathrm{K}}
\newcommand{\KU}{\mathrm{KU}}
\renewcommand{\L}{\mathrm{L}}
\newcommand{\Map}{\operatorname{Map}}
\newcommand{\Mod}{\mathrm{Mod}}
\newcommand{\MU}{\mathrm{MU}}
\newcommand{\N}{\num{N}}
\newcommand{\Nm}{\mathrm{Nm}}
\newcommand{\Nyg}{\mathrm{Nyg}}
\renewcommand{\O}{\shf{O}}

\newcommand{\PrL}{\mathrm{Pr}^{\mathrm{L}}}
\newcommand{\Q}{\num{Q}}

\newcommand{\RGamma}{\mathrm{R}\Gamma}
\renewcommand{\S}{\num{S}}
\newcommand{\Spt}{\mathrm{Spt}}
\newcommand{\Spc}{\mathrm{Spc}}
\newcommand{\Syn}{\mathrm{Syn}}
\newcommand{\TC}{\mathrm{TC}}
\newcommand{\THH}{\mathrm{THH}}
\newcommand{\TP}{\mathrm{TP}}
\newcommand{\Th}{\mathrm{Th}}
\newcommand{\U}{\mathrm{U}}
\newcommand{\V}{\shf{L}}

\newcommand{\X}{\mathrm{X}}
\newcommand{\Y}{\mathrm{Y}}
\newcommand{\Z}{\num{Z}}

\newcommand{\can}{\mathrm{can}}

\newcommand{\cofib}{\operatorname*{cofib}}
\newcommand{\colim}{\operatorname*{colim}}

\newcommand{\cpl}[1]{#1^\wedge}
\newcommand{\cir}{\mathrm{S}^1}
\newcommand{\dual}[1]{#1^\vee}

\newcommand{\fib}{\operatorname*{fib}}
\newcommand{\fil}{\operatorname{fil}}

\newcommand{\gr}{\operatorname{gr}}
\newcommand{\h}{\mathrm{h}}
\newcommand{\id}{\mathrm{id}}
\newcommand{\im}{\mathrm{im}}
\renewcommand{\j}{\mathrm{j}}
\newcommand{\ku}{\mathrm{ku}}
\newcommand{\lau}[1]{(\!\!(#1)\!\!)}
\newcommand{\mot}{\mathrm{mot}}
\renewcommand{\o}[1]{\overline{#1}}
\newcommand{\op}{\mathrm{op}}
\newcommand{\pcpl}[1]{#1^\wedge_p}
\newcommand{\pow}[1]{\lsem #1 \rsem}

\newcommand{\prism}{{\mathbbl{\Delta}}}
\newenvironment{psmallmatrix}{\left(\begin{smallmatrix}}{\end{smallmatrix}\right)}
\newcommand{\sh}{\mathrm{sh}}
\newcommand{\sph}{\mathrm{S}}
\renewcommand{\t}{\mathrm{t}}
\newcommand{\tr}{\mathrm{tr}}
\newcommand{\triv}{\mathrm{triv}}
\renewcommand{\u}{\underline}
\newcommand{\uMap}{\u{\smash{\Map}}}

\title{$\THH(\Z)$ and the image of J}
\author{Sanath K. Devalapurkar and Arpon Raksit}
\date{}

\begin{document}
\maketitle

\begin{abstract}
  Let $p$ be an odd prime number and $\j_p$ the $p$-complete connective image of J spectrum. We establish an equivalence of cyclotomic $\E_\infty$-rings $\pcpl{\THH(\Z)} \iso \sh(\j_p^\triv)$ and an equivalence of $\E_\infty$-rings $\pcpl{\TP(\Z)} \iso \smash{\j_p^{\t\cir}}$. We also record a few applications of this: a new perspective, with some new information, on the description of $\pcpl{\TC(\Z)}$ as a spectrum; height $1$ analogues of the fiber squares of Antieau--Mathew--Morrow--Nikolaus, resulting in new calculations in $\K(1)$-localized algebraic K-theory; and a proof of a slight refinement of the noncommutative crystalline--de Rham comparison result of Petrov--Vologodsky.
\end{abstract}

{\setcounter{tocdepth}{1}
  \small\tableofcontents\bigskip}


\addtocounter{section}{-1}
\label{i}
\section{Introduction}


\subsection{Background and main result}
\label{i--main}

Let $p$ be a prime number. Many of the recent developments concerning topological Hochschild homology and algebraic K-theory rest on a calculation due to B\"okstedt \cite{bokstedt--thhfp} for the finite field $\F_p$, which establishes an isomorphism of graded rings
\[
  \pi_*(\THH(\F_p)) \iso \F_p[u],
\]
where $u$ has degree $2$. In the same work, B\"okstedt also established isomorphisms
\[
  \pi_*(\THH(\Z)) \iso
  \begin{cases}
    \Z & \text{if}\ * = 0 \\
    \Z/n & \text{if}\ * = 2n-1\ (n \in \Z_{\ge 1}) \\
    0 & \text{otherwise};
  \end{cases}
\]
this second calculation has been revisited by Franjou--Pirashvili \cite{franjou-pirashvili--maclane-Z}, Lindenstrauss--Madsen \cite{lindenstrauss-madsen--number-rings}, Krause--Nikolaus \cite{krause-nikolaus--dvrs}, Liu--Wang \cite{liu-wang--TC-local-fields}, and Bhatt--Lurie \cite{bhatt-lurie--APC}.

In \cite{nikolaus-scholze--TC}, Nikolaus--Scholze refined B\"okstedt's calculation of $\pi_*(\THH(\F_p))$ to a description of $\THH(\F_p)$ as a cyclotomic $\E_\infty$-ring. To state this result, let us recall from their work two constructions, other than topological Hochschild homology, that produce cyclotomic spectra (the details of which will be reviewed in the main text below):
\begin{enumerate}
\item A spectrum $X$ determines a ``trivial'' cyclotomic spectrum $X^\triv$, with underlying spectrum $X$.
\item A connective cyclotomic spectrum $Y$ can be ``shifted'' to obtain another cyclotomic spectrum $\sh(Y)$, with underlying spectrum $\tau_{\ge 0}(Y^{\t\Cp})$. Moreover, there is a canonical map of cyclotomic spectra $Y \to \sh(Y)$.
\end{enumerate}
Both of these constructions are compatible with $\E_\infty$-ring structures, so for $A$ a connective $\E_\infty$-ring, $A^\triv$ and $\sh(A^\triv)$ are naturally cyclotomic $\E_\infty$-rings, and there is a natural map of such $A^\triv \to \sh(A^\triv)$. The Nikolaus--Scholze description of $\THH(\F_p)$ can now be stated as follows:

\begin{theorem}[Nikolaus--Scholze, {\citebare[Corollary IV.4.16]{nikolaus-scholze--TC}}]
  \label{i--main--thhFp}
  There is a canonical equivalence of cyclotomic $\E_\infty$-rings
  \[
    \THH(\F_p) \iso \sh(\Z_p^\triv).
  \]
  In addition, the canonical map $\Z_p^\triv \to \sh(\Z_p^\triv)$ becomes an equivalence upon applying $(-)^{\t\Cp}$ and hence upon applying $\smash{(-)^{\t\cir}}$. In particular, there is a canonical equivalence of $\E_\infty$-rings
  \[
    \TP(\F_p) \iso \smash{\Z_p^{\t\cir}}.
  \]
\end{theorem}

One can see in \cref{i--main--thhFp} a rather precise instance of the ``redshift'' and ``blueshift'' phenomena of chromatic homotopy theory, as well as their interaction: $\F_p$ has chromatic height $-1$, while $\Z_p$ has chromatic height $0$. The central result of this paper is a description of $\pcpl{\THH(\Z)}$ for $p$ odd, analogous to \cref{i--main--thhFp}, but taking place at one chromatic height higher.

\begin{notation}
  \label{i--main--j}
  Assume that $p$ is odd. We define $\E_\infty$-rings
  \[
    \smash{\J_p := \L_{\K(1)}\S, \qquad \j_p := \tau_{\ge0}(\J_p)}.
  \]
  Here $\L_{\K(1)}\S$ denotes the $\K(1)$-local sphere spectrum at the prime $p$, which we recall may be described as follows. Let $\KU_p$ denote the $p$-completed complex K-theory spectrum. Recall that $\KU_p$ admits a unique $\E_\infty$-ring structure and that, as an $\E_\infty$-ring, it admits a canonical action of the group of $p$-adic units $\Z_p^\times$. Let $\Gamma \subseteq \Z_p^\times$ be the subgroup generated by the $(p-1)$-st roots of unity and the element $1+p$, so $\Gamma \iso \F_p^\times \times \Z$. Then we have an equivalence of $\E_\infty$-rings
  \[
    \J_p \iso \KU_p^{\h\Gamma}.
  \]
\end{notation}

\begin{remark}
  \label{i--main--j-description}
  The spectra $\J_p$ and $\j_p$ are traditional objects of algebraic topology, arising originally from work of Milnor \cite{milnor--J}, Adams \cite{adams--j-iv}, and Mahowald \cite{mahowald--J-announcement} on the image of the J-homomorphism. The precise relationship is as follows. Letting $\U := \colim_{n \in \N} \U(n)$ be the infinite unitary group and $\S$ the sphere spectrum, the classical J-homomorphism construction of Whitehead \cite{whitehead--J} provides a map of graded abelian groups $J_* : \pi_*(\U) \to \pi_*(\S)$. The unit map $\S \to \j_p$ induces an isomorphism $\pcpl{\im(J_*)} \to \pi_*(\j_p)$ in degrees $* \ge 1$, so that $\j_p$ is an $\E_\infty$-ring whose positive degree homotopy groups encode the $p$-primary image of the J-homomorphism in the stable homotopy groups of spheres (here, it is important that $p$ is odd).
\end{remark}

\begin{theorem}
  \label{i--main--thhZp}
  Assume that $p$ is odd. Then there is a canonical equivalence of cyclotomic $\E_\infty$-rings
  \[
    \pcpl{\THH(\Z)} \iso \sh(\j_p^\triv).
  \]
  In addition, the canonical map $\j_p^\triv \to \sh(\j_p^\triv)$ becomes an equivalence upon applying $(-)^{\t\Cp}$ and hence upon applying $\smash{(-)^{\t\cir}}$. In particular, there is a canonical equivalence of $\E_\infty$-rings
  \[
    \pcpl{\TP(\Z)} \iso \smash{\j_p^{\t\cir}}.
  \]
\end{theorem}

The proof of \cref{i--main--thhZp} is the subject of \cref{t}. The main ingredients are a calculation of the homotopy groups of $\smash{\j_p^{\t\Cp}}$ and a few important, known facts about $\THH(\Z)$, namely the calculation of the homotopy groups of $\smash{\THH(\Z)^{\t\Cp}}$ and forms of the Segal conjecture for $\pcpl{\THH(\Z)}$ and Lichtenbaum--Quillen conjecture for $\pcpl{\TC(\Z)}$. These facts about $\THH(\Z)$ are all originally due to B\"okstedt--Madsen \cite{bokstedt-madsen--TCZ-again,bokstedt-madsen--TCZ}; they have also since been understood anew from other perspectives.

\begin{remark}
  \label{i--main--ku}
  In \cite[Theorem 6.4.1]{skd--thesis}, the first author uses \cref{i--main--thhZp} to provide a similar calculation of the relative topological Hochschild homology $\pcpl{\THH(\Z_p[\zeta_p]/\S\pow{q-1})}$ as $\sh(\ku_p^\triv)$, following an argument suggested by Lurie. That an equivalence of this form holds at the level of $\cir$-equivariant $\E_1$-rings was first communicated to us by Nikolaus, and though this does not play a logical role in this paper, it was part of the basis on which we found \cref{i--main--thhZp}. We note that these ideas are also central to the recent work of Meyer--Wagner \cite{meyer-wagner--q-hodge} on $q$--de Rham/Hodge cohomology.
\end{remark}


\subsection{Application: revisiting $\TC(\Z)$}
\label{i--tcz}

As was just indicated, the cyclotomic structure of $\pcpl{\THH(\Z)}$ was studied thoroughly in the work of B\"okstedt--Madsen \cite{bokstedt-madsen--TCZ-again,bokstedt-madsen--TCZ}. Moreover, the $\K(1)$-local sphere spectrum plays a key role in their analysis. \cref{i--main--thhZp} gives a more structured articulation to some of their ideas.

B\"okstedt--Madsen's study culminates in the following calculation of the spectrum $\pcpl{\TC(\Z)}$.\footnote{Technically, it is only the connective cover of the spectrum that they address, and $\pcpl{\TC(\Z)}$ is $(-1)$-connective but not connective. We are assured however that this was a matter of culture/exposition at the time of their writing. A written account of the relevant analysis in degree $-1$ may be found in work of Rognes \cite[Proof of Proposition 3.3]{rognes--whitehead}.}

\begin{notation}
  \label{i--tcz--ell}
  For $p$ odd, we let $\smash{\ell_p := \tau_{\ge 0}(\KU_p^{\h\F_p^\times})}$ be the connective, $p$-complete Adams summand (here $\F_p^\times$ is acting as the group of $(p-1)$-st roots of unity in $\Z_p^\times$; see \cref{i--main--j}).
\end{notation}

\begin{theorem}[B\"okstedt--Madsen \citebare{bokstedt-madsen--TCZ-again,bokstedt-madsen--TCZ}]
  \label{i--tcz--old}
  Assume that $p$ is odd. Then there exists an equivalence of spectra
  \[
    \pcpl{\TC(\Z)} \iso \j_p \oplus \j_p[1] \oplus \bigoplus_{0 \le k \le p, k \notin \{1,p-1\}} \ell_p[2k-1].
  \]
\end{theorem}

This calculation is relevant to studying the algebraic K-theory of $\S$ via the Dundas--Goodwillie--McCarthy theorem, as done in work of Klein--Rognes \cite{klein-rognes--linearization}, Madsen--Schlichtkrull \cite[Corollary 1.2]{madsen-schlichtkrull--transfer}, Rognes \cite{rognes--whitehead,rognes--KS2}, and Blumberg--Mandell \cite{blumberg-mandell--KS}. More precisely, that involves the map $\pcpl{\TC(\S)} \to \pcpl{\TC(\Z)}$ induced by the unit map $\S \to \Z$, and so it is of interest how the description of \cref{i--tcz--old} interacts with the following description of $\pcpl{\TC(\S)}$.

\begin{theorem}[B\"okstedt--Hsiang--Madsen \citebare{bokstedt-hsiang-madsen--TC}]
  \label{i--tcz--tcs}
  Assume that $p$ is odd. Then there is a canonical equivalence of spectra
  \[
    \pcpl{\TC(\S)} \iso \S_p \oplus \S_p[1] \oplus \pcpl{\Y},
  \]
  where $\Y$ denotes the fiber of the reduced $\cir$-transfer map $\o\tr : \Sigma^{\infty+1}\B\cir \to \S$.\footnote{The canonicity of this equivalence seems to not be fully justified in \cite{bokstedt-hsiang-madsen--TC}, but see the review by Rognes \cite[\textsection 1.12]{rognes--KS2}, as well as the treatment of Nikolaus--Scholze \cite[\textsection IV.3]{nikolaus-scholze--TC}. It will be recalled from the latter point of view in the proof of \cref{z--ol--k1} here.}
\end{theorem}

Understanding the compatibility between the above two results is intertwined with understanding the canonicity of the equivalence in \cref{i--tcz--old}. The proof of B\"okstedt--Madsen specifies the inclusions of the summands $\j_p$ and $\j_p[1]$---they are indeed induced by the maps from the summands $\S_p$ and $\S_p[1]$ of $\pcpl{\TC(\S)}$---but the proof does not make canonical the inclusion and identification of the remainder. A posteriori, though, one can calculate that the retractions of these maps in from $\j_p$ and $\j_p[1]$ are unique up to homotopy, giving a canonical decomposition of $\pcpl{\TC(\Z)}$ into these two summands and a third that is noncanonically equivalent to $\bigoplus_{0 \le k \le p, k \notin \{1,p-1\}} \ell_p[2k-1]$; these calculations appear in the work of Blumberg--Mandell \cite{blumberg-mandell--KS}. Subsequent work of Blumberg--Mandell \cite{blumberg-mandell--eigensplitting} gives a canonical description of the third summand (see \cref{i--tcz--noncanonical} for more on this matter).

With these canonical decompositions in mind, we may ponder the diagonalizability of the map $\pcpl{\TC(\S)} \to \pcpl{\TC(\Z)}$. The remaining question is the behavior of the map on the summand $\pcpl{\Y}$ of the source. As of \cite{blumberg-mandell--KS}, this question has not been completely settled: the compatibility of the projection maps $\pcpl{\TC(\Z)} \to \j_p$ and $\pcpl{\TC(\S)} \to \S_p$ was left open there.

In \cref{z}, we explain a new method for describing the spectrum $\pcpl{\TC(\Z)}$, using \cref{i--main--thhZp} as a starting point. It proves the following refinement of \cref{i--tcz--old}, resolving the above open end from \cite{blumberg-mandell--KS}. Moreover, what we learn along the way helps us make new calculations of the $\K(1)$-local K-theory of certain ring spectra, discussed in \cref{i--k1k}.

\begin{theorem}
  \label{i--tcz--thm}
  Assume that $p$ is odd. Then:
  \begin{enumerate}
  \item \label{i--tcz--thm--description}
    There is a canonical equivalence of spectra
    \[
      \pcpl{\TC(\Z)} \iso \j_p \oplus \j_p[1] \oplus \X,
    \]
    where $\X$ is a spectrum that is noncanonically equivalent to $\bigoplus_{0 \le k \le p, k \notin \{1,p-1\}} \ell_p[2k-1]$.
  \item \label{i--tcz--thm--linearization}
    With respect to the equivalences in \cref{i--tcz--thm--description} and \cref{i--tcz--tcs}, the map $\pcpl{\TC(\S)} \to \pcpl{\TC(\Z)}$ induced by the unit map $\S \to \Z$ is canonically equivalent to the direct sum of the unit map $\S_p \to \j_p$, the suspension thereof $\S_p[1] \to \j_p[1]$, and a certain map $\pcpl{\Y} \to \X$.
  \end{enumerate}
\end{theorem}

\begin{remark}
  \label{i--tcz--not-new}
  As indicated in the preceding discussion, statement \cref{i--tcz--thm--description} of \cref{i--tcz--thm} is not new to this paper. However, our proof is new, and furthermore establishes \cref{i--tcz--thm--linearization}.
\end{remark}

\begin{remark}
  \label{i--tcz--ring}
  \cref{i--main--thhZp} describes $\pcpl{\THH(\Z)}$ as a cyclotomic $\E_\infty$-ring, so it is possible that it could be used to study the $\E_\infty$-ring structure of $\pcpl{\TC(\Z)}$. We do not attempt to do so here.
\end{remark}

\begin{remark}
  \label{i--tcz--noncanonical}
  As mentioned above, in \cite{blumberg-mandell--eigensplitting}, Blumberg--Mandell give a canonical description of the spectrum called $\X$ in \cref{i--tcz--thm}; their description has a Galois-theoretic nature (building on work of Soul\'e \cite{soule--regulators}, Thomason \cite{thomason--etale}, and Dwyer--Mitchell \cite{dwyer-mitchell--K-integers}). The proof of \cref{i--tcz--thm} here gives a different, homotopy-theoretic characterization of $\X$. There is presumably something to be learned by comparing the two perspectives.
\end{remark}

\begin{remark}
  \label{i--tcz--two}
  Rognes \cite{rognes--KZ2,rognes--TCZ2} has analyzed the structure of the spectrum $\cpl{\TC(\Z)}_2$: it is similar to but more complicated than that of $\pcpl{\TC(\Z)}$ for $p$ odd. While a certain analogue of the spectrum $\j_p$ does feature in that result, it seems to us that there is not an analogue of $\j_p$ that is related to $\THH(\Z)^\wedge_2$ precisely as in \cref{i--main--thhZp}. See \cref{i--pv--two} for a concrete warning to this effect, and see \cref{i--pr--two} for a more positive perspective.
\end{remark}


\subsection{A technical but useful refinement of the main result}
\label{i--nil}

Let us return to the following assertion contained in \cref{i--main--thhFp}: there is a map of cyclotomic $\E_\infty$-rings
\[
  \Z_p^\triv \to \sh(\Z_p^\triv) \iso \THH(\F_p)
\]
satisfying the property that it becomes an equivalence upon applying $\smash{(-)^{\t\cir}/p}$. In \cite{ammn--beilinson}, Antieau--Mathew--Morrow--Nikolaus observed that the same property holds for the induced map $\Z_p^\triv \otimes X \to \THH(\F_p) \otimes X$, where $X$ is any $\cir$-equivariant spectrum; more precisely, they showed that the fiber of the canonical map $\Z_p^\triv \to \sh(\Z_p^\triv)$ is ``nilpotent'' as an $\cir$-equivariant spectrum. They then applied this observation to prove certain results concerning rationalized $p$-adic topological cyclic homology and algebraic K-theory (we will come to this application in \cref{i--k1k}).

Now, \cref{i--main--thhZp} contains the parallel assertion that, for $p$ odd, there is a map of cyclotomic $\E_\infty$-rings
\[
  \j_p^\triv \to \sh(\j_p^\triv) \iso \pcpl{\THH(\Z)}
\]
that becomes an equivalence upon applying $\smash{(-)^{\t\cir}/p}$. In the applications to be discussed in \cref{i--k1k,i--pv}, we will use the following refinement of this assertion, analogous to the observation of Antieau--Mathew--Morrow--Nikolaus stated above. This refinement is proved in \cref{j} (where we also review the notion of nilpotence appearing here), by direct calculations with $\j_p$, independent of the theory of topological Hochschild homology.

\begin{notation}
  \label{i--nil--v1}
  Recall that, for $p$ odd, there is a canonical element $v_1 \in \pi_{2p-2}(\S/p)$. We will only really be using the image of this element in $\pi_{2p-2}(\j_p/p)$. It is uniquely characterized by its image in $\pi_{2p-2}(\KU_p/p)$, which is $\smash{\o{\beta}^{p-1}}$, where $\o\beta \in \pi_2(\KU_p/p)$ is the residue of the Bott class $\beta \in \pi_2(\KU_p)$.
\end{notation}

\begin{theorem}
  \label{i--nil--thm}
  Assume that $p$ is odd. Let $\smash{K_{\j_p^\triv}}$ denote the cofiber of the canonical map $\j_p^\triv \to \sh(\j_p^\triv)$. Then the $\cir$-equivariant spectrum $\smash{K_{\j_p^\triv}/(p,v_1)}$ is nilpotent. In particular, for any $\cir$-equivariant spectrum $X$, the canonical map $\j_p^\triv \otimes X \to \sh(\j_p^\triv) \otimes X$ becomes an equivalence upon applying $\smash{(-)^{\t\cir}/(p,v_1)}$.
\end{theorem}


\subsection{Application: $\K(1)$-local $\TC$ and K-theory}
\label{i--k1k}

As alluded to at the beginning of \cref{i--nil}, the fact that there is a map of cyclotomic $\E_\infty$-rings $\Z_p^\triv \to \THH(\F_p)$ that becomes an equivalence upon applying $\smash{(X \otimes -)^{\t\cir}/p}$ for any $\cir$-equivariant spectrum $X$ is an ingredient in the proof of the following result (inspired by earlier work of Bloch--Esnault--Kerz \cite{bek--deformation} and Beilinson \cite{beilinson--fiber}):

\begin{theorem}[Antieau--Mathew--Morrow--Nikolaus \citebare{ammn--beilinson}]
  \label{i--k1k--ammn}
  For $R$ a ring, there is a natural fiber square of spectra
  \[
    \begin{tikzcd}
      \Q \otimes \pcpl{\TC(R)} \ar[r] \ar[d] &
      \Q \otimes \pcpl{\TC(R/pR)} \ar[d] \\
      \Q \otimes \pcpl{\HC^-(R)} \ar[r] &
      \Q \otimes \pcpl{\HP(R)}
    \end{tikzcd}
  \]
  in which the upper horizontal map is induced by the reduction map $R \to R/pR$ and the lower horizontal and left vertical maps are the canonical ones. In particular, there is a natural equivalence between the fiber of the upper horizontal map with \[
    \Q \otimes \pcpl{\HC(R)}[1].
  \]
  If $R$ is commutative and henselian along $(p)$, then the induced square
  \[
    \begin{tikzcd}
      \Q \otimes \pcpl{\K(R)} \ar[r] \ar[d] &
      \Q \otimes \pcpl{\K(R/pR)} \ar[d] \\
      \Q \otimes \pcpl{\HC^-(R)} \ar[r] &
      \Q \otimes \pcpl{\HP(R)}
    \end{tikzcd}
  \]
  is also cartesian; in particular, the fiber of the upper horizontal map here has the same description as above.
\end{theorem}

In \cref{k1k}, we apply arguments from \cite{ammn--beilinson} together with our \cref{i--main--thhZp,i--nil--thm} to prove the following variant of \cref{i--k1k--ammn}, for $\K(1)$-localized topological cyclic homology and algebraic K-theory of connective ring spectra.

\begin{theorem}
  \label{i--k1k--main}
  Assume that $p$ is odd. Then for $R$ a connective $\E_1$-ring, there are natural fiber squares of spectra
  \[
    \begin{tikzcd}
      \L_{\K(1)}\TC(R) \ar[r] \ar[d] &
      \L_{\K(1)}\TC(\pi_0(R)) \ar[d] \\
      \pcpl{\TC^-(\L_{\K(1)}R)} \ar[r] &
      \pcpl{\TP(\L_{\K(1)}R)}
    \end{tikzcd}
  \]
  and
  \[
    \begin{tikzcd}
      \L_{\K(1)}\K(R) \ar[r] \ar[d] &
      \L_{\K(1)}\K(\pi_0(R)[1/p]) \ar[d] \\
      \pcpl{\TC^-(\L_{\K(1)}R)} \ar[r] &
      \pcpl{\TP(\L_{\K(1)}R)}
    \end{tikzcd}
  \]
  in which the upper horizontal maps are induced by the truncation map $R \to \pi_0(R)$ and the localization map $\pi_0(R) \to \pi_0(R)[1/p]$ and the lower horizontal maps are the canonical one. In particular, there are natural equivalences between the fibers of these upper horizontal maps with
  \[
    \pcpl{(\THH(\L_{\K(1)}R)_{\h\cir})}[1].
  \]
\end{theorem}

Note in \cref{i--k1k--main} that the lower rows of the fiber squares (and hence the horizontal fibers) depend only on the $\K(1)$-localization of $R$. In particular, \cref{i--k1k--main} implies that, for $p$ odd, $\K(1)$-local K-theory is a truncating invariant on $\K(1)$-acyclic, connective $\E_1$-rings; this fact was already proved by Land--Mathew--Meier--Tamme in \cite{lmmt--purity} (where it is proved also for $p=2$). We note however that our proof of \cref{i--k1k--main} uses as input the fact that $\K(1)$-local K-theory is a truncating invariant on connective $\E_1$-$\Z$-algebras; for this, we may appeal to \cite{lmmt--purity} or to the work of Bhatt--Clausen--Mathew \cite{bhatt-clausen-mathew--k1k}.

In any case, what is more novel about \cref{i--k1k--main} is that it precisely measures the discrepancy of $\K(1)$-local K-theory from being a truncating invariant on connective $\E_1$-rings in general. In \cref{k1k--eg}, we use the result to calculate the $\K(1)$-local K-theory of several $\E_1$-rings that are not $\K(1)$-acyclic, summarized in the following statement.

\begin{theorem}
  \label{i--k1k--eg}
  Assume that $p$ is odd. Then:
  \begin{enumerate}
  \item \label{i--k1k--eg--j}
    the canonical maps $\L_{\K(1)}\K(\S_p) \to \L_{\K(1)}\K(\j_p) \to \L_{\K(1)}\K(\J_p)$ are equivalences;
  \item \label{i--k1k--eg--ku}
    there are canonical equivalences of spectra
    \begin{align*}
      \L_{\K(1)}\K(\ku_p) &\iso \L_{\K(1)}\K(\Z_p) \oplus \pcpl{\KU_p[\B\cir]}[1], \\
      \L_{\K(1)}\K(\KU_p) &\iso \L_{\K(1)}\K(\ku_p) \oplus \L_{\K(1)}\K(\Z_p)[1].
    \end{align*}
  \end{enumerate}
\end{theorem}

\begin{remark}
  \label{i--k1k--S}
  We have explicit descriptions of all the spectra appearing in \cref{i--k1k--eg}. First, the $\K(1)$-localized cyclotomic trace map $\L_{\K(1)}\K(\Z_p) \to \L_{\K(1)}\TC(\Z_p)$ is an equivalence, and the Dundas--Goodwillie--McCarthy theorem then implies that the same is true when $\Z_p$ is replaced by $\S_p$. Next, \cref{i--tcz--thm} gives an equivalence $\L_{\K(1)}\TC(\Z_p) \iso \J_p \oplus \J_p[1] \oplus \X'$, where $\X'$ is a spectrum noncanonically equivalent to $\KU_p[-1]$ (establishing this equivalence is in fact the penultimate step in the proof of \cref{i--tcz--thm}; see \cref{z--ol--k1}). Finally, our analysis in \cref{z} also affords a description of $\L_{\K(1)}\TC(\S_p)$ as a spectrum: there is a canonical equivalence
  \[
    \L_{\K(1)}\TC(\S_p) \iso \J_p \oplus \J_p[1] \oplus \pcpl{\big(\bigoplus_{n \in \Z_{\ge0}} \KU_p[1]\big)} \oplus \Y',
  \]
  where $\Y'$ is another spectrum noncanonically equivalent to $\KU_p[-1]$; see \cref{z--tr--fiber-description}.
\end{remark}


\subsection{Application: revisiting noncommutative crystalline--de Rham comparison}
\label{i--pv}

Finally, in \cref{c}, we apply \cref{i--main--thhZp,i--nil--thm} to give a new proof of a result of Petrov--Vologodsky \cite{petrov-vologodsky--TP} on the comparison between the classical periodic cyclic homology of a stable $\Z_p$-linear $\infty$-category and the topological periodic cyclic homology of its mod $p$ reduction, again for $p$ odd. This result can be regarded as a noncommutative analogue of the comparison between the de Rham cohomology of a $\Z_p$-scheme and the crystalline/prismatic cohomology of its special fiber.

\begin{theorem}[cf. Petrov--Vologodsky \citebare{petrov-vologodsky--TP}]
  \label{i--pv--thm}
  Assume that $p$ is odd. Then for $\cat{C}$ a dualizable $\Z_p$-linear stable, presentable $\infty$-category, there is a natural equivalence
  \[
    \pcpl{\TP(\cat{C} \otimes_{\Z_p} \F_p)} \iso \pcpl{\HP(\cat{C}/\Z_p)},
  \]
  which is lax symmetric monoidal and which recovers the equivalence $\TP(\F_p) \iso \Z_p^{\t\cir} \iso \HP(\Z_p/\Z_p)$ of \cref{i--main--thhFp} in the case $\cat{C} = \Mod_{\Z_p}$.\footnote{Here we take dualizable stable, presentable $\infty$-categories as our domain for Hochschild/cyclic invariants: recall that Hochschild homology relative to an $\E_\infty$-ring $R$ is given by dimension/Euler characteristic in the symmetric monoidal $\infty$-category $\PrL_R$ of $R$-linear stable, presentable $\infty$-categories (see e.g. \cite[\textsection 4.5]{hss--traces}), and that dualizability in this context is independent of $R$, being equivalent to the categorical property of compact assembly, a weakening of compact generation \cite[Theorem D.7.0.7]{lurie--SAG}.}
\end{theorem}

\begin{remark}
  \label{i--pv--linearity}
  \cref{i--pv--thm} is a slight strengthening of the result of \cite{petrov-vologodsky--TP} due to its last clause, establishing $\smash{\Z_p^{\t\cir}}$-linearity (in particular $\Z_p$-linearity) of the equivalence here.
\end{remark}

\begin{remark}
  \label{i--pv--two}
  The conclusion of \cref{i--pv--thm} is false when $p=2$: as noted in \cite[Remark 2.8]{petrov-vologodsky--TP}, there is no ring homomorphism $\pi_0(\HP(\F_2/\Z_2)) \to \Z_2$ (because the divided powers of $2$ do not converge to $0$ in $\Z_2$), while the multiplication map $\F_2 \otimes_{\Z_2} \F_2 \to \F_2$ induces a ring homomorphism $\pi_0(\TP(\F_2 \otimes_{\Z_2} \F_2)) \to \pi_0(\TP(\F_2)) \iso \Z_2$. Any analogue of \cref{i--main--thhZp,i--nil--thm} for $p=2$ must tend to this issue (again, see \cref{i--pr--two} for indication of one such analogue).
\end{remark}


\subsection{Future work: relation to prismatization}
\label{i--pr}

Let us end this introduction by indicating one more perspective on \cref{i--main--thhZp}, which will be addressed in detail in forthcoming joint work of the authors with Jeremy Hahn and Allen Yuan.

The theory of ``prismatization'', due to Drinfeld \cite{drinfeld--prismatization} and Bhatt--Lurie \cite{bhatt--F-gauges,bhatt-lurie--APC}, associates to $\Z_p$ a collection of stacks
\[
  \Z_p^\Conj, \quad \Z_p^\HT,\quad  \Z_p^\prism,\quad \Z_p^\Nyg,\quad \Z_p^\Syn,
\]
each equipped with a  line bundle denoted $\O\{1\}$. On the one hand, these objects are a part of arithmetic geometry and number theory, e.g. there is a relation between quasicoherent sheaves on $\Z_p^\Syn$ and crystalline representations of the absolute Galois group of $\Q_p$. At the same time, they are closely connected to the cyclotomic $\E_\infty$-ring $\pcpl{\THH(\Z)}$, by virtue of the motivic filtrations introduced by Bhatt--Morrow--Scholze \cite{bms2}: for example, we have filtrations $\fil^*_\mot \pcpl{\THH(\Z)}$ and $\fil^*_\mot \pcpl{\TC(\Z)}$ with associated graded objects given by
\[
  \gr^n_\mot \pcpl{\THH(\Z)} \iso \RGamma(\Z_p^\Conj;\O\{n\})[2n], \quad
  \gr^n_\mot \pcpl{\TC(\Z)} \iso \RGamma(\Z_p^\Syn;\O\{n\})[2n],
\]
where $\O\{n\}$ denotes $\O\{1\}^{\otimes n}$, for $n \in \Z$.

In fact, the latter connection is even stronger than just stated. In the aforementioned forthcoming work with Hahn and Yuan, we give new constructions of the above stacks directly in terms of the cyclotomic $\E_\infty$-ring $\pcpl{\THH(\Z)}$, by elaborating on the ``even filtration'' construction of \cite{hrw--even}. Combining this with \cref{i--main--thhZp} gives, for $p$ odd, a procedure that begins with the $\E_\infty$-ring $\j_p$, produces the cyclotomic $\E_\infty$-ring $\pcpl{\THH(\Z)}$, and ends with the stack $\Z_p^\Syn$. The even filtration may also be applied directly to $\j_p$, and in fact this recovers a certain stack $\smash{\F_1^\Syn}$ that has been introduced by Lurie (in work yet to be published). 

\begin{remark}
  \label{i--pr--two}
  In contrast to our negative comments in \cref{i--tcz--two,i--pv--two} about the case $p=2$, Lurie has noted that there is a sensible definition of the stack $\smash{\F_1^\Syn}$ for $p=2$, and he has also proved a partial analogue of \cref{i--main--thhZp} at the level of stacks, for all $p$. With Hahn and Yuan, we will use this to formulate variants of other results of this paper as well, at the level of stacks/motivic associated graded objects, for all $p$.
\end{remark}


\subsection{Acknowledgements}
\label{i--ack}

We warmly thank Jeremy Hahn, Lars Hesselholt, Jacob Lurie, Akhil Mathew, Thomas Nikolaus, and Allen Yuan for many conversations that contributed to this project. The first stages were based on unpublished work of Nikolaus and by subsequent discussion of those ideas with Yuan; the idea of considering the object $\j_{p,0}$ that appears in \cref{j}, and many other motivating ideas related to this material, were provided by Lurie; and the results in \cref{k1k} were suggested to AR by Hahn, following discussion with Hesselholt and Mathew, while they were independently discovered by SKD. We are also grateful to Ishan Levy, Maxime Ramzi, and Vadim Vologodsky for helpful discussions, to Akhil Mathew and Tomer Schlank for pointing out an error during a talk on this work, and to Lars Hesselholt, John Rognes, Kush Singhal, Ferdinand Wagner, and a referee for corrections on earlier drafts.

SKD's work was supported by NSF DGE-2140743, and AR's by NSF DMS-2103152 and the Bell Systems Fellowship at IAS.


\subsection{Conventions}
\label{i--conv}

\begin{enumerate}
\item \label{i--conv--p}
  For the remainder of the paper, $p$ denotes an odd prime number.
\item \label{i--conv--cats}
  We let $\Spc$ denote the $\infty$-category of spaces, $\Spt$ the $\infty$-category of spectra, and $\CAlg$ the $\infty$-category of $\E_\infty$-rings. For $R$ an $\E_\infty$-ring, we let $\Mod_R$ denote the $\infty$-category of $R$-modules in $\Spt$.
\item \label{i--conv--pcpl-cats}
  We let $\pcpl{\Spt}$ and $\pcpl{\CAlg}$ denote the full subcategories of $\Spt$ and $\CAlg$ spanned by the $p$-complete objects, and we let $\pcpl{(-)}$ denote the $p$-completion functor. We write $\S_p$ for $\pcpl{\S}$.
\item \label{i--conv--K1}
  We let $\Spt_{\K(1)}$ denote the full subcategory of $\Spt$ spanned by the $\K(1)$-local spectra (at the fixed prime $p$), and we let $\L_{\K(1)}$ denote the $\K(1)$-localization functor. As in \cref{i--main--j}, we let $\J_p := \L_{\K(1)}\S$ and $\j_p := \tau_{\ge0}(\J_p)$, and we let $\KU_p$ be the $p$-complete complex K-theory spectrum. We will often use without comment the fact that $\K(1)$-localization is ``$p$-completely smashing'', i.e. that for any spectrum $X$, we have a natural equivalence
  \[
    \pcpl{(\J_p \otimes X)} \isoto \L_{\K(1)}X,
  \]
  or equivalently that for any $\J_p$-module $Y$, we have a natural equivalence $\pcpl{Y} \isoto \L_{\K(1)}Y$.
\item \label{i--conv--chains}
  Let $X$ be a spectrum. We let $X[-] : \Spc \to \Spt$ denote the unique colimit preserving functor sending $* \mapsto X$, so that $X[S] \iso X \otimes \Sigma^\infty_+ S$, and we let $X^{(-)} : \Spc^\op \to \Spt$ denote the unique limit preserving functor sending $* \mapsto X$. (Note that we also use the similar notation $X[n]$ to denote the $n$-fold suspension of $X$.) In addition, for $S$ a space, we set $X\{S\} := \fib(X[S] \to X)$, where the map is induced by the projection map $S \to *$; a choice of basepoint of $S$ induces an equivalence $X\{S\} \iso X \otimes \Sigma^\infty S$ and a splitting $X[S] \iso X \oplus X\{S\}$. In case $X$ is equipped with an $\E_\infty$-ring structure, all of these constructions and identifications lift canonically to $\Mod_X$, by virtue of the forgetful functor $\Mod_X \to \Spt$ preserving colimits and limits.
\item \label{i--conv--cyc}
  In this paper, by an \emph{$\cir$-equivariant spectrum} we mean a functor $\B\cir \to \Spt$, the \emph{underlying spectrum} of which is the value of the functor at the canonical basepoint of $\B\cir$. We often abuse notation by identifying an $\cir$-equivariant spectrum with its underlying spectrum.
  
  By a \emph{cyclotomic spectrum} we mean an $\cir$-equivariant spectrum $X$ together with an $\cir$-equivariant map $\phi : X \to X^{\t\Cp}$. The latter is what is referred to as a \emph{$p$-typical cyclotomic spectrum} in \cite{antieau-nikolaus--cyclotomic}, and in general differs slightly from the notion introduced by Nikolaus--Scholze \cite[Definition II.1.1]{nikolaus-scholze--TC}. However, the cyclotomic spectra of primary interest in this paper are $p$-complete and connective, so there is practically no difference \cite[Remark II.1.3]{nikolaus-scholze--TC}. Notationally, we generally identify a cyclotomic spectrum with its underlying $\cir$-equivariant spectrum; and we refer to the map $\phi$ as the \emph{Frobenius} map. We let $\CycSpt$ denote the $\infty$-category of cyclotomic spectra.
\item \label{i--conv--gamma}
  As in \cref{i--main--j}, $\Gamma \subseteq \Z_p^\times$ will denote the subgroup generated by the $(p-1)$-st roots of unity and the element $1+p$. We furthermore let $\Gamma_0 \subseteq \Gamma$ denote the subgroup generated just by the element $1+p$ (so that $\Gamma_0 \cong \Z$), and we make the standard identification $\Gamma \iso \F_p^\times \times \Gamma_0$.

  We will often use the simple observation that the functor of homotopy $\Gamma$-fixed points commutes with limits and colimits of $p$-complete spectra. Indeed, this is true for homotopy $\Gamma_0$-fixed points (even without $p$-completion), because it is a finite (co)limit. It is also true for homotopy $\F_p^\times$-fixed points, because the assumption of $p$-completeness implies that the norm map induces an equivalence between homotopy $\F_p^\times$-orbits and homotopy $\F_p^\times$-fixed points, and the former preserves colimits while the latter preserves limits.

  For example, for $X$ an $\cir$-equivariant spectrum with $\Gamma$-action, we have a natural equivalence $(X^{\h\Gamma})^{\t\cir} \simeq (X^{\t\cir})^{\h\Gamma}$ (and similarly for $(-)^{\t\cir}$ replaced by $(-)^{\t\Cp}$). This follows from the above, since homotopy $\Gamma$-fixed points commutes with homotopy $\cir$-orbits and homotopy $\cir$-fixed points (and similarly for $\cir$ replaced by $\Cp$).
\end{enumerate}

\section{Proof of the main theorem}
\label{t}

In this section, we will prove \cref{i--main--thhZp}, describing $\pcpl{\THH(\Z)}$ in terms of $\j_p$. We begin with some preliminary material in in \cref{t--cyc} and \cref{t--tCp}: in the former, we review the constructions of cyclotomic spectra that feature in the statement of \cref{i--main--thhZp}, and in the latter, we calculate the homotopy groups of $\j_p^{\t\Cp}$. We then come to the proof in \cref{t--pf}.


\subsection{Constructions of cyclotomic spectra}
\label{t--cyc}

Here we recall the constructions $(-)^\triv$ and $\sh(-)$ in the setting of cyclotomic spectra, following Nikolaus--Scholze.

\begin{construction}[{\citebare[Proposition IV.4.14]{nikolaus-scholze--TC}}]
  \label{t--cyc--triv}
  As $\CycSpt$ is a stable, presentable symmetric monoidal $\infty$-category, there is a unique colimit preserving symmetric monoidal functor $\Spt \to \CycSpt$; we denote it by $(-)^\triv$. Concretely, for a spectrum $X$, the cyclotomic spectrum $X^\triv$ can be described as follows:
  \begin{itemize}
  \item the underlying $\cir$-equivariant spectrum is $X$ with the trivial $\cir$-action;
  \item the Frobenius map $\phi_{X^\triv} : X \to X^{\t\Cp}$ is the composition
    \[
      X \to X^{\B\Cp} \iso X^{\h\Cp} \lblto{\can} X^{\t\Cp},
    \]
    where the first map is induced by the projection to the point, $\B\Cp \to *$, and $\can$ denotes the canonical map from homotopy fixed points to the Tate construction (each of these maps carrying a canonical $\cir$-equivariant structure).
  \end{itemize}
\end{construction}

\begin{remark}
  \label{t--cyc--TC}
  The functor $(-)^\triv : \Spt \to \CycSpt$ defined above admits a right adjoint, namely the functor $\TC : \CycSpt \to \Spt$, given by
  \[
    \TC(X) \iso \uMap_{\CycSpt}(\S^\triv,X) \iso \fib(\phi-\can : X^{\h\cir} \to (X^{\t\Cp})^{\h\cir});
  \]
  here $\phi : X^{\h\cir} \to (X^{\t\Cp})^{\h\cir}$ is obtained by applying $(-)^{\h\cir}$ to the Frobenius map $\phi_X : X \to X^{\t\Cp}$ and $\smash{\can : X^{\h\cir} \to (X^{\t\Cp})^{\h\cir}}$ is given by applying $\smash{(-)^{\h\cir}}$ to the canonical map $\smash{\can : X^{\h\Cp} \to X^{\t\Cp}}$ and precomposing with the canonical equivalence $X^{\h\cir} \iso (X^{\h\Cp})^{\h\cir}$.
\end{remark}

\begin{construction}[{\citebare[Construction IV.4.15]{nikolaus-scholze--TC}}]
  \label{t--cyc--sh}
  Let $X$ be a connective cyclotomic spectrum. Then the cyclotomic spectrum $\sh(X)$ is defined as follows:
  \begin{itemize}
  \item the underlying $\cir$-equivariant spectrum is $\tau_{\ge 0}(X^{\t\Cp})$ with the residual $\cir$-action;
  \item the Frobenius map $\phi_{\sh(X)} : \tau_{\ge 0}(X^{\t\Cp}) \to (\tau_{\ge 0}(X^{\t\Cp}))^{\t\Cp}$ is the composition
    \[
      \tau_{\ge 0}(X^{\t\Cp}) \to X^{\t\Cp} \lblto{(\tau_{\ge 0}(\phi_X))^{\t\Cp}} (\tau_{\ge 0}(X^{\t\Cp}))^{\t\Cp},
    \]
    where the first map the connective cover map and $\tau_{\ge 0}(\phi_X) : X \iso \tau_{\ge 0}(X) \to \tau_{\ge 0}(X^{\t\Cp})$ is the connective cover of the Frobenius map of $X$.
  \end{itemize}
  Furthermore, we define the map of cyclotomic spectra $\phi_X^0 : X \to \sh(X)$ to consist of the $\cir$-equivariant map $\tau_{\ge0}(\phi_X) : X \to \tau_{\ge0}(X^{\t\Cp})$ together with the evidently commutative diagram
  \[
    \begin{tikzcd}
      X \ar[r, "\phi_X"] \ar[d, "\tau_{\ge0}(\phi_X)", swap] &
      X^{\t\Cp} \ar[d, "(\tau_{\ge0}(\phi_X))^{\t\Cp}"] \\
      \tau_{\ge0}(X^{\t\Cp}) \ar[r, "\phi_{\sh(X)}"] &
      (\tau_{\ge0}(X^{\t\Cp}))^{\t\Cp}.
    \end{tikzcd}
  \]
\end{construction}

\begin{remark}
  \label{t--cyc--sh-monoidality}
  Let $\CycSpt_{\ge 0} \subseteq \CycSpt$ denote the full subcategory spanned by connective cyclotomic spectra. Using that the Tate construction $(-)^{\t\Cp}$ and connective cover $\tau_{\ge 0}(-)$ define lax symmetric monoidal endofunctors on the $\infty$-category of $\cir$-equivariant spectra, we see that \cref{t--cyc--sh} defines a lax symmetric monoidal functor $\sh : \CycSpt_{\ge 0} \to \CycSpt_{\ge 0}$ and a lax symmetric monoidal transformation $\phi^0 : \id \to \sh$.
\end{remark}


\subsection{The homotopy groups of $\j_p^{\t\Cp}$}
\label{t--tCp}

\begin{notation}
  \label{t--tCp--ku}
  Let $\ku_p := \tau_{\ge 0}(\KU_p)$. We let $\beta \in \pi_2(\ku_p)$ denote the Bott element, so that Bott periodicity gives us a graded ring isomorphism $\pi_*(\ku_p) \iso \Z_p[\beta]$. We let $t \in \smash{\pi_{-2}(\ku_p^{\B\cir})}$ denote the standard complex orientation, and we set $q := \beta t + 1 \in \smash{\pi_0(\ku_p^{\B\cir})}$. Recall that the map $\smash{\ku_p^{\B\cir} \iso \ku_p^{\h\cir} \to \ku_p^{\t\cir}}$ is a localization at the class $t$; we set $\smash{u := t^{-1} \in \pi_2(\ku_p^{\t\cir})}$.
\end{notation}

\begin{notation}
  \label{t--tCp--ell}
  Recall that $\ell_p \iso \smash{\ku_p^{\h\F_p^\times}}$ is the connective, $p$-complete Adams summand. We choose a complex orientation $\smash{t_0 \in \pi_{-2}(\ell_p^{\B\cir})}$ and set $\smash{u_0 := t_0^{-1} \in \pi_2(\ell_p^{\t\cir})}$.
\end{notation}

\begin{notation}
  \label{t--tCp--pow}
  For any symbol $x$, we let $\Z_p\pow{x-1}$ denote the $(p,x-1)$-completion of the polynomial ring $\Z[x]$.
\end{notation}

\begin{notation}
  \label{t--tCp--zetap}
  We use the notation
  \[
    [p]_q := \frac{q^p-1}{q-1} = 1 + q + \cdots + q^{p-1} \in \Z[q]
  \]
  and set $\Z_p[\zeta_p] := \Z_p[q]/([p]_q) \iso \Z_p\pow{q-1}/([p]_q)$, with $\zeta_p$ denoting the residue of $q$ in the quotient, a primitive $p$-th root of unity.
\end{notation}

\begin{proposition}
  \label{t--tCp--oriented}
  We have the following identifications of graded rings:
  \begin{align*}
    \pi_*(\ku_p^{\h\cir}) \iso \Z_p\pow{q-1}[\beta,t]/(\beta t - (q-1))
    \qquad \pi_*(\ku_p^{\t\cir}) \iso \Z_p\pow{q-1}[u^{\pm 1}] \\
    \pi_*(\ku_p^{\t\Cp}) \iso \Z_p[\zeta_p][u^{\pm 1}] 
    \qquad \pi_*(\ell_p^{\t\Cp}) \iso \Z_p[u_0^{\pm 1}]
    \qquad \pi_*(\Z_p^{\t\Cp}) \iso \F_p[u^{\pm 1}].
  \end{align*}
  Moreover, the map
  \[
    \pi_*(\phi_{\ku_p^\triv}) : \pi_*(\ku_p) \to \pi_*(\ku_p^{\t\Cp})
  \]
  sends $\beta \mapsto (\zeta_p-1)u$.
\end{proposition}

\begin{proof}
  The identifications follow from Bott periodicity, the identification of the Quillen formal group law associated to $\ku_p$ (with its complex orientation $t$) with the law $x + y + \beta xy$, and standard facts about complex oriented ring spectra. The claim about the map $\phi_{\ku_p^\triv} : \ku_p \to \ku^{\t\Cp}$ follows from its factorization as the composition
  \[
    \ku_p \to \ku_p^{\B\cir} \iso \ku_p^{\h\cir} \to \ku_p^{\t\cir} \to \ku_p^{\t\Cp},
  \]
  (each map being the canonical one), the relation $\beta t = q-1$ in $\smash{\pi_0(\ku_p^{\h\cir})}$, the definition $u = t^{-1}$ in $\smash{\pi_2(\ku_p^{\t\cir})}$, and the fact that the last map sends $q \mapsto \zeta_p$.
\end{proof}

\begin{notation}
  \label{t--tCp--psi-ntn}
  We use the standard notation $\psi^g$ for the action of $g \in \Z_p^\times$ on $\ku_p$ and on $\ell_p$, and we will use the formulas $\smash{\psi^g(q) = q^g \in \pi_0(\ku_p^{\h\cir})}$ and $\smash{\psi^g(\beta) = g\beta \in \pi_2(\ku_p)}$.
\end{notation}

\begin{lemma}
  \label{t--tCp--psi}
  For each $n \in \Z$, the actions of $\psi^{1+p}$ on $\pi_{2n}(\smash{\ku_p^{\t\Cp}})$ and $\pi_{2n}(\smash{\ell_p^{\t\Cp}})$ are both given by multiplication by $(1+p)^n$.
\end{lemma}

\begin{proof}
  The map $\smash{\ell_p^{\t\Cp} \to \ku_p^{\t\Cp}}$ is equivariant with respect to $\psi^{1+p}$ and injective on homotopy groups, so it suffices to show the statement for $\smash{\pi_{2n}(\ku_p^{\t\Cp})}$. Furthermore, by \cref{t--tCp--oriented} and multiplicativity of $\psi^{1+p}$, it suffices to show that $\psi^{1+p}(\zeta_p) = \zeta_p$ and $\psi^{1+p}(u) = (1+p)u$. The former follows from the fact that $\psi^{1+p}(q) = q^{1+p}$ in $\smash{\pi_0(\ku_p^{\h\cir})}$. For the latter, we use that the map $\smash{\phi_{\ku_p^\triv} : \ku_p \to \ku_p^{\t\Cp}}$ is equivariant with respect to $\psi^{1+p}$. Let $\beta' \in \smash{\pi_2(\ku_p^{\t\Cp})}$ denote the image of $\beta \in \pi_2(\ku_p)$ under this map. Since $\psi^{1+p}(\beta) = (1+p)\beta$, we have that $\psi^{1+p}(\beta') = (1+p)\beta'$, and from \cref{t--tCp--oriented} we have the relation $\beta' = (\zeta_p-1)u$. Combining what we have said so far, we find
  \[
    (\zeta_p-1)\psi^{1+p}(u) = \psi^{1+p}((\zeta_p-1)u) = \psi^{1+p}(\beta') = (1+p)\beta' = (1+p)(\zeta_p-1)u,
  \]
  which implies $\psi^{1+p}(u) = (1+p)u$.
\end{proof}

For the next statement we recall that, since $p$ is odd, $\S/p$ admits an $\A_2$-ring structure, i.e. a unital multiplication. Hence, an $\A_2$-ring structure (in particular an $\E_\infty$-ring structure) on a spectrum $R$ induces an $\A_2$-ring structure on $R/p$, further inducing a unital multiplication on the mod $p$ homotopy groups $\pi_*(R/p)$.

\begin{proposition}
  \label{t--tCp--j}
  \begin{enumerate}[leftmargin=*]
  \item \label{t--tCp--j--integral}
    There are isomorphisms of abelian groups
    \[
      \pi_*(\j_p^{\t\Cp}) \iso
      \begin{cases}
        \Z_p & \text{if}\ * = 0 \\
        \Z_p/n\Z_p & \text{if}\ * = 2n-1\ (n \in \Z) \\
        0 & \text{otherwise}
      \end{cases}
    \]
    (to be clear, for $n=0$ we mean $\pi_{-1}(\j_p^{\t\Cp}) \iso \Z_p$).
  \item \label{t--tCp--j--modp}
    There are isomorphisms of abelian groups
    \[
      \pi_*(\j_p^{\t\Cp}/p)
      \iso
      \begin{cases}
        \F_p & \text{if}\ * = 2pm\ \text{or}\ * = 2pm-1\ (m \in \Z) \\
        0 & \text{otherwise},
      \end{cases}
    \]
    and the map $\pi_*(\j_p^{\t\Cp}/p) \to \pi_*(\ell_p^{\t\Cp}/p)$ is an isomorphism in degrees divisible by $2p$.
  \item \label{t--tCp--j--modp-ring}
    There is an isomorphism of graded rings
    \[
      \pi_*(\j_p^{\t\Cp}/p) \iso \F_p[u_1^{\pm 1},w]/(w^2),
    \]
    where $u_1$ has degree $2p$ and $w$ has degree $-1$.
  \end{enumerate}
\end{proposition}

\begin{proof}
  By definition of $\j_p$ and the triviality of the action of $\Z_p^\times$ on $\pi_0(\ku_p) \iso \pi_0(\ell_p) \iso \Z_p$, we have a cartesian square of $\E_\infty$-rings
  \[
    \begin{tikzcd}
      \j_p \ar[r] \ar[d] &
      \ell_p^{\h\Gamma_0} \ar[d] \\
      \Z_p \ar[r] &
      \Z_p^{\h\Gamma_0}.
    \end{tikzcd}
  \]
  Applying $(-)^{\t\Cp}$ to this diagram and commuting $(-)^{\t\Cp}$ with the finite limit $(-)^{\h\Gamma_0}$, we obtain a cartesian square
  \[
    \begin{tikzcd}
      \j_p^{\t\Cp} \ar[r] \ar[d] &
      (\ell_p^{\t\Cp})^{\h\Gamma_0} \ar[d] \\
      \Z_p^{\t\Cp} \ar[r] &
      (\Z_p^{\t\Cp})^{\h\Gamma_0}.
    \end{tikzcd}
  \]
  Passing to the associated long exact sequence of homotopy groups and using \cref{t--tCp--oriented} and \cref{t--tCp--psi} to compute the homotopy groups of $(\ell_p^{\t\Cp})^{\h\Gamma_0}$, $\Z_p^{\t\Cp}$, and $(\Z_p^{\t\Cp})^{\h\Gamma_0}$, as well as the maps among them, we obtain the following:
  \begin{itemize}[leftmargin=*]
  \item For nonzero $n \in \Z$, we have an exact sequence
    \[
      \Z_p/((1+p)^{n+1}-1)\Z_p \to \F_p \to \pi_{2n}(\j_p^{\t\Cp}) \to \F_p \to \F_p \to \pi_{2n-1}(\j_p^{\t\Cp}) \to \Z_p/((1+p)^n-1)\Z_p \to \F_p
    \]
    in which the first and last maps are the reduction maps and the map $\F_p \to \F_p$ is the identity. Together with the fact that $(1+p)^n-1$ and $pn$ have equal $p$-adic valuation (here we use that $p$ is odd), this gives the desired calculation of $\pi_*(\j_p^{\t\Cp})$ for degrees $*$ other than $0$ and $-1$.
  \item We have an exact sequence
    \[
      \Z_p/((1+p)-1)\Z_p \to \F_p \to \pi_0(\j_p^{\t\Cp}) \to \Z_p \oplus \F_p \to \F_p \to \pi_{-1}(\j_p^{\t\Cp}) \to \Z_p \to \F_p
    \]
    in which the first map is the obvious isomorphism, the last map is the reduction map, and the map $\Z_p \oplus \F_p \to \F_p$ is the sum of the reduction map and the identity map. This gives the desired calculation of $\pi_*(\j_p^{\t\Cp})$ for $* \in \{0,-1\}$.
  \end{itemize}
  This analysis proves \cref{t--tCp--j--integral} and \cref{t--tCp--j--modp}. For \cref{t--tCp--j--modp-ring}, we note that \cref{t--tCp--j--modp}, together with \cref{t--tCp--oriented}, gives us an isomorphism of graded rings $\smash{\pi_{2p*}(\j_p^{\t\Cp}/p) \iso \F_p[u_1^{\pm 1}]}$ (where $u_1$ maps to $\smash{u_0^p \in \pi_{2p}(\ell_p^{\t\Cp}/p)}$). Multiplication by $u_1$ must then give isomorphisms $\pi_{2pm-1}(\j_p^{\t\Cp}/p) \iso \pi_{2p(m+1)-1}(\j_p^{\t\Cp}/p)$, so it remains only to observe that a generator $w$ in degree $-1$ must satisfy $w^2=0$, as the group in degree $-2$ is zero.
\end{proof}


\subsection{The proof}
\label{t--pf}

\begin{lemma}
  \label{t--pf--LQ-uniqueness}
  Let $0 \le n \le \infty$ and let $A$ be a $p$-complete $\E_n$-ring such that the canonical map $\pi_m(A) \to \pi_m(\L_{\K(1)}A)$ is an isomorphism for $m \ge 2$. Then there is a unique map of $\E_n$-rings $\j_p \to A$; in other words, the space of $\E_n$-ring maps $\j_p \to A$ is contractible.
\end{lemma}

\begin{proof}
  Since $\j_p$ is connective, a map of $\E_n$-rings $\j_p \to A$ is equivalent to a map of $\E_n$-rings $\j_p \to \tau_{\ge 0}(A)$. The hypothesis on $A$ implies that the commutative square of $\E_n$-rings
  \[
    \begin{tikzcd}
      \tau_{\ge0}(A) \ar[r] \ar[d] &
      \tau_{\ge 0}(\L_{\K(1)} A) \ar[d] \\
      \tau_{\le 1}(\tau_{\ge0}(A)) \ar[r] &
      \tau_{\le 1}(\tau_{\ge 0}(\L_{\K(1)} A))
    \end{tikzcd}
  \]
  is cartesian. It thus suffices to show that there is a unique map of $\E_n$-rings from $\j_p$ to each of the other three objects in the square (where, again, ``unique'' means that the space of $\E_n$-ring maps from $\j_p$ to these objects in the square) is contractible:
  \begin{itemize}
  \item A map $\j_p \to \tau_{\ge 0}(\L_{\K(1)} A)$ is equivalent to a map $\j_p \to \L_{\K(1)}A$. To see that such a map is unique, it suffices to see that the unit map $\S \to \j_p$ is a $\K(1)$-local equivalence. Since the unit map $\S \to \J_p$ is by definition a $\K(1)$-local equivalence, it furthermore suffices to see that that the connective cover map $\j_p \to \J_p$ is a $\K(1)$-local equivalence; this is true because its fiber is coconnective and hence $\K(1)$-acyclic.
  \item Similarly, for any $1$-truncated, connective, $p$-complete $\E_n$-ring $B$, there is a unique map $\j_p \to B$ because the unit map $\S_p \to \j_p$ induces an equivalence on $\tau_{\le1}(-)$. \qedhere
  \end{itemize}
\end{proof}

\begin{proposition}
  \label{t--pf--map}
  There is a unique map of cyclotomic $\E_\infty$-rings $\alpha : \j_p^\triv \to \pcpl{\THH(\Z)}$.
\end{proposition}

\begin{proof}
  Maps of cyclotomic $\E_\infty$-rings $\j_p^\triv \to \pcpl{\THH(\Z)}$ are equivalent to maps of $\E_\infty$-rings $\j_p \to \pcpl{\TC(\Z)}$, so it suffices to show that there is a unique map of the latter sort. This follows from \cref{t--pf--LQ-uniqueness}, as the canonical map $\pi_m(\pcpl{\TC(\Z)}) \to \pi_m(\L_{\K(1)}\pcpl{\TC(\Z)})$ is an isomorphism for $m \ge 2$---this is a $\TC$-theoretic Lichtenbaum--Quillen statement, due originally to B\"okstedt--Madsen \cite{bokstedt-madsen--TCZ} and revisited by Hesselholt--Madsen \cite{hesselholt-madsen--local-fields,hesselholt-madsen--dRW}, Mathew \cite[\textsection 4]{mathew--K1TR}, and Liu--Wang \cite[Remark 1.7]{liu-wang--TC-local-fields}; it may also be deduced from the work of Bhatt--Morrow--Scholze \cite{bms2} and Bhatt--Lurie \cite{bhatt--F-gauges,bhatt-lurie--APC} on syntomic cohomology (cf. \cite[Remark 2.21]{bhatt--icm}).
\end{proof}

\begin{remark}
  \label{t--pf--map-comparison}
  \cref{t--pf--map} can be compared to the fact that there is a unique map of cyclotomic $\E_\infty$-rings $\o\alpha : \Z_p^\triv \to \THH(\F_p)$, which follows from the computation of $\TC(\F_p)$, as discussed in \cite[\textsection IV.4]{nikolaus-scholze--TC}. In fact, there is precise comparison to make: the diagram of cyclotomic $\E_\infty$-rings
  \[
    \begin{tikzcd}
      \j_p^\triv \ar[r, "\alpha"] \ar[d] &
      \pcpl{\THH(\Z)} \ar[d] \\
      \Z_p^\triv \ar[r, "\o\alpha"] &
      \THH(\F_p)
    \end{tikzcd}
  \]
  is commutative; here the left vertical map is induced by the truncation map $\j_p \to \Z_p$ and the right vertical map is induced by the reduction map $\Z_p \to \F_p$, and the commutativity is again immediate from the computation of $\TC(\F_p)$.
\end{remark}

\begin{proof}[Proof of \cref{i--main--thhZp}]
  Let $\alpha : \j_p^\triv \to \pcpl{\THH(\Z)}$ be as in \cref{t--pf--map}, and consider the resulting commutative diagram of cyclotomic $\E_\infty$-rings
  \[
    \begin{tikzcd}
      \j_p^\triv \ar[r, "\alpha"] \ar[d, "\phi^0", swap] &
      \pcpl{\THH(\Z)} \ar[d, "\phi^0"] \\
      \sh(\j_p^\triv) \ar[r, "\sh(\alpha)"] &
      \sh(\pcpl{\THH(\Z)}).
    \end{tikzcd}
  \]
  The right vertical map is an equivalence---this is the Segal conjecture for $\pcpl{\THH(\Z)}$, proved first by B\"okstedt--Madsen \cite[Lemma 6.5]{bokstedt-madsen--TCZ} and treated again by Mathew \cite[\textsection 5]{mathew--K1TR} and Hahn--Wilson \cite[\textsection 4]{hahn-wilson--redshift}; it follows also from the work of Bhatt--Morrow--Scholze \cite{bms2} and Bhatt--Lurie \cite{bhatt-lurie--APC} on motivic filtrations and prismatic cohomology, cf. \cite[Construction 2.4]{bhatt-mathew--syntomic}. To finish the proof, it will suffice to show that the map $\smash{\alpha^{\t\Cp} : \j_p^{\t\Cp} \to \THH(\Z)^{\t\Cp}}$ is an equivalence (which also implies that the bottom horizontal map $\sh(\alpha)$ is an equivalence, by definition of $\sh(-)$).

  By $p$-completeness, it suffices to prove that the map of spectra $\smash{\alpha^{\t\Cp}/p : \j_p^{\t\Cp}/p \to \THH(\Z)^{\t\Cp}/p}$ is an equivalence. We will show that it induces isomorphisms on homotopy groups. By \cref{t--tCp--j}, we have an isomorphism of graded rings
  \[
    \pi_*(\j_p^{\t\Cp}/p) \iso \F_p[u_1^{\pm 1},w]/(w^2),
  \]
  and moreover the $p$-Bockstein map $\smash{\pi_{2p}(\j_p^{\t\Cp}/p) \to \pi_{2p-1}(\j_p^{\t\Cp}/p)}$ is an isomorphism. The same description applies to $\pi_*(\THH(\Z)^{\t\Cp}/p)$, by the work of B\"okstedt \cite{bokstedt--thhfp} and B\"okstedt--Madsen \cite{bokstedt-madsen--TCZ} (again, in place of the latter, one may alternatively appeal to the work of Bhatt--Morrow--Scholze \cite{bms2} and Bhatt--Lurie \cite{bhatt-lurie--APC}). It is immediate from the ring structures that $\pi_*(\alpha^{\t\Cp}/p)$ is an isomorphism in degrees divisible by $2p$. The $p$-Bockstein relation then implies it is also an isomorphism in degree $2p-1$, and hence in all other degrees.
\end{proof}

\section{$\TC(\Z)$}
\label{z}

In this section, we use \cref{i--main--thhZp} to analyze $\pcpl{\TC(\Z)}$, our goal being to prove \cref{i--tcz--thm}. The full analysis requires some digression into background material, so we present the main steps of the argument in \cref{z--ol} while deferring a couple of proofs to the subsequent subsections.


\subsection{The main argument}
\label{z--ol}

The bulk of our analysis will concern the $\K(1)$-localization of $\TC(\Z)$. We begin by introducing a few key constructions.

\begin{notation}
  \label{z--ol--p-cover}
  For any $\E_\infty$-ring $R$, we let $\smash{[p]^* : R^{\B\cir} \to R^{\B\cir}}$ and $\smash{[p]_* : R[\B\cir] \to R[\B\cir]}$ denote the maps induced by the $p$-fold covering map $[p] : \cir \to \cir$. For $R \in \{\j_p,\J_p,\ku_p,\KU_p\}$, we implicitly regard $R$ as equipped with trivial $\cir$-action, and we freely make the canonical identifications $R_{\h\cir} \iso R[\B\cir]$ and $\smash{R^{\h\cir} \iso R^{\B\cir}}$.
\end{notation}

\begin{remark}
  \label{z--ol--p-cover-graded}
  Let us note something about the map $[p]_* : R[\B\cir] \to R[\B\cir]$ that will be used a couple of times later in this subsection: it is the filtered colimit of maps $[p]_{*,n} : R[\CP^n] \to R[\CP^n]$ for $n \in \Z_{\ge 0}$, where we have $[p]_{*,0} \iso \id$ and commutative diagrams
  \[
    \begin{tikzcd}
      R[\CP^n] \ar[r] \ar[d, "{[p]_{*,n}}"] &
      R[\CP^{n+1}] \ar[r] \ar[d, "{[p]_{*,n+1}}"] &
      R[2n+2] \ar[d, "p^{n+1}"] \\
      R[\CP^n] \ar[r] &
      R[\CP^{n+1}] \ar[r] &
      R[2n+2]
    \end{tikzcd}
  \]
  in which the rows are cofiber sequences and the rightmost vertical map denotes multiplication by $p^{n+1}$. See \cite[Lemma IV.3.5]{nikolaus-scholze--TC} for similar discussion.
\end{remark}

\begin{lemma}
  \label{z--ol--KU-phi}
  There is a unique $\Gamma$-equivariant map of $\E_\infty$-$\KU_p$-algebras $\smash{\Phi : \pcpl{(\KU_p^{\t\cir})} \to \pcpl{(\KU_p^{\t\cir})}}$ making the following diagram of $\Gamma$-equivariant $\E_\infty$-$\KU_p$-algebras commute:
  \[
    \begin{tikzcd}
      \KU_p^{\h\cir} \ar[r, "{[p]^*}"] \ar[d, "\can", swap] &
      \KU_p^{\h\cir} \ar[d, "\can"] \\
      \pcpl{(\KU_p^{\t\cir})} \ar[r, "\Phi"] &
      \pcpl{(\KU_p^{\t\cir})}.
    \end{tikzcd}
  \]
\end{lemma}

\begin{proof}
  The map $\smash{\can : \KU_p^{\h\cir} \to \KU_p^{\t\cir}}$ is a localization of $\E_\infty$-rings, at the element $t \in \pi_{-2}(\KU_p^{\h\cir})$, or equivalently at the element $\smash{q-1 = \beta t\in \pi_0(\KU_p^{\h\cir})}$ (for these elements see \cref{t--tCp--ku}). Thus, what we need to check is that $\smash{[p]^* : \KU_p^{\h\cir} \to \KU_p^{\h\cir}}$ sends $q-1$ to a unit. It sends $q-1 \mapsto q^p-1$, and the congruence $q^p-1 \equiv (q-1)^p$ modulo $p$ implies that $q^p-1$ is invertible in the $p$-complete ring $\smash{\pi_0(\pcpl{(\KU_p^{\t\cir})}) \iso \pcpl{\Z_p\lau{q-1}}}$.
\end{proof}

\begin{construction}
  \label{z--ol--phi}
  Let $\smash{\Phi : \pcpl{(\KU_p^{\t\cir})} \to \pcpl{(\KU_p^{\t\cir})}}$ be as in \cref{z--ol--KU-phi}. Taking fixed points for the action of $\Gamma$, we obtain a map of $\E_\infty$-rings $\varphi: \pcpl{(\J_p^{\t\cir})} \to \pcpl{(\J_p^{\t\cir})}$ making the following diagram of $\E_\infty$-rings commute:
  \begin{equation}
    \label{z--ol--phi--compatibility}
    \begin{tikzcd}
      \J_p^{\h\cir} \ar[r, "{[p]^*}"] \ar[d, "\can", swap] &
      \J_p^{\h\cir} \ar[d, "\can"] \\
      \pcpl{(\J_p^{\t\cir})} \ar[r, "\varphi"] &
      \pcpl{(\J_p^{\t\cir})}.
    \end{tikzcd}
  \end{equation}
\end{construction}

We now explain how the preceding constructions are relevant to understanding $\L_{\K(1)}\TC(\Z)$.

\begin{theorem}
  \label{z--ol--phi-TC}
  Let $\varphi: \pcpl{(\J_p^{\t\cir})} \to \pcpl{(\J_p^{\t\cir})}$ be as in \cref{z--ol--phi}. Then there are canonical equalizer diagrams of $\E_\infty$-rings
  \[
    \begin{tikzcd}
      \L_{\K(1)}\TC(\Z) \ar[r] &
      \pcpl{(\J_p^{\t\cir})} \ar[r, "\id", shift left] \ar[r, "\varphi"', shift right] & \pcpl{(\J_p^{\t\cir})},
    \end{tikzcd}
    \quad
    \begin{tikzcd}
    \L_{\K(1)}\TC(\S) \ar[r] &
      \J_p^{\h\cir} \ar[r, "\can", shift left] \ar[r, "{\can \circ [p]^*}"', shift right] & \pcpl{(\J_p^{\t\cir})},
    \end{tikzcd}
  \]
  under which the map $\L_{\K(1)}\TC(\S) \to \L_{\K(1)}\TC(\Z)$ induced by the unit map $\S \to \Z$ canonically identifies with the map induced by $\can : \J_p^{\h\cir} \to \pcpl{(\J_p^{\t\cir})}$ and the commutativity of \cref{z--ol--phi--compatibility}.
\end{theorem}

\begin{remark}
  \label{z--ol--hesselholt-madsen}
  The spectrum $\pcpl{(\J_p^{\t\cir})}$ was studied by Hesselholt--Madsen \cite{hesselholt-madsen--J}. The introduction of that paper states that it was motivated partially by the connection between this object and $\TC(\Z)$ suggested by the work of B\"okstedt--Madsen \cite{bokstedt-madsen--TCZ}, though no precise connection was known to them. \cref{z--ol--phi-TC} now gives a precise connection.
\end{remark}

Our proof of \cref{z--ol--phi-TC} will repeatedly use the following fact.

\begin{lemma}
  \label{z--ol--coconnective}
  Let $\psi : X \to Y$ be a map of $\cir$-equivariant spectra with coconnective fiber. Then $\psi$ as well as the induced maps
  \[
    \psi_{\h\cir} : X_{\h\cir} \to Y_{\h\cir}, \quad
    \psi^{\h\cir} : X^{\h\cir} \to Y^{\h\cir}, \quad
    \psi^{\t\cir} : X^{\t\cir} \to Y^{\t\cir}
  \]
  are $\K(1)$-local equivalences.
\end{lemma}

\begin{proof}
  For $\psi$ and $\smash{\psi^{\h\cir}}$ this results from the fact that coconnective spectra are $\K(1)$-acyclic. As $\K(1)$-local equivalences are stable under colimits, we then deduce the claim for $\psi_{\h\cir}$ and $\smash{\psi^{\t\cir}}$.
\end{proof}

\begin{proof}[Proof of \cref{z--ol--phi-TC}]
  We have equivalences of cyclotomic $\E_\infty$-rings $\THH(\S) \iso \S^\triv$ and $\pcpl{\THH(\Z)} \iso \sh(\j_p^\triv)$, the former being inverse to the unit map and the latter that of \cref{i--main--thhZp}. The proof will proceed by analyzing the commutative diagram of cyclotomic $\E_\infty$-rings
  \[
    \begin{tikzcd}
      \S^\triv \ar[r] &
      \j_p^\triv \ar[r] \ar[d] &
      \sh(\j_p^\triv) \ar[d] \\
      &
      \ku_p^\triv \ar[r] &
      \sh(\ku_p^\triv),
    \end{tikzcd}
  \]
  the upper left map being the unit map, the other horizontal maps being the maps $\phi^0$ of \cref{t--cyc--sh}, and the vertical maps being induced by the canonical map $\j_p \to \ku_p$.

  Note first that the map $\TC(\S^\triv) \to \TC(\j_p^{\triv})$ is a $\K(1)$-local equivalence. Indeed, by \cite[Remark 2.4]{ammn--beilinson}, it identifies after $p$-completion with the map $\TC(\S) \to \j_p \otimes \TC(\S)$ induced by the unit map $\S_p \to \j_p$, so it suffices to see that this unit map is a $\K(1)$-local equivalence. As explained in the proof of \cref{t--pf--LQ-uniqueness}, this follows from \cref{z--ol--coconnective}.

  Next, the map $\j_p^\triv \to \sh(\j_p^\triv)$ induces commutative diagrams of $\E_\infty$-rings
  \begin{equation}
      \label{z--ol--phi-TC--two}
      \begin{tikzcd}
      \j_p^{\h\cir} \ar[r] \ar[d, "\can", swap] &
      (\tau_{\ge0}(\j_p^{\t\Cp}))^{\h\cir} \ar[d, "\can"]  \\
      \j_p^{\t\cir} \ar[r] &
      (\tau_{\ge0}(\j_p^{\t\Cp}))^{\t\cir},
    \end{tikzcd}
    \qquad\qquad
    \begin{tikzcd}
      \j_p^{\h\cir} \ar[r] \ar[d, "{\can \circ [p]^*}", swap] &
      (\tau_{\ge0}(\j_p^{\t\Cp}))^{\h\cir} \ar[d, "\phi"]  \ar[dl, "\phi'", swap] \\
      \j_p^{\t\cir} \ar[r] &
      (\tau_{\ge0}(\j_p^{\t\Cp}))^{\t\cir};
    \end{tikzcd}
  \end{equation}
  here the map $\phi'$ is induced by the connective cover map $\tau_{\ge0}(\j_p^{\t\Cp}) \to \j_p^{\t\Cp}$ (see \cref{t--cyc--sh}). By \cref{i--main--thhZp}, the lower horizontal map is an equivalence and $\smash{\can : (\tau_{\ge0}(\j_p^{\t\Cp}))^{\h\cir} \to (\tau_{\ge0}(\j_p^{\t\Cp}))^{\t\cir}}$ is a $\K(1)$-local equivalence: for the latter, note that its fiber identifies with $(\pcpl{\THH(\Z)})_{\h\cir}[1]$, which is a colimit of $\Z$-modules, hence is $\K(1)$-acyclic. We define $\smash{\varphi_0 : \pcpl{(\J_p^{\t\cir})} \to \pcpl{(\J_p^{\t\cir})}}$ to be the map corresponding to the composition
  \[
    \L_{\K(1)}((\tau_{\ge0}(\j_p^{\t\Cp}))^{\t\cir}) \lblto {\can^{-1}} \L_{\K(1)}((\tau_{\ge0}(\j_p^{\t\Cp}))^{\h\cir}) \lblto{\phi} \L_{\K(1)}((\tau_{\ge0}(\j_p^{\t\Cp}))^{\t\cir})
  \]
  under the zig-zag of equivalences
  \[
    \smash{\L_{\K(1)}((\tau_{\ge0}(\j_p^{\t\Cp}))^{\t\cir}) \from \L_{\K(1)}(\j_p^{\t\cir}) \isoto \pcpl{(\J_p^{\t\cir})}}
  \]
  where the second map is an equivalence by \cref{z--ol--coconnective}. That is, $\varphi_0$ is the composition
  \[
    \smash{\pcpl{(\J_p^{\t\cir})} \iso \L_{\K(1)}(\j_p^{\t\cir}) \to \L_{\K(1)}((\tau_{\ge0}(\j_p^{\t\Cp}))^{\t\cir}) \lblto{\can^{-1}} \L_{\K(1)}((\tau_{\ge0}(\j_p^{\t\Cp}))^{\h\cir}) \lblto{\phi'}
    \L_{\K(1)}(\j_p^{\t\cir}) \iso \pcpl{(\J_p^{\t\cir})}.}
  \]

  By the preceding discussion, the commutative squares \cref{z--ol--phi-TC--two} identify upon $\K(1)$-localization with commutative squares
  \begin{equation}
    \label{z--ol--phi-TC--J}
    \begin{tikzcd}
      \J_p^{\h\cir} \ar[r, "\can"] \ar[d, "\can", swap] &
      \pcpl{(\J_p^{\t\cir})} \ar[d, "\id"]  \\
      \pcpl{(\J_p^{\t\cir})} \ar[r, "\id"] &
      \pcpl{(\J_p^{\t\cir})},
    \end{tikzcd}
    \qquad\qquad
    \begin{tikzcd}
      \J_p^{\h\cir} \ar[r, "\can"] \ar[d, "{\can \circ [p]^*}", swap] &
      \pcpl{(\J_p^{\t\cir})} \ar[d, "\varphi_0"]  \\
      \pcpl{(\J_p^{\t\cir})} \ar[r, "\id"] &
      \pcpl{(\J_p^{\t\cir})},
    \end{tikzcd}
  \end{equation}
  the left hand one being the tautological one. Recalling the definition of $\TC$ and the equivalence $\L_{\K(1)}\TC(\S) \iso \L_{\K(1)}\TC(\j_p^\triv)$ from above, all that remains to be shown is that the map $\varphi_0$ and the commutativity of the right hand square in \cref{z--ol--phi-TC--J} are equivalent to the map $\varphi$ and the commutativity of the square \cref{z--ol--phi--compatibility} of \cref{z--ol--phi}.

  To prove this last assertion, we consider the map $\ku_p^\triv \to \sh(\ku_p^\triv)$. As in the discussion with $\j_p$ above, this induces commutative diagrams of $\E_\infty$-rings
  \[
    \begin{tikzcd}
      \ku_p^{\h\cir} \ar[r] \ar[d, "\can", swap] &
      (\tau_{\ge0}(\ku_p^{\t\Cp}))^{\h\cir} \ar[d, "\can"]  \\
      \ku_p^{\t\cir} \ar[r] &
      (\tau_{\ge0}(\ku_p^{\t\Cp}))^{\t\cir},
    \end{tikzcd}
    \qquad\qquad
    \begin{tikzcd}
      \ku_p^{\h\cir} \ar[r] \ar[d, "{\can \circ [p]^*}", swap] &
      (\tau_{\ge0}(\ku_p^{\t\Cp}))^{\h\cir} \ar[d, "\phi"]  \ar[dl, "\phi'", swap] \\
      \ku_p^{\t\cir} \ar[r] &
      (\tau_{\ge0}(\ku_p^{\t\Cp}))^{\t\cir}.
    \end{tikzcd}
  \]
  Here the lower horizontal maps are not equivalences, but $\smash{\can : (\tau_{\ge0}(\ku_p^{\t\Cp}))^{\h\cir} \to (\tau_{\ge0}(\ku_p^{\t\Cp}))^{\t\cir}}$ is again a $\K(1)$-local equivalence, as $\tau_{\ge0}(\ku_p^{\t\Cp})$ is an algebra over $\smash{\tau_{\ge0}(\j_p^{\t\Cp})}$ and hence also a $\Z$-module. We may therefore replicate the second description of the map $\varphi_0$ above to define a map of $\E_\infty$-rings $\smash{\Phi_0 : \pcpl{(\KU_p^{\t\cir})} \to \pcpl{(\KU_p^{\t\cir})}}$, namely the composition
  \begin{align*}
    \pcpl{(\KU_p^{\t\cir})} \iso \L_{\K(1)}(\ku_p^{\t\cir}) \to {}& \L_{\K(1)}((\tau_{\ge0}(\ku_p^{\t\Cp}))^{\t\cir}) \\
    \lblto{\can^{-1}} {}& \L_{\K(1)}((\tau_{\ge0}(\ku_p^{\t\Cp}))^{\h\cir}) \lblto{\phi'}
    \L_{\K(1)}(\ku_p^{\t\cir}) \iso \pcpl{(\KU_p^{\t\cir})},
  \end{align*}
  the first and last equivalences again coming from \cref{z--ol--coconnective}. The preceding commutative diagrams induce a commutative square
  \begin{equation}
    \label{z--ol--phi-TC--KU}
    \begin{tikzcd}
      \smash{\KU_p^{\h\cir}} \ar[r, "\can"] \ar[d, "{\can \circ [p]^*}", swap] &
      \smash{\pcpl{(\KU_p^{\t\cir})}} \ar[d, "\Phi_0"]  \\
      \pcpl{(\KU_p^{\t\cir})} \ar[r, "\id"] &
      \pcpl{(\KU_p^{\t\cir})},
    \end{tikzcd}
  \end{equation}
  identifying $\Phi_0$ with the map $\Phi$ of \cref{z--ol--KU-phi} and the commuting of \cref{z--ol--phi-TC--KU} with that of the square there. Finally, it follows from the compatibility of the constructions of $\varphi_0$ and $\Phi_0$ with respect to the $\Gamma$-equivariant map $\j_p \to \ku_p$ that the right hand square in \cref{z--ol--phi-TC--J} may be recovered by applying $\Gamma$--fixed points to \cref{z--ol--phi-TC--KU}, which finishes the proof.
\end{proof}

To be able to use \cref{z--ol--phi-TC} to calculate $\L_{\K(1)}\TC(\Z)$, we must better understand the map $\varphi$ of \cref{z--ol--phi}. The proof of \cref{z--ol--phi-TC} included a second construction of $\varphi$, and in \cref{z--ta}, we will give a third, invoking the $\K(1)$-local ambidexterity of the space $\B\Cp$ and some generalities on the Tate construction. There we will use this third perspective to prove \cref{z--ol--orbits} below, stating which requires us to recall a basic result in $\K(1)$-local stable homotopy theory.

\begin{notation}
  \label{z--ol--q}
  The class $q = \beta t + 1 \in \pi_0(\KU_p^{\B\cir})$ is represented by a map of $\K(1)$-local spectra $\pcpl{\J_p[\B\cir]} \to \KU_p$, unique up to homotopy; we will denote this map also by $q$.
\end{notation}

\begin{proposition}[{Westerland \citebare{westerland--higher-j}}]
  \label{z--ol--Bcir-splitting}
  The map of $\K(1)$-local spectra
  \[
    \begin{psmallmatrix}[p]_*\\q\end{psmallmatrix} : \pcpl{\J_p[\B\cir]} \to \pcpl{\J_p[\B\cir]} \oplus \KU_p
  \]
  is an equivalence.
\end{proposition}

\begin{proof}
  It suffices to check that the map induces an isomorphism on $\KU_p$-cohomology. The $\KU_p$-cohomology of $\pcpl{\J_p[\B\cir]}$ is the $\KU_p$-cohomology of $\B\cir$, which is even periodic and given in degree zero by $\Z_p\pow{q-1}$. The $\KU_p$-cohomology of $\KU_p$ is also even periodic, and given in degree zero by the completed group ring $\Z_p\pow{\Z_p^\times}$. Rewriting $\Z_p\pow{q-1}$ as the completed group ring $\Z_p\pow{\Z_p}$ (see \cref{j--ku--spherical-pow}), the map in degree zero $\KU_p$-cohomology identifies with the map of abelian groups
  \[
    \Z_p\pow{\Z_p} \times \Z_p\pow{\Z_p^\times} \iso \Z_p\pow{\Z_p \amalg \Z_p^\times} \to \Z_p\pow{\Z_p}
  \]
  induced by the isomorphism of profinite sets $\Z_p \amalg \Z_p^\times \to \Z_p$ given by multiplication by $p$ on the $\Z_p$ summand and the inclusion on the $\Z_p^\times$ summand.
\end{proof}

\begin{remark}
  \label{z--ol--Bcir-splitting-alt}
  The commutative group structure of $\B\cir$ determines an $\E_\infty$-ring structure on $\pcpl{\J_p[\B\cir]}$, with respect to which the maps $[p]_* : \pcpl{\J_p[\B\cir]} \to \pcpl{\J_p[\B\cir]}$ and $q : \pcpl{\J_p[\B\cir]} \to \KU_p$ canonically promote to $\E_\infty$-ring maps. Thus, \cref{z--ol--Bcir-splitting} in fact determines a product decomposition of $\pcpl{\J_p[\B\cir]}$ as an $\E_\infty$-ring. In \cite[\textsection 4.2]{carmeli-schlank-yanovski--cyclotomic}, Carmeli--Schlank--Yanovski give a different perspective on this product decomposition, which we will review in our proof of \cref{z--ta--nu}.

  We note in addition that the analysis of $\pcpl{\J_p[\B\cir]} \iso \L_{\K(1)}(\S[\B\cir])$ goes back to Ravenel's original work \cite[\textsection 9]{ravenel--localization}, and that results of a similar form to \cref{z--ol--Bcir-splitting} above and \cref{z--ol--Bcir-infinite-splitting} below may also be found in work of Hansen \cite{hansen--K-localizations} and Hesselholt--Madsen \cite{hesselholt-madsen--J}.
\end{remark}

\begin{notation}
  \label{z--ol--nu}
  In light of \cref{z--ol--Bcir-splitting}, we let $\nu : \pcpl{\J_p[\B\cir]} \to \pcpl{\J_p[\B\cir]}$ be the unique such map of spectra with identifications $[p]_* \circ \nu \iso \id$ and $q \circ \nu \iso 0$, and we let $\xi : \KU_p \to \pcpl{\J_p[\B\cir]}$ be the unique such map of spectra with identifications $[p]_* \circ \xi \iso 0$ and $q \circ \xi \iso \id$.
\end{notation}

\begin{lemma}
  \label{z--ol--orbits}
  The map of spectra $\pcpl{\J_p[\B\cir]}[1] \to \pcpl{\J_p[\B\cir]}[1]$ obtained by taking vertical fibers in \cref{z--ol--phi--compatibility} is canonically equivalent to the suspension of the map $\nu$ of \cref{z--ol--nu}.
\end{lemma}

It will be helpful to also record the following reformulation of \cref{z--ol--Bcir-splitting}.

\begin{proposition}
  \label{z--ol--Bcir-infinite-splitting}
  Let $c_* : \pcpl{\J_p[\B\cir]} \to \J_p$ be the map induced by the projection map $c : \B\cir \to *$, and recall that $\pcpl{\J_p\{\B\cir\}}$ denotes the fiber of $c_*$ (\cref{i--conv}\cref{i--conv--chains}). Then there is a canonical nullhomotopy of $c_* \circ \xi : \KU_p \to \J_p$, and the induced map
  \[
    (\xi \circ \nu^{\circ n})_{n \in \Z_{\ge0}} : \pcpl{\big(\bigoplus_{n \in \Z_{\ge0}} \KU_p\big)} \to \pcpl{\J_p\{\B\cir\}}
  \]
  is an equivalence.
\end{proposition}

\begin{proof}
   We have a canonical homotopy $c_* \iso c_* \circ [p]_*$, inducing the desired nullhomotopy of $c_* \circ \xi$. Next, we consider the following commutative diagram:
   \[
    \begin{tikzcd}[row sep=large,column sep=large]
      0 \ar[r] \ar[d] &
      \KU_p \ar[r, "{\begin{psmallmatrix}0 \\ \id\end{psmallmatrix}}"] \ar[d, "{\begin{psmallmatrix}0\\\id\end{psmallmatrix}}"] &
      \KU_p \oplus \KU_p \ar[r] \ar[d, "{\begin{psmallmatrix}0 \amp 0 \\ \id \amp 0 \\ 0 \amp \id\end{psmallmatrix}}"] &
      \cdots \\
      \pcpl{\J_p[\B\cir]} \ar[r, "{\begin{psmallmatrix}[p]_*\\q\end{psmallmatrix}}"] \ar[d, "\id"] &
      \pcpl{\J_p[\B\cir]} \oplus \KU_p \ar[r, "{\begin{psmallmatrix}[p]_* \amp 0 \\ q \amp 0 \\ 0 \amp \id\end{psmallmatrix}}"] \ar[d, "{\begin{psmallmatrix}\id \amp 0\end{psmallmatrix}}"] &
      \pcpl{\J_p[\B\cir]} \oplus \KU_p \oplus \KU_p \ar[r] \ar[d, "{\begin{psmallmatrix}\id \amp 0 \amp 0 \end{psmallmatrix}}"]  &
      \cdots \\
      \pcpl{\J_p[\B\cir]} \ar[r, "{[p]_*}"] \ar[d, "c_*"] &
      \pcpl{\J_p[\B\cir]} \ar[r, "{[p]_*}"] \ar[d, "c_*"] &
      \pcpl{\J_p[\B\cir]} \ar[r] \ar[d, "c_*"] &
      \cdots \\
      \J_p \ar[r, "\id"] &
      \J_p \ar[r, "\id"] &
      \J_p \ar[r] &
      \cdots
    \end{tikzcd}
  \]
  The vertical maps between the first three rows form a fiber sequence, and all the horizontal maps in the second row are equivalences (induced by that of \cref{z--ol--Bcir-splitting}). The claim now follows from taking $p$-completed colimits of the rows, noting that the map from the third row to the fourth induces an equivalence on $p$-completed colimits by \cref{z--ol--p-cover-graded}.
\end{proof}

The last ingredients we will need in our calculation of $\L_{\K(1)}\TC(\Z)$ concern the $\cir$-transfer map. Let us recall one description of it, which is how it will arise in the arguments below.

\begin{notation}
  \label{z--ol--basepoint}
  Let $R$ be an $\E_\infty$-ring. We let $b_* : R \to R[\B\cir]$ and $b^* : R^{\B\cir} \to R$ denote the maps induced by the basepoint $b : * \to \B\cir$. The \emph{$\cir$-transfer map} $\tr : R[\B\cir][1] \to R$ is the composition
  \[
    R[\B\cir][1] \lblto{\Nm} R^{\B\cir} \lblto{b^*} R,
  \]
  where $\Nm$ is the \emph{$\cir$-norm map}, i.e. the map whose cofiber is $R^{\t\cir}$. The \emph{reduced $\cir$-transfer map} $\o\tr : R\{\B\cir\}[1] \to R$ (appearing in \cref{i--tcz--tcs}) is the restriction of $\tr$ (recall from \cref{i--conv}\cref{i--conv--chains} that $R\{\B\cir\} = \fib(R[\B\cir] \to R)$).
\end{notation}

In \cref{z--tr}, we will review the $\cir$-norm and -transfer maps in more detail and prove \cref{z--ol--transfer} below.

\begin{notation}
  \label{z--ol--tau}
  We define $\tau : \KU_p \to \J_p[-1]$ to be the composition
  \[
    \KU_p \lblto{\xi} \pcpl{\J_p[\B\cir]} \lblto{\tr} \J_p[-1],
  \]
  where $\xi$ is as in \cref{z--ol--nu} and $\tr$ is the $\cir$-transfer map of \cref{z--ol--basepoint}.
\end{notation}

\begin{lemma}
  \label{z--ol--transfer}
  Suppose given a nullhomotopy of the composition $\smash{\J_p \lblto{1} \KU_p \lblto{\tau} \J_p[-1]}$, where $1$ denotes the unit map and $\tau$ is as in \cref{z--ol--tau}, inducing a map $\tau' : \KU_p/\J_p \to \J_p[-1]$. Then there exists (noncanonically) an equivalence $\cofib(\tau') \iso \KU_p[-1]$.
\end{lemma}

We now put together the results above to obtain the following description of $\L_{\K(1)}\TC(\Z)$.

\begin{theorem}
  \label{z--ol--k1}
  \begin{enumerate}[leftmargin=*]
  \item \label{z--ol--k1--description}
    There is a canonical equivalence of spectra
    \[
      \L_{\K(1)}\TC(\Z) \iso \J_p \oplus \J_p[1] \oplus \X',
    \]
    where $\X' \iso \cofib(\tau')$ for $\tau'$ a map as in \cref{z--ol--transfer} (so that $\X'$ is noncanonically equivalent to $\KU_p[-1]$).
  \item \label{z--ol--k1--linearization}
    With respect to the equivalence in \cref{z--ol--k1--description} and the equivalence
    \[
      \L_{\K(1)}\TC(\S) \iso \J_p \oplus \J_p[1] \oplus \L_{\K(1)}\Y
    \]
    determined by \cref{i--tcz--tcs}, the map $\L_{\K(1)}\TC(\S) \to \L_{\K(1)}\TC(\Z)$ induced by the unit map $\S \to \Z$ is canonically equivalent to the direct sum of the identity maps on $\J_p$ and $\J_p[1]$ with a certain map $\L_{\K(1)}\Y \to \X'$.
  \end{enumerate}
\end{theorem}

\begin{proof}
  We consider the commutative diagrams
  \begin{equation}
    \label{z--ol--k1--fibers}
    \begin{tikzcd}
      0 \ar[r] \ar[d] &
      \pcpl{\J_p[\B\cir]}[1] \ar[d, "\Nm"] \\
      \J_p^{\B\cir} \ar[r, "{\id-[p]^*}"] \ar[d, "\id", swap] &
      \J_p^{\B\cir} \ar[d, "\can"] \\
      \J_p^{\B\cir} \ar[r, "{\can - \can \circ [p]^*}"] &
      \J_p^{\t\cir},
    \end{tikzcd}
    \qquad\qquad
    \begin{tikzcd}
      \pcpl{\J_p[\B\cir]}[1] \ar[r, "\id-\nu"] \ar[d, "\Nm", swap] &
      \pcpl{\J_p[\B\cir]}[1] \ar[d, "\Nm"] \\
      \J_p^{\B\cir} \ar[r, "{\id-[p]^*}"] \ar[d, "\can", swap] &
      \J_p^{\B\cir} \ar[d, "\can"] \\
      \J_p^{\t\cir} \ar[r, "{\id - \varphi}"] &
      \J_p^{\t\cir},
    \end{tikzcd}
  \end{equation}
  where the columns are fiber sequences, $\varphi$ is as in \cref{z--ol--phi}, and $\nu$ is as in \cref{z--ol--nu,z--ol--orbits}. By \cref{z--ol--phi-TC}, the fibers of the bottom rows are canonically equivalent to $\L_{\K(1)}\TC(\S)$ and $\L_{\K(1)}\TC(\Z)$, respectively, and moreover the map $\L_{\K(1)}\TC(\S) \to \L_{\K(1)}\TC(\Z)$ induced by the unit map $\S \to \Z$ corresponds under these identifications to the map induced by the canonical map from the left hand diagram to the right hand diagram. We will understand these bottom fibers (and the map between them) by analyzing the fibers of the upper two rows (and the maps among them).

  Regarding the common middle rows in \cref{z--ol--k1--fibers}, it follows from \cref{z--ol--p-cover-graded} that the commutative squares
  \[
    \label{z--ol--k1--Bcir}
    \begin{tikzcd}
      \J_p \ar[r, "0"] \ar[d, "b_*"] &
      \J_p \ar[d, "b_*"] \\
      \J_p[\B\cir] \ar[r, "{\id-[p]_*}"] &
      \J_p[\B\cir],
    \end{tikzcd}
    \qquad
    \begin{tikzcd}
      \J_p^{\B\cir} \ar[r, "{\id-[p]^*}"] \ar[d, "b^*"] &
      \J_p^{\B\cir} \ar[d, "b^*"] \\
      \J_p \ar[r, "0"] &
      \J_p
    \end{tikzcd}
  \]
  (the right square being dual to the left) are cartesian. This identifies the fibers of the middle rows in \cref{z--ol--k1--fibers} with $\J_p \oplus \J_p[-1]$.

  Let us now analyze the left hand diagram in \cref{z--ol--k1--fibers}; this is simply a review of the identification of $\L_{\K(1)}\TC(\S)$ stated in \cref{z--ol--k1--linearization}. The composition
  \[
    \pcpl{\J_p[\B\cir]} \lblto{\Nm} \J_p^{\B\cir} \lblto{b^*} \J_p
  \]
  is the $\cir$-transfer map $\tr$, so taking horizontal fibers in the left hand diagram in \cref{z--ol--k1--fibers} gives a cofiber sequence
  \begin{equation}
    \label{z--ol--k1--S-tr}
    \pcpl{\J_p[\B\cir]} \lblto{\begin{psmallmatrix}0\\\tr\end{psmallmatrix}} \J_p \oplus \J_p[-1] \to \L_{\K(1)}\TC(\S).
  \end{equation}
  Invoking the splitting $\J_p[\B\cir] \iso \J_p \oplus \J_p\{\B\cir\}$ induced by the basepoint, and using that the restriction of the transfer map along $b_* : \J_p \to \J_p[\B\cir]$ is the Hopf map $\eta$ (\cref{z--tr--trivial}), which is canonically null because $p$ is odd, the preceding cofiber sequence may be rewritten as a cofiber sequence
  \[
  \J_p \oplus \pcpl{\J_p\{\B\cir\}} \lblto{\begin{psmallmatrix}0 \amp 0 \\ 0 \amp \o\tr\end{psmallmatrix}} \J_p \oplus \J_p[-1] \to \L_{\K(1)}\TC(\S).
  \]
  This recovers the splitting $\L_{\K(1)}\TC(\S) \iso \J_p \oplus \J_p[1] \oplus \L_{\K(1)}\Y$.

  We now move to the right hand diagram in \cref{z--ol--k1--fibers}. Using the splitting $\J_p[\B\cir] \iso \J_p \oplus \J_p\{\B\cir\}$ and \cref{z--ol--Bcir-infinite-splitting}, we obtain an equivalence
  \[
    \begin{psmallmatrix}b_* & (\xi \circ \nu^{\circ n})_{n \in \Z_{\ge0}}\end{psmallmatrix} : 
    \J_p \oplus \pcpl{\big(\bigoplus_{n \in \Z_{\ge0}} \KU_p\big)} \isoto \pcpl{\J_p[\B\cir]}.
  \]
  This identification carries the map $\id - \nu : \pcpl{\J_p[\B\cir]} \to \pcpl{\J_p[\B\cir]}$ to the map
  \[
    \begin{psmallmatrix}0 \amp 0 \\ \theta \amp \id - \sigma\end{psmallmatrix} : \J_p \oplus \pcpl{\big(\bigoplus_{n \in \Z_{\ge0}} \KU_p\big)} \to \J_p \oplus \pcpl{\big(\bigoplus_{n \in \Z_{\ge0}} \KU_p\big)},
  \]
  where $\theta : \J_p \to \bigoplus_{n \in \Z_{\ge0}} \KU_p$ is the composition of the unit map $1 : \J_p \to \KU_p$ with the inclusion of the factor indexed by $0$ and $\sigma : \bigoplus_{n \in \Z_{\ge0}} \KU_p \to \bigoplus_{n \in \Z_{\ge0}} \KU_p$ is the shift operator sending the factor indexed by $n$ to the factor indexed by $n+1$ by the identity map; the description for the first column of the matrix here comes from the canonical homotopies $[p]_* \circ b_* \iso b_*$ and $q \circ b_* \iso 1$, which determine a homotopy $b_* = \nu \circ b_* + \xi \circ 1$ and hence a homotopy $(\id - \nu) \circ b_* \iso \xi \circ 1$. The cofiber of the above map identifies in an evident manner with $\J_p \oplus \KU_p/\J_p$. Considering now all of the horizontal fibers in the right hand diagram in \cref{z--ol--k1--fibers}, we obtain a cofiber sequence
  \[
    \J_p \oplus \KU_p/\J_p \lblto{\begin{psmallmatrix}0 \amp 0 \\ 0 \amp \tau'\end{psmallmatrix}} \J_p \oplus \J_p[-1] \to \L_{\K(1)}\TC(\Z).
  \]
  Here the nullhomotopy of the upper left entry of the matrix follows from the construction; the nullhomotopies of the lower left and upper right entries use the vanishing of the composition $\tr \circ b_*$, as in the previous paragraph; and the map $\tau' : \KU_p/\J_p \to \J_p[-1]$ is a map as in \cref{z--ol--transfer}. This gives \cref{z--ol--k1--description}. We then see \cref{z--ol--k1--linearization} by comparing the analyses in this paragraph and the previous one.
\end{proof}

Finally, after recalling one more general fact, we may deduce the main result of the section, on $\pcpl{\TC(\Z)}$ before $\K(1)$-localization.

\begin{lemma}[{\citebare[Lemma 2.5 or Remark 2.14]{cmm--rigidity}}]
  \label{z--ol--connective}
  Let $M$ be a connective, $p$-complete cyclotomic spectrum. Then $\TC(M)$ is $(-1)$-connective.
\end{lemma}

\begin{proof}[Proof of \cref{i--tcz--thm}]
  By \cref{z--ol--connective}, the spectra $\pcpl{\TC(\S)} \iso \TC(\pcpl{\THH(\S)})$ and $\pcpl{\TC(\Z)} \iso \TC(\pcpl{\THH(\Z)})$ are $(-1)$-connective. Moreover, since the map $\pcpl{\THH(\S)} \to \pcpl{\THH(\Z)}$ is $(2p-3)$-connective, the map $\pcpl{\TC(\S)} \to \pcpl{\TC(\Z)}$ is $(2p-4)$-connective. We thus obtain a commutative diagram
  \[
    \begin{tikzcd}
      \pcpl{\TC(\S)} \ar[r] \ar[d] &
      \pcpl{\TC(\Z)} \ar[r] \ar[d] &
      \tau_{\ge-1}(\L_{\K(1)}\TC(\Z)) \ar[d] \\
      \tau_{\le 1}(\pcpl{\TC(\S)}) \ar[r, "\sim"] &
      \tau_{\le 1}(\pcpl{\TC(\Z)}) \ar[r] &
      \tau_{\le 1}\tau_{\ge-1}(\L_{\K(1)}\TC(\Z)),
    \end{tikzcd}
  \]
  where the lower left horizontal map is an equivalence by the preceding connectivity assertion, and the right hand square is induced by the $\K(1)$-localization map $\pcpl{\TC(\Z)} \to \L_{\K(1)} \TC(\Z)$. As discussed in the proof of \cref{t--pf--map}, this localization map induces isomorphisms on homotopy groups in degrees $\ge 2$, i.e. the induced map on vertical fibers in the right hand square is an equivalence, and hence that square is cartesian.

  Applying \cref{i--tcz--tcs,z--ol--k1}, the outer rectangle in the above diagram identifies canonically with a diagram
  \begin{equation}
    \label{z--ol--final--outer}
    \begin{tikzcd}
      \S_p \oplus \S_p[1] \oplus \pcpl{\Y} \ar[r] \ar[d] &
      \tau_{\ge-1}(\J_p) \oplus \tau_{\ge-1}(\J_p[1]) \oplus \tau_{\ge-1}(\X') \ar[d] \\
      \Z_p \oplus \Z_p[1] \oplus \Z_p[-1] \ar[r] &
      \tau_{\le1}\tau_{\ge-1}(\J_p) \oplus \tau_{\le1}\tau_{\ge-1}(\J_p[1]) \oplus \tau_{\le1}\tau_{\ge-1}(\X').
    \end{tikzcd}
  \end{equation}
  Here $\Y$ denotes the fiber of the reduced $\cir$-transfer map $\o\tr : \S\{\B\cir\}[1] \to \S$ and $\X'$ is a spectrum noncanonically equivalent to $\KU_p[-1]$; in the lower left hand corner, we have applied the canonical identifications
  \[
    \tau_{\le1}(\S_p) \iso \Z_p, \quad
    \tau_{\le1}(\S_p[1]) \iso \Z_p[1], \quad
    \tau_{\le1}(\pcpl{\Y}) \iso \Z_p[-1];
  \]
  and all the maps in the diagram are diagonal with respect to the written direct sum decompositions, with the first two factors of the horizontal maps being induced by the unit map $\S_p \to \J_p$.

  Putting together the previous two paragraphs, we obtain an identification
  \[
    \pcpl{\TC(\Z)} \iso \j_p \oplus \j_p[1] \oplus \X,
  \]
  where $\X$ is the spectrum defined by the pullback diagram
  \[
    \begin{tikzcd}
    \X \ar[r] \ar[d] &
    \tau_{\ge -1}(\X') \ar[d] \\
    \Z_p[-1] \ar[r] &
    \tau_{\le1}(\tau_{\ge -1}(\X'))
    \end{tikzcd}
  \]
  arising from the third factors in \cref{z--ol--final--outer}, and we also obtain the decomposition of the map $\pcpl{\TC(\S)} \to \pcpl{\TC(\Z)}$ as described in the theorem statement.

  To finish the proof, we give a noncanonical description of $\X$. Choosing an equivalence $\X' \iso \KU_p[-1]$ gives us a pullback diagram
  \[
    \begin{tikzcd}
    \X \ar[r] \ar[d] &
    \ku_p[-1] \ar[d] \\
    \Z_p[-1] \ar[r] &
    \tau_{\le1}(\ku_p[-1]).
    \end{tikzcd}
  \]
  We claim that the horizontal maps here are isomorphisms on $\pi_{-1}$. Once this is established, it is straightforward to use the decomposition $\ku_p \iso \bigoplus_{0 \le k \le p-2} \ell_p[2k]$ to obtain the desired equivalence $\X \iso \bigoplus_{0 \le k \le p, k \notin \{1,p-1\}} \ell_p[2k-1]$. On $\pi_{-1}$, the horizontal maps give an endomorphism of $\Z_p$, so it suffices to check that this is nonzero modulo $p$. For this, since the upper horizontal map is an isomorphism on $\pi_{2p-3}$, it suffices to check that $v_1$ acts nontrivially on $\pi_{-1}(\X/p)$. This is verified by Rognes \cite[Proof of Proposition 3.3]{rognes--whitehead}, by comparison with $\pi_{-1}(\Y/p)$. Alternatively, that $\pi_*(\TC(\Z)/p)$ is a free $\F_p[v_1]$-module is also proved by Liu--Wang \cite[Theorem 1.1]{liu-wang--TC-local-fields}.
\end{proof}


\subsection{Ambidexterity}
\label{z--ta}

Our aim in this subsection is to prove \cref{z--ol--orbits}, which amounts to a description of the map $\smash{\varphi : \pcpl{(\J_p^{\t\cir})} \to \pcpl{(\J_p^{\t\cir})}}$ of \cref{z--ol--phi}. In light of the diagram \cref{z--ol--phi--compatibility}, we may think of $\varphi$ as expressing a functoriality of the Tate construction with respect to the map of groups $[p] : \cir \to \cir$. The same can be said about the map $\smash{\Phi : \pcpl{(\KU_p^{\t\cir})} \to \pcpl{(\KU_p^{\t\cir})}}$ of \cref{z--ol--phi}. What we will discuss in this subsection is that this functoriality in fact canonically exists for all $\K(1)$-local spectra with $\cir$-action, by virtue of the map $[p] : \B\cir \to \B\cir$ being \emph{ambidextrous} for the $\infty$-category $\Spt_{\K(1)}$ of $\K(1)$-local spectra.

Let us begin by reviewing a more basic instance of this phenomenon.

\begin{remark}
  \label{z--ta--warmup}
  Let $f : H \to G$ be a map of finite groups, and let $X$ be a spectrum with $G$-action, which determines also an $H$-action by restriction along $f$. Then the standard functoriality of limits and colimits gives natural maps $f^* : X^{\h G} \to X^{\h H}$ and $f_* : X_{\h H} \to X_{\h G}$, and the former natural transformation is canonically lax symmetric monoidal. Assuming that $f$ is \emph{injective}, there is a canonical lax symmetric monoidal transformation $f^* : X ^{\t G} \to X^{\t H}$ making the diagram
  \begin{equation}
    \label{z--ta--warmup--tate}
    \begin{tikzcd}
      X^{\h G} \ar[r, "f^*"] \ar[d, "\can", swap] &
      X^{\h H} \ar[d, "\can"] \\
      X^{\t G} \ar[r, "f^*"] &
      X^{\t H}
    \end{tikzcd}
  \end{equation}
  commute.

  Ignoring lax symmetric monoidal structures, we may rephrase this by taking vertical fibers: under the assumption that $f$ is injective, there is a ``wrong way'' map on orbits $f^* : X_{\h G} \to X_{\h H}$ making the diagram
  \begin{equation}
    \label{z--ta--orbits}
    \begin{tikzcd}
      X_{\h G} \ar[r, "f^*"] \ar[d, "\Nm_G", swap] &
      X_{\h H} \ar[d, "\Nm_H"] \\
      X^{\h G} \ar[r, "f^*"] &
      X^{\h H}
    \end{tikzcd}
  \end{equation}
  commute. This can be understood as follows. Recall that the norm map $\Nm_G$ is induced by the endomorphism of $X$ given by the following composition:
  \[
    X \lblto{\delta} \prod_G X \iso \bigoplus_G X \lblto{\epsilon} X,
  \]
  with $\delta$ the diagonal map, $\epsilon$ the summation map, and the middle equivalence given by the (inverse to the) canonical map $\bigoplus_G X \to \prod_G X$ (an equivalence because $G$ is finite); the same goes for $\Nm_H$. With this in mind, pondering the diagram \cref{z--ta--orbits} suggests what the above wrong way map is: it is also a kind of norm map, over the set of orbits $G/H$. More precisely, it is given by applying $(-)_{\h G}$ to the composition
  \[
    X \lblto{\delta} \prod_{G/H} X \iso \bigoplus_{G/H} X.
  \]
  Again, the crucial mechanism here is the identification between finite coproducts and finite products of spectra.

  Now, we could define the map $f^* : X_{\h G} \to X_{\h H}$ as in the previous paragraph and then check that the diagram \cref{z--ta--orbits} commutes. Taking vertical cofibers in that diagram, we would obtain the map $f^* : X^{\t G} \to X^{\t H}$; however, from this perspective, the lax symmetric monoidal structure on this map is unclear. Of course, the lax symmetric monoidal structure on the Tate construction itself is unclear from its definition as the cofiber of the norm map.

  So, for the matter of lax symmetric monoidal structures, another perspective is needed. One account can be found in the work of Nikolaus--Scholze \cite[\textsection I.3]{nikolaus-scholze--TC}; there, the canonical map $\can : X^{\h G} \to X^{\t G}$ is characterized by a universal property, or rather two universal properties: one with lax symmetric monoidal structure and one without. These universal properties can be used to check that there is fact a unique natural transformation $f^* : X^{\t G} \to X^{\t H}$ making the diagram \cref{z--ta--warmup--tate} commute, and that it carries a unique lax symmetric monoidal structure making that commutative diagram one of lax symmetric monoidal functors. The first uniqueness assertion implies that the wrong way map on orbits described above is the unique such map making \cref{z--ta--orbits} commute.
\end{remark}

Here we are interested not in an injective map of finite groups but in the finite covering map $[p] : \cir \to \cir$. If we tried to run through the discussion in \cref{z--ta--warmup} in this case, then in place of the the finite orbit set $G/H$ we would find the classifying space $\B\Cp$, this being the fiber of the induced map on classifying spaces $[p] : \B\cir \to \B\cir$. And so where we identified a product over $G/H$ with a coproduct over $G/H$, we would need to identify a limit over $\B\C_p$ with a colimit over $\B\C_p$, i.e. homotopy $\Cp$-fixed points with homotopy $\Cp$-orbits. This is not something we can do generally in the $\infty$-category of spectra, but we can after $\K(1)$-localization:

\begin{theorem}[Hovey--Sadofsky \citebare{hovey-sadofsky--blueshift}]
  \label{z--ta--ambidexterity}
  Let $G$ be a finite group and let $X$ be a $\K(1)$-local spectrum with $G$-action. Then the norm map $\Nm_G : X_{\h G} \to X^{\h G}$ is a $\K(1)$-local equivalence.
\end{theorem}

\begin{remark}
  \label{z--ta--ambidexterity-generalizations}
  \cref{z--ta--ambidexterity} has several generalizations. First, the cited work of Hovey--Sadofsky in fact proves the same result with $\K(1)$ replaced by $\K(n)$ for any nonnegative integer $n$. Second, work of Kuhn \cite{kuhn--blueshift} strengthens the theorem to apply with $\K(n)$ replaced by $\mathrm{T}(n)$. Third, work of Hopkins--Lurie \cite{hopkins-lurie--ambidexterity} generalizes the result in a different direction, replacing the finite group $G$, or rather its classifying space $\B G$, with an arbitrary $\pi$-finite space; they also introduce the term \emph{ambidexterity} for this phenomenon. And finally, work of Carmeli--Schlank--Yanovski \cite{carmeli-schlank-yanovski--ambidexterity} gives a common generalization of all of these results.
\end{remark}

As suggested by the comments above, \cref{z--ta--ambidexterity} allows us to adapt the explicit construction of the wrong way map on orbits  in \cref{z--ta--warmup} to our setting. An abstract account of this construction, as well as of the results of Nikolaus--Scholze mentioned at the end of \cref{z--ta--warmup}, which is applicable in the present situation, may be found in \cite[\textsection\textsection 1.8, 1.9, 2.1]{raksit--tate}. We summarize the relevant points here.

\begin{notation}
  \label{z--ta--p-functors}
  The map $[p] : \B\cir \to \B\cir$ induces a restriction functor $\smash{[p]^* : \Spt^{\B\cir} \to \Spt^{\B\cir}}$, a left Kan extension functor $\smash{[p]_\sharp : \Spt^{\B\cir} \to \Spt^{\B\cir}}$, and finally a right Kan extension functor $\smash{[p]_* : \Spt^{\B\cir} \to \Spt^{\B\cir}}$. The standard functoriality of limits and colimits defines natural maps $[p]^* : X^{\h\cir} \to [p]^*(X)^{\h\cir}$ and $[p]_* : [p]^*(X)_{\h\cir} \to X_{\h\cir}$ for $\smash{X \in \Spt^{\B\cir}}$ (which specialize to the maps of \cref{z--ol--p-cover}).
\end{notation}

\begin{remark}
  \label{z--ta--kan-concrete}
  Recall that the Kan extension functors of \cref{z--ta--p-functors} are given by the formation of colimit and limit over the fiber $\B\Cp$ of $[p] : \B\cir \to \B\cir$. That is, for $X$ an $\cir$-equivariant spectrum,  $[p]_\sharp(X)$ and $[p]_*(X)$ refer to $X_{\h\Cp}$ and $X^{\h\Cp}$ with their residual $\cir$-actions. We note also that the norm map $\Nm_{\Cp} : X_{\h\Cp} \to X^{\h\Cp}$ canonically refines to an $\cir$-equivariant map (namely the norm map $\Nm_{[p]} : [p]_\sharp(X) \to [p]_*(X)$ associated to the map $[p] : \B\cir \to \B\cir$).
\end{remark}

\begin{construction}[{\citebare[Construction 1.9.2]{raksit--tate}}]
  \label{z--ta--orbits-map}
  Let $X$ be a $\K(1)$-local $\cir$-equivariant spectrum. With \cref{z--ta--ambidexterity}, \cref{z--ta--p-functors}, and \cref{z--ta--kan-concrete} in mind, we define the map $[p]^* : \pcpl{(X_{\h\cir})} \to \pcpl{([p]^*(X)_{\h\cir})}$ to be that obtained by applying $\pcpl{((-)_{\h\cir})}$ to the composition
  \[
    X \to X^{\B\Cp} \lbliso{\Nm_{\Cp}} \pcpl{(X[\B\Cp])}.
  \]
  Here the $\cir$-actions on $X[\B\Cp]$ and $X^{\B\Cp}$ are given by the constructions $[p]_\sharp[p]^*(X)$ and $[p]_*[p]^*(X)$; and the first map is induced by the projection map $\B\Cp \to *$, or in other words is the unit map $X \to [p]_*[p]^*(X)$.
\end{construction}

\begin{example}
  \label{z--ta--orbits-map-eg}
  Taking $X$ to be $\J_p$ with trivial action in \cref{z--ta--orbits-map}, we obtain a map $[p]^* : \pcpl{\J_p[\B\cir]} \to \pcpl{\J_p[\B\cir]}$.
\end{example}

\begin{proposition}[{\citebare[Theorem 1.9.3 and Example 2.1.19]{raksit--tate}}]
  \label{z--ta--outsource}
  The following statements hold:
  \begin{enumerate}[leftmargin=*]
  \item \label{z--ta--outsource--map}
    There is a unique natural transformation $[p]^* : \pcpl{((-)^{\t\cir})} \to \pcpl{([p]^*(-)^{\t\cir})}$ of functors from $\smash{\Spt_{\K(1)}^{\B\cir}}$ to $\Spt_{\K(1)}$ such that the diagram
    \[
      \begin{tikzcd}
        (-)^{\h\cir} \ar[r, "{[p]^*}"] \ar[d, "\can", swap] &
        {[p]^*(-)^{\h\cir}} \ar[d, "\can"] \\
        \pcpl{((-)^{\t\cir})} \ar[r, "{[p]^*}"] &
        \pcpl{([p]^*(-)^{\t\cir})}
      \end{tikzcd}
    \]
    of such functors commutes.
  \item \label{z--ta--outsource--monoidal}
    The natural transformation $[p]^* : \pcpl{((-)^{\t\cir})} \to \pcpl{([p]^*(-)^{\t\cir})}$ of \cref{z--ta--outsource--map} carries a unique lax symmetric monoidal structure making the above commutative diagram one of lax symmetric monoidal functors.
  \item \label{z--ta--outsource--fiber}
    The natural transformation $\pcpl{((-)_{\h\cir})} \to \pcpl{([p]^*(-)_{\h\cir})}$ obtained by taking vertical fibers in the diagram of \cref{z--ta--outsource--map} canonically identifies with the map $[p]^*$ of \cref{z--ta--orbits-map}.
  \end{enumerate}
\end{proposition}

We now use this result to settle a matter from \cref{z--ol}.

\begin{proposition}
  \label{z--ta--nu}
  The map $[p]^* : \pcpl{\J_p[\B\cir]} \to \pcpl{\J_p[\B\cir]}$ of \cref{z--ta--orbits-map-eg} is canonically equivalent to the map $\nu$ of \cref{z--ol--nu}.
\end{proposition}

\begin{proof}
  In this argument, for a map of spaces $f : T \to S$, we will use $f^*$, $f_\sharp$, and $f_*$ to denote the associated restriction, left Kan extension, and right Kan extension functors relating $\Spt^T_{\K(1)}$ and $\smash{\Spt^S_{\K(1)}}$ (so there is an implicit $p$-completion/$\K(1)$-localization in left Kan extension here, in contrast with elsewhere in the paper). We begin by recalling a description of $\nu$ from work of Carmeli--Schlank--Yanovski. To avoid confusion, let us denote the delooping of $[p] : \B\cir \to \B\cir$ by $[p]' : \B^2\cir \to \B^2\cir$, and let us write $a' : * \to \B^2\Cp$ and $b' : * \to \B^2\cir$ for the canonical basepoints. We may then form the following diagram of cartesian squares:
  \[
    \begin{tikzcd}
      \B\Cp \ar[r, "\beta"] \ar[d, "\pi", swap] &
      \B\cir \ar[r, "c"] \ar[d, "{[p]}", swap] &
      * \ar[d, "a'"] \\
      * \ar[r, "b"] &
      \B\cir \ar[r, "e"] \ar[d, "c", swap] &
      \B^2\Cp \ar[r, "\pi'"] \ar[d, "\beta'"] &
      * \ar[d, "b'"] \\
      &
      * \ar[r, "b'"] &
      \B^2\cir \ar[r, "{[p]'}"] &
      \B^2\cir.
    \end{tikzcd}
  \]
  Set $L := b'_\sharp(\J_p)$, which has underlying spectrum
  \[
    b'^*(L) = b'^*b'_\sharp(\J_p) \iso c_\sharp c^*(\J_p) \iso \pcpl{\J_p[\B\cir]}.
  \]
  The canonical homotopy $[p]'b' \iso b'$ induces an equivalence $[p]'_\sharp(L) \iso L$, and we deduce that $[p]'_\sharp(L)$ and $[p]'^*[p]'_\sharp(L)$ also have underlying spectrum $\pcpl{\J_p[\B\cir]}$. The unit map $\smash{\eta^\sharp_{[p]'} : L \to [p]'^*[p]'_\sharp(L)}$ has underlying map of spectra $[p]_* : \pcpl{\J_p[\B\cir]} \to \pcpl{\J_p[\B\cir]}$. Unravelling \cite[Propositions 4.5 and 4.11]{carmeli-schlank-yanovski--cyclotomic} and \cite[Proposition 4.3.4 and Remark 4.3.5]{carmeli-schlank-yanovski--height}, and using that the $\Spt_{\K(1)}$-cardinality of $\B\Cp$ is equal to $1$ \cite[Definition 2.1.5 and Proposition 2.2.5]{carmeli-schlank-yanovski--height}, shows that the composition
  \begin{equation}
    \label{z--ta--nu--CSY}
    [p]'^*[p]'_\sharp(L) \lblto{\Nm_{[p]'}} [p]'^*[p]'_*(L) \lblto{\epsilon^*_{[p]'}} L,
  \end{equation}
  where $\epsilon^*_{[p]'}$ is the counit map, has underlying map of spectra $\nu : \pcpl{\J_p[\B\cir]} \to \pcpl{\J_p[\B\cir]}$.

  We now explain how to identify the above description of $\nu$ with $[p]^*$. We have $b' \iso \beta'a'$, inducing an equivalence $L \iso \beta'_\sharp a'_\sharp(\J_p)$; at the level of underlying spectra, this is the equivalence
  \[
    \pcpl{\J_p[\B\cir]} \iso c_\sharp [p]_\sharp(\J_p) \iso \pcpl{((\pcpl{\J_p[\B\Cp]})_{\h\cir})}
  \]
  which was involved in \cref{z--ta--orbits-map}. We will use this expression for $L$ together with a couple of coherence properties of norm maps to recast the composition \cref{z--ta--nu--CSY}. First, we have the following commutative diagram of functors from $\smash{\Spt_{\K(1)}^{\B^2\Cp}}$ to $\smash{\Spt_{\K(1)}^{\B^2\cir}}$:
  \[
    \begin{tikzcd}
      {[p]'^*[p]'_\sharp \beta'_\sharp} \ar[d, "\Nm_{[p]'}", swap] &
      {[p]'^*b'_\sharp\pi'_\sharp} \ar[l, "\sim", swap] \ar[d, "\Nm_{\pi'}", swap] &
       \\
      {[p]'^*[p]'_* \beta'_\sharp} \ar[d, "\epsilon^*_{[p]'}", swap] &
      {[p]'^*b'_\sharp\pi'_*} \ar[l] &
      \beta'_\sharp\pi'^*\pi'_\sharp \ar[ul, "\sim", swap] \ar[d, "\Nm_{\pi'}"] \\
      \beta'_\sharp &
      &
      \beta'_\sharp\pi'^*\pi'_*. \ar[ll, "\epsilon^*_{\pi'}"] \ar[ul, "\sim", swap]
    \end{tikzcd}
  \]
  Here the diagonal arrows are induced by the projection isomorphism $\beta'_\sharp \pi'^* \isoto [p]'^*b'_\sharp$ and the right hand parallelogram commutes by naturality; the unlabelled arrow is induced by the exchange map $b'_\sharp \pi'_* \to [p]'_* \beta'_\sharp$, which is defined so that the lower trapezoid commutes; and the commutativity of the upper left square is established for example in \cite[Construction 1.6.6]{raksit--tate}. Second, we have the following commutative diagram of functors from $\smash{\Spt_{\K(1)}}$ to $\smash{\Spt_{\K(1)}^{\B^2\Cp}}$:
  \[
    \begin{tikzcd}
      &
      \pi'^*\pi'_\sharp a'_\sharp \ar[dr, "\Nm_{\pi'}", bend left=10] \ar[d, "\Nm_{\pi'a'}", swap] \ar[dl, "\sim", swap, bend right=10] &
       \\
      \pi'^* \ar[r, "\sim"] \ar[d, "\eta^*_{a'}", swap] &
      \pi'^*\pi'_*a'_* \ar[d, "\epsilon^*_{\pi'}"] &
      \pi'^*\pi'_* a'_\sharp \ar[l, "\Nm_{a'}", swap] \ar[d, "\epsilon^*_{\pi'}"] \\
      a'_*a'^*\pi'^* \ar[r, "\sim"] &
      a'_* &
      a'_\sharp. \ar[l, "\Nm_{a'}", swap]
    \end{tikzcd}
  \]
  Here the arrows labelled as equivalences come from $\pi'a' : * \to *$ being the identity map, which also gives the commutativity of the left half of the diagram; in the right half, the lower square commutes by naturality and the upper triangle commutes by e.g. \cite[Construction 1.7.8]{raksit--tate}.

  In combination, the preceding two commutative diagrams identify \cref{z--ta--nu--CSY} with the image under $\smash{\beta'_\sharp : \Spt_{\K(1)}^{\B^2\Cp} \to \Spt_{\K(1)}^{\B^2\cir}}$ of the composition
  \[
    \pi'^*(\J_p) \lblto{\eta^*_{a'}} a'_*a'^*\pi'^*(\J_p) \iso a'_*(\J_p) \lbliso{\Nm_{a'}} a'_\sharp(\J_p).
  \]
  After further applying the forgetful functor $\smash{b'^* : \Spt_{\K(1)}^{\B^2\cir} \to \Spt_{\K(1)}}$, this recovers the map $[p]^* : \pcpl{\J_p[\B\cir]} \to \pcpl{\J_p[\B\cir]}$ as defined in \cref{z--ta--orbits-map}.
\end{proof}

\begin{remark}
  \label{z--ta--nu-alt}
  By definition of the map $\nu$, \cref{z--ta--nu} is equivalent to the commutativity of the following diagram:
  \[
    \begin{tikzcd}
      &
      &
      \pcpl{\J_p[\B\cir]} \\
      \pcpl{\J_p[\B\cir]} \ar[r, "{[p]^*}"] \ar[urr, "\id", bend left=10] \ar[drr, "0", swap, bend right=10] &
      \pcpl{\J_p[\B\cir]} \ar[ur, "{[p]_*}", swap] \ar[dr, "q"] \\
      &
      &
      \KU_p.
    \end{tikzcd}
  \]
  It's possible to prove this directly, rather than appealing to the description of $\nu$ from \cite{carmeli-schlank-yanovski--height,carmeli-schlank-yanovski--cyclotomic} as we did above. The commutativity of the upper triangle boils down to the calculation of the $\Spt_{\K(1)}$-cardinality of $\B\Cp$ cited in the proof above. For the lower triangle, we are interested in the image of the class $q$ under a wrong way map on $\KU_p$-cohomology associated to $[p] : \B\cir \to \B\cir$. This wrong way map can be understood in representation theoretic terms, as in work of Yuan \cite[Example 2.16]{yuan--semiadditive}, and the desired commutativity results from the fact that the tautological complex character of $\cir$ (which determines the class $q$) has vanishing $\Cp$-fixed points.
\end{remark}

\begin{proof}[Proof of \cref{z--ol--orbits}]
  Let $[p]^* : \pcpl{((-)^{\t\cir})} \to \pcpl{([p]^*(-)^{\t\cir})}$ be as in \cref{z--ta--outsource}. As this map is natural and lax symmetric monoidal, applying it to the $\Gamma$-equivariant $\K(1)$-local $\E_\infty$-ring $\KU_p$ (equipped with trivial $\cir$-action) must recover the map $\smash{\Phi : \pcpl{(\KU_p^{\t\cir})} \to \pcpl{(\KU_p^{\t\cir})}}$ of \cref{z--ol--KU-phi}, by virtue of the uniqueness assertion in that statement. Again by naturality, it follows that the map $\smash{\varphi : \pcpl{(\J_p^{\t\cir})} \to \pcpl{(\J_p^{\t\cir})}}$ obtained by applying $\Gamma$-fixed points also identifies with $[p]^*$. It thus follows from \cref{z--ta--outsource}\cref{z--ta--outsource--fiber} that the map $\pcpl{\J_p[\B\cir]} \to \pcpl{\J_p[\B\cir]}$ obtained by taking vertical fibers in the diagram \cref{z--ol--phi--compatibility} identifies with the map $[p]^*$ of \cref{z--ta--orbits-map-eg}, and hence with the map $\nu$ by \cref{z--ta--nu}.
\end{proof}


\subsection{The $\cir$-transfer}
\label{z--tr}

In this subsection we will review one construction of the $\cir$-norm and -transfer maps in the $\infty$-category of spectra, along with their relation to the theory of complex orientations, and then we will give a description of the $\K(1)$-localized $\cir$-transfer map, in particular proving \cref{z--ol--transfer}.

\begin{notation}
  \label{z--tr--dual}
  In this subsection, for $R$ an $\E_\infty$-ring and $X$ an $R$-module, we will write $\dual{X}$ for the dual module $\uMap_{\Mod_R}(X,R)$ (the base $R$ to be understood from context).
\end{notation}

\begin{construction}
  \label{z--tr--cir-dual}
  The basepoint of $\cir$ induces a decomposition $\S[\cir] \iso \S \oplus \S\{\cir\} \iso \S \oplus \S[1]$; let us denote by $c : \S[\cir] \to \S$ and $d : \S[\cir] \to \S[1]$ the associated projection maps. The latter is adjoint to a map $\S \to \dual{\S[\cir]}[1]$, which extends uniquely to an $\cir$-equivariant map
  \[
    e : \S[\cir] \to \dual{\S[\cir]}[1].
  \]
  Said differently, $e$ is adjoint to the composition
  \[
    \S[\cir] \otimes \S[\cir] \iso \S[\cir \times \cir] \lblto{m} \S[\cir] \lblto{d} \S[1],
  \]
  where $m$ is induced by the multiplication map $\cir \times \cir \to \cir$.

  The map $e$ is an equivalence. Indeed, the equivalence $\S[\cir] \iso \S \oplus \S[1]$ induces an equivalence $\S[\cir] \otimes \S[\cir] \iso \S \oplus \S[1] \oplus \S[1] \oplus \S[2]$, under which the composition $d \circ m$ above identifies with the map
  \[
    \S \oplus \S[1] \oplus \S[1] \oplus \S[2] \lblto{\begin{psmallmatrix}0 \amp \id \amp \id \amp \eta\end{psmallmatrix}} \S[1],
  \]
  where $\eta$ denotes the Hopf map; and thus, under the induced equivalence $\dual{\S[\cir]}[1] \iso (\S \oplus \S[-1])[1] \iso \S[1] \oplus \S$, the map $e$ identifies with the map
  \[
    \S \oplus \S[1] \lblto{\begin{psmallmatrix}0 \amp \id \\ \id \amp \eta \end{psmallmatrix}} \S[1] \oplus \S,
  \]
  which is an equivalence.
\end{construction}

\begin{construction}
  \label{z--tr--nm-tr}
  As in \cref{z--tr--cir-dual}, let $c : \S[\cir] \to \S$ be map induced by the projection map $\cir \to *$; note that this is canonically $\cir$-equivariant. Applying duality and the equivalence $e$ of \cref{z--tr--cir-dual}, we obtain the composition of $\cir$-equivariant maps
  \[
    \S \lblto{\dual{c}} \dual{\S[\cir]} \lblto{e^{-1}} \S[\cir][-1];
  \]
  let us denote the composite map by $f : \S \to \S[\cir][-1]$.

  Now, recalling that $\cir$-equivariant spectra may be identified with modules over the group algebra $\S[\cir]$, we may consider, for any $\cir$-equivariant spectrum $X$, the composition
  \[
    X_{\h\cir}[1] \iso X \otimes_{\S[\cir]} \S[1] \lblto{f} X \otimes_{\S[\cir]} \S[\cir] \iso X.
  \]
  We denote this composite map by $\tr : X_{\h\cir}[1] \to X$; this is the \emph{$\cir$-transfer map}. It comes by construction as an $\cir$-equivariant map, where the source has trivial action, and so it factors uniquely through a map $\Nm : X_{\h\cir}[1] \to X^{\h\cir}$; this is the \emph{$\cir$-norm map}.
\end{construction}

\begin{remark}
  \label{z--tr--trivial}
  Let $R$ be an $\E_\infty$-ring. Equipping it with the trivial $\cir$-action and specializing \cref{z--tr--nm-tr} to this case, we obtain the norm map $\Nm : R[\B\cir][1] \to R^{\B\cir}$ and the transfer map $\tr : R[\B\cir][1] \to R$ appearing in \cref{z--ol--basepoint}. We note that the discussion in \cref{z--tr--cir-dual} implies that the composition
  \[
    R[1] \lblto{b_*} R[\B\cir][1] \lblto{\tr} R
  \]
  canonically identifies with the Hopf map $\eta$ (recall that $b_*$ denotes the map induced by the basepoint of $\B\cir$).
\end{remark}

We now explain the relation of the $\cir$-transfer map to complex orientations.

\begin{notation}
  \label{z--tr--rep}
  We let $\V$ denote the tautological representation of $\cir$ on $\mathbb{C}$, which we regard as a vector bundle on $\B\cir$. We let $\S^\V$ denote the associated stable sphere bundle, an $\cir$-equivariant spectrum with underlying spectrum $\S[2]$. And we let $\Th_{\B\cir}(\V)$ denote the associated Thom spectrum, given by $(\S^\V)_{\h\cir}$. Replacing $\V$ by its virtual negative $-\V$ gives rise similarly to $\S^{-\V}$ and $\Th_{\B\cir}(-\V)$.
\end{notation}

\begin{remark}
  \label{z--tr--euler}
  Recall that we have a cofiber sequence of $\cir$-equivariant spectra
  \[
    \S[\cir] \lblto{c} \S \to \S^\V,
  \]
  sometimes called the \emph{Euler sequence}. Applying $(-)_{\h\cir}$, we obtain a cofiber sequence
  \[
    \S \lblto{b_*} \S[\B\cir] \to \Th_{\B\cir}(\V).
  \]
  This gives an equivalence $\Th_{\B\cir}(\V) \iso \S\{\B\cir\}$, under which the quotient map $\S^\V \to \Th_{\B\cir}(\V)$ identifies (non-equivariantly) with the map $\S[2] \iso \S\{\sph^2\} \to \S\{\B\cir\}$ induced by the map of pointed spaces $\sph^2 \iso \Sigma \cir \to \B\cir$ adjoint to the unit equivalence $\cir \isoto \Omega \B\cir$.
\end{remark}

\begin{definition}
  \label{z--tr--orientation}
  Let $R$ be an $\E_\infty$-ring and let $X$ be an $R$-module. A \emph{complex orientation of $X$ over $R$} is an equivalence of $\cir$-equivariant $R$-modules $X \otimes \S^\V \iso X[2]$; in light of \cref{z--tr--euler}, this is equivalent to a map of $R$-modules $X \otimes_R R\{\B\cir\} \to X[2]$ extending the canonical equivalence $X \otimes_R R\{\sph^2\} \isoto X[2]$. We say that $X$ is \emph{complex orientable over $R$} if it admits a complex orientation over $R$.
\end{definition}

\begin{remark}
  \label{z--tr--orientation-ring}
  In the situation of \cref{z--tr--orientation}, suppose that $X$ is equipped with an $\A_2$-algebra structure in $\Mod_R$, i.e. maps of $R$-modules $v : R \to X$ and $w : X \otimes_R X \to X$ and a commutative diagram
  \[
    \begin{tikzcd}
      X \ar[r, "\id \otimes v"] \ar[dr, "\id", swap] &
      X \otimes_R X \ar[d, "w"] \\
      &
      X.
    \end{tikzcd}
  \]
  Then a complex orientation of $X$ over $R$ is equivalent to a map of $R$-modules $R\{\B\cir\} \to X[2]$ extending the map $\smash{R\{\sph^2\} \iso R[2] \lblto{v} X[2]}$. This is a more standard notion of complex orientation. Note that it does not actually depend on $R$.
\end{remark}

\begin{proposition}
  \label{z--tr--null}
  Let $R$ be an $\E_\infty$-ring and let $X$ be an $R$-module. Then a nullhomotopy of the transfer map $\tr : X \otimes_R R[\B\cir][1] \to X$ is equivalent to an equivalence of $\cir$-equivariant $R$-modules $X \otimes \S^{-\V} \iso X[-2]$.
\end{proposition}

\begin{proof}
  Dual to the discussion in \cref{z--tr--euler}, we have a fiber sequence $\smash{\S^{-\V} \to \S \lblto{f} \S[\cir][-1]}$, which upon applying $(-)_{\h\cir}$ gives a fiber sequence $\smash{\Th_{\B\cir}(-\V) \to \S[\B\cir] \lblto{\tr} \S[-1]}$, where the boundary map $\S[-2] \to \Th_{\B\cir}(-\V)$ is the quotient map (regarding $\S[-2]$ as the underlying spectrum of $\S^{-\V}$). Tensoring with $X$ and rotating, we have a fiber sequence
  \[
    X \otimes_R R[\B\cir][-1] \lblto{\tr} X[-2] \to X \otimes \Th_{\B\cir}(-\V),
  \]
  Giving a nullhomotopy of the first map is equivalent to giving a splitting $X \otimes \Th_{\B\cir}(-\V) \to X[-2]$ of the second map. As stated above, the second map is the quotient map, so such a splitting is equivalent to an equivalence of $\cir$-equivariant $R$-modules $X \otimes \S^{-\V} \iso X[-2]$. 
\end{proof}

\begin{corollary}
  \label{z--tr--null-orientation}
  Let $R$ be an $\E_\infty$-ring and let $X$ be an $R$-module. Then:
  \begin{enumerate}
  \item \label{z--tr--null-orientation--maps}
    there are canonical maps from the space of nullhomotopies of the transfer map
    \[
      \tr : X \otimes_R R[\B\cir][1] \to X
    \]
    to the space of complex orientations of $\dual{X}$ over $R$ and from the space of complex orientations of $X$ over $R$ to the space of nullhomotopies of the transfer map
    \[
      \tr : \dual{X} \otimes_R R[\B\cir][1] \to \dual{X};
    \]
  \item \label{z--tr--null-orientation--equivalence}
    if $X$ is a reflexive $R$-module, the two maps in \cref{z--tr--null-orientation--maps} are equivalences.
  \end{enumerate}
\end{corollary}

\begin{proof}
  This follows from comparing \cref{z--tr--orientation,z--tr--null}, the maps in the claim being given by $R$-linear duality.
\end{proof}

\begin{remark}
  \label{z--tr--null-ring}
  We make a parallel to \cref{z--tr--orientation-ring} in the situation of \cref{z--tr--null}: supposing that $X$ is equipped with an $\A_2$-algebra structure in $\Mod_R$, a nullhomotopy of the transfer map $\tr : X \otimes_R R[\B\cir][1] \to X$ is equivalent to a nullhomotopy of the composite
  \[
    R[\B\cir][1] \lblto{\tr} R \to X
  \]
  (the second map being the unit of the $\A_2$-algebra structure).
\end{remark}

Our goal for the remainder of the subsection is to analyze the $\K(1)$-local transfer map which featured in \cref{z--ol}.

\begin{notation}
  \label{z--tr--generator}
  Recall that $\J_p \iso \KU_p^{\h\Gamma}$. A choice of element $g \in \Gamma$
  which topologically generates $\Z_p^\times$ gives a fiber sequence
  \[
    \J_p \lblto{1} \KU_p \lblto{\psi^g - \id} \KU_p.
  \]
  This determines an equivalence $\rho_g : \KU_p/\J_p \iso \KU_p$ and a boundary map $\partial_g : \KU_p \to \J_p[1]$. We let $\partial : \KU_p/\J_p \to \J_p[1]$ denote the tautological boundary map, so that $\partial_g \iso \partial \circ \rho_g^{-1}$.
\end{notation}

\begin{remark}
  \label{z--tr--noncanonical}
  By applying $(-)^{\KU_p}$ to the fiber sequence of \cref{z--tr--generator}, one finds that the map of $\KU_p$-modules $\smash{\KU_p[-1] \to \J_p^{\KU_p}}$ classifying $\partial_g$ is an equivalence. More concretely, we have that $\pi_*(\J_p^{\KU_p})$ is concentrated in odd degrees, with $\pi_{2n-1}(\J_p^{\KU_p})$ being a free $\Z_p$-module of rank one for each $n \in \Z$, generated by the class corresponding to the composition
  \[
    \KU_p \lblto{\beta^n} \KU_p[-2n] \lblto{\partial_g} \J_p[-2n+1].
  \]
  We will use the following consequence of this calculation below: if $X$ is a spectrum equivalent (noncanonically) to $\KU_p$ and $\chi : X \to \J_p[-1]$ is a map that is not divisible by $p$, then $\chi$ is homotopic to the composition of an equivalence $X \to \KU_p[-2]$ with $\partial_g : \KU_p[-2] \to \J_p[-1]$, and hence $\cofib(\chi)$ is equivalent (noncanonically) to $\KU_p[-1]$.
\end{remark}

\begin{lemma}
  \label{z--tr--J-nonzero}
  The composition $\J_p[\B\cir][1] \lblto{\tr} \J_p \to \J_p/p$ (the second map being the tautological one) is nonzero.
\end{lemma}

\begin{proof}
  (We will implicitly invoke \cref{z--tr--orientation-ring,z--tr--null-ring} here, noting that $\S/p$ admits an $\A_2$-ring structure, as $p$ is odd.) Suppose that it were zero. Applying \cref{z--tr--null-orientation} with $R = \J_p$ and $X = \J_p/p$, we deduce that $\dual{(\J_p/p)}$ is complex orientable over $\J_p$. Since $\dual{(\J_p/p)} \iso \J_p/p[-1]$, this implies that $\J_p/p$ is complex orientable over $\J_p$. But this is false: there are no nonzero maps $\J_p\{\B\cir\} \to \J_p/p[2]$, by virtue of \cref{z--ol--Bcir-infinite-splitting} and the fact that $\smash{\pi_{-2}((\J_p/p)^{\KU_p}) \iso 0}$ (\cref{z--tr--noncanonical}).
\end{proof}

\begin{lemma}
  \label{z--tr--nu}
  Let $\nu : \pcpl{\J_p[\B\cir]} \to \pcpl{\J_p[\B\cir]}$ be as in \cref{z--ol--nu}. Then there is a canonical homotopy $\tr \circ \nu \iso \tr : \pcpl{\J_p[\B\cir]}[1] \to \J_p$.
\end{lemma}

\begin{proof}
  Consider the diagram
  \[
    \begin{tikzcd}
      \pcpl{\J_p[\B\cir]}[1] \ar[r, "\nu"] \ar[d, "\Nm", swap] &
      \pcpl{\J_p[\B\cir]}[1] \ar[d, "\Nm"] \\
      \J_p^{\B\cir} \ar[r, "{[p]^*}"] \ar[d, "b^*", swap] &
      \J_p^{\B\cir} \ar[d ,"b^*"] \\
      \J_p \ar[r, "\id"] &
      \J_p.
    \end{tikzcd}
  \]
  The top square commutes by \cref{z--ol--orbits}, and the bottom square commutes evidently. The resulting commutativity of the outer rectangle gives the claim.
\end{proof}

\begin{theorem}
  \label{z--tr--J-description}
  Under the equivalence (of \cref{z--ol--Bcir-infinite-splitting})
  \[
    (\xi \circ \nu^{\circ n})_{n \in \Z_{\ge0}} : \pcpl{\big(\bigoplus_{n \in \Z_{\ge0}} \KU_p\big)} \isoto \pcpl{\J_p\{\B\cir\}},
  \]
  the reduced $\cir$-transfer map $\o\tr : \pcpl{\J_p\{\B\cir\}} \to \J_p[-1]$ identifies with the composition
  \begin{equation}
    \label{z--tr--J-description--map}
    \pcpl{\big(\bigoplus_{n \in \Z_{\ge0}} \KU_p\big)} \lblto{(\id)_{n \in \Z_{\ge 0}}} \KU_p \lblto{\tau} \J_p[-1],
  \end{equation}
  where $\tau = \tr \circ \xi$ (\cref{z--ol--tau}). Moreover, $\tau$ is not divisible by $p$.
\end{theorem}

\begin{proof}
  The first statement is immediate from \cref{z--tr--nu}. We deduce that if $\tau$ were divisible by $p$, then the reduced transfer map $\o\tr$ would be too. In fact, then the full transfer map $\tr : \J_p[\B\cir][1] \to \J_p$ would be divisible by $p$, since the restriction of $\tr$ along the basepoint $b_* : \J_p \to \J_p[\B\cir]$ is zero (by \cref{z--tr--trivial}, it is given by the Hopf map $\eta$, which is zero here since $p$ is odd). But we know that this is not the case by \cref{z--tr--J-nonzero}.
\end{proof}

\begin{proof}[Proof of \cref{z--ol--transfer}]
  By \cref{z--tr--noncanonical}, it suffices to show that the map $\tau' : \KU_p/\J_p \to \J_p[-1]$ is not divisible by $p$. If it were, then so too would $\tau$, contradicting \cref{z--tr--J-description}.
\end{proof}

\begin{corollary}
  \label{z--tr--fiber-description}
  As in \cref{i--tcz--tcs}, let $\Y$ denote the fiber of the reduced $\cir$-transfer map $\o\tr : \S\{\B\cir\}[1] \to \S$. Then there is canonical equivalence
  \[
    \L_{\K(1)}\Y \iso \pcpl{\big(\bigoplus_{n \in \Z_{\ge 0}} \KU_p[1]\big)} \oplus \Y',
  \]
  where $\Y'$ is a spectrum noncanonically equivalent to $\KU_p[-1]$.
\end{corollary}

\begin{proof}
  By the identification \cref{z--tr--J-description--map} in \cref{z--tr--J-description} of the $\K(1)$-localization of $\o\tr$, we have a fiber sequence
  \[
    \pcpl{\big(\bigoplus_{n \in \Z_{\ge 0}} \KU_p[1]\big)} \to \L_{\K(1)}\Y \to \Y',
  \]
  where $\Y'$ is defined to be $\fib(\tau : \KU_p[1] \to \J_p)$ (equivalently $\cofib(\tau : \KU_p \to \J_p[-1])$) and the first term comes from the fiber sequence
  \[
    \pcpl{\big(\bigoplus_{n \in \Z_{\ge0}} \KU_p\big)} \lblto{\id-\sigma} \pcpl{\big(\bigoplus_{n \in \Z_{\ge0}} \KU_p\big)} \lblto{(\id)_{n \in \Z_{\ge 0}}} \KU_p,
  \]
  where $\sigma$ is the shift operator as in the proof of \cref{z--ol--k1}. \cref{z--tr--J-description} also tells us that $\tau$ is not divisible by $p$, so from \cref{z--tr--noncanonical} we know that $\Y'$ is equivalent to $\KU_p[-1]$. A splitting of the first fiber sequence above is induced by a splitting of the second, a canonical one being given by taking the inclusion of the zeroth factor $\KU_p \to \bigoplus_{n \in \Z_{\ge0}} \KU_p$ as a section of the summation map $\bigoplus_{n \in \Z_{\ge0}} \KU_p \to \KU_p$.
\end{proof}

\begin{corollary}
  \label{z--tr--k1-orientation}
  Let $X$ be a $\K(1)$-local $\A_2$-ring. Then the unit map $v : \J_p \to X$ extends to a map of spectra $v' : \KU_p \to X$ if and only if $X$ is complex orientable.
\end{corollary}

\begin{proof}
  By \cref{z--tr--null-orientation,z--tr--null-ring}, a complex orientation of $X$ is equivalent to a nullhomotopy of the composite $\smash{\pcpl{\J_p[\B\cir]} \lblto{\tr} \J_p[-1] \lblto{v} X[-1]}$. And by \cref{z--tr--J-description,z--tr--noncanonical}, such a nullhomotopy exists if and only if there exists a nullhomotopy of the composite $\smash{\KU_p/\J_p[-1] \lblto{\partial} \J_p \lblto{v} X}$.
\end{proof}

\section{The canonical map $\smash{\j_p^\triv \to \sh(\j_p^\triv)}$}
\label{j}

In this section, we prove \cref{i--nil--thm}, which will be put to use in tandem with \cref{i--main--thhZp} in the next two sections of the paper. We prepare in \cref{j--nil} and \cref{j--cyc} by reviewing the notion of ``nilpotence'' for $\cir$-equivariant spectra and setting a bit of notation that will be used in the remainder of the section; in \cref{j--ku}, we analyze the canonical map $\phi^0 : \ku_p^\triv \to \sh(\ku_p^\triv)$; and in \cref{j--j}, we finally deduce the desired result about the canonical map $\phi^0 : \j_p^\triv \to \sh(\j_p^\triv)$.


\subsection{Nilpotence}
\label{j--nil}

Here we follow Mathew \cite[\textsection 4]{mathew--descent-nilpotence}, though the discussion there is in the context of a finite group $G$ rather than the circle group $\cir$.

\begin{notation}
  \label{j--nil--free}
  Consider the $\cir$-equivariant spectrum $\S[\cir]$. For any $\cir$-equivariant $\E_\infty$-ring $R$, we set $R[\cir] := R \otimes \S[\cir] \in \Mod_R(\Spt^{\B\cir})$.
\end{notation}

\begin{definition}
  \label{j--nil--def}
  Let $R$ be an $\cir$-equivariant $\E_\infty$-ring. We say that an $\cir$-equivariant $R$-module $X$ is \emph{nilpotent} if it is contained in the thick tensor ideal of $\smash{\Mod_R(\Spt^{\B\cir})}$ generated by $R[\cir]$, i.e. the thick subcategory generated by objects of the form $R[\cir] \otimes_R Y$ for $\smash{Y \in \Mod_R(\Spt^{\B\cir})}$.
\end{definition}

\begin{remark}
  \label{j--nil--adjunction}
  Let $R \to R'$ be a map of $\cir$-equivariant $\E_\infty$-rings. Then both the base change functor $\smash{\Mod_R(\Spt^{\B\cir}) \to \Mod_{R'}(\Spt^{\B\cir})}$ and the forgetful functor $\smash{\Mod_{R'}(\Spt^{\B\cir}) \to \Mod_R(\Spt^{\B\cir})}$ preserve nilpotent objects.
\end{remark}

\begin{proposition}
  \label{j--nil--tate-vanishing}
  Let $R$ be an $\cir$-equivariant $\E_\infty$-ring and let $X$ be a nilpotent $\cir$-equivariant $R$-module. Then $X^{\t\cir} \iso 0$.
\end{proposition}

\begin{proof}
  By \cref{j--nil--adjunction}, we may assume $R = \S$ (with trivial $\cir$-action). The result is then due to Klein \cite[Corollary 10.2]{klein--dualizing} (or see \cite[\textsection 2.4]{raksit--hhdr}).
\end{proof}

\begin{proposition}[{cf. \citebare[Theorem 4.15]{mathew--descent-nilpotence}}]
  \label{j--nil--or}
  Let $R$ be an $\E_\infty$-ring equipped with a complex orientation $t \in \smash{\pi_{-2}(R^{\B\cir})}$, and let $X$ be an $\cir$-equivariant $R$-module. Then $X$ is nilpotent if and only if the action of $t$ on $X^{\h\cir}$ is nilpotent as a map in $\smash{\Mod_{R^{\B\cir}}}$.
\end{proposition}

\begin{proof}
  We recall from \cite[Theorem 7.43]{mnn--nilpotence-descent} that the fixed points construction $(-)^{\h\cir}$ defines an equivalence of symmetric monoidal $\infty$-categories $\smash{\Mod_R^{\B\cir} \to \cpl{(\Mod_{R^{\B\cir}})}_t}$, sending $R[\cir]$ to the suspension of $R \iso R^{\B\cir}/t$. This reduces us to checking that an $R^{\B\cir}$-module $X'$ lies in the thick subcategory of $\Mod_{R^{\B\cir}}$ generated by objects of the form $Y'/t$ for $Y' \in \Mod_{R^{\B\cir}}$ if and only if the action of $t$ on $X'$ is nilpotent. This follows from the fact that the latter condition is equivalent to $X'$ being a retract of $X'/t^n$ for some positive integer $n$.
\end{proof}

\begin{remark}
  \label{j--nil--mod-p-v1}
  Let $X$ be an $\cir$-equivariant spectrum. If $X/(p,v_1)$ is nilpotent, then by \cref{j--nil--tate-vanishing}, we have
  \[
    X^{\t\cir}/(p,v_1) \iso (X/(p,v_1))^{\t\cir} \iso 0 \implies \cpl{(X^{\t\cir})_{(p,v_1)}} \iso 0,
  \]
  where $\cpl{(-)}_{(p,v_1)}$ denotes Bousfield localization with respect to $\S/(p,v_1)$. In the following sections, we will want to know that in fact $\smash{\pcpl{(X^{\t\cir})} \iso 0}$. To this end, let us note that certain hypotheses on $X$ imply that the canonical map
  \[
    \smash{\pcpl{(X^{\t\cir})} \to \cpl{(X^{\t\cir})_{(p,v_1)}}}
  \]
  is an equivalence. This holds if and only if $X^{\t\cir}/p$ is local with respect to $\S/(p,v_1)$, which is equivalent to the vanishing of the limit of the diagram
  \[
    \cdots \to X^{\t\cir}/p[3(2p-2)] \lblto{v_1} X^{\t\cir}/p[2(2p-2)] \lblto{v_1} X^{\t\cir}/p[2p-2] \lblto{v_1} X^{\t\cir}/p.
  \]
  This vanishing occurs in particular when:
  \begin{enumerate}
  \item $X$ admits the structure of $\cir$-equivariant $\Z$-module, as then the maps in the above diagram are null;
  \item $X$ is bounded below, as the terms in the diagram
    \[
      \cdots \to X/p[3(2p-2)] \lblto{v_1} X/p[2(2p-2)] \lblto{v_1} X/p[2p-2] \lblto{v_1} X/p
    \]
    are increasing in connectivity, and applying $(-)^{\t\cir}$ to such a diagram results in one whose limit vanishes.
  \end{enumerate}
\end{remark}


\subsection{Some notation}
\label{j--cyc}

We will make some use in this section of the following elaboration on the map $\phi^0_X$ of \cref{t--cyc--sh}.

\begin{notation}
  \label{j--cyc--filtered-phi}
  For $X$ a cyclotomic spectrum, we define
  \[
    \phi_X^* := \tau_{\ge*}(\phi_X) : \tau_{\ge*}(X) \to \tau_{\ge*}(X^{\t\Cp}),
    \quad K_X := \cofib(\phi_X^0 : \tau_{\ge0}(X) \to \tau_{\ge0}(X^{\t\Cp})),
  \]
  the former a map of filtered $\cir$-equivariant spectra and the latter an $\cir$-equivariant spectrum.
\end{notation}

\begin{remark}
  \label{j--cyc--filtered-phi-monoidality}
  The construction $X \mapsto \phi_X^*$ above determines a lax symmetric monoidal functor $\smash{\CycSpt \to \Fun(\Delta^1,\FilSpt^{\B\cir})}$, where $\FilSpt$ denotes the $\infty$-category $\Fun(\Z^\op,\Spt)$ of filtered spectra, equipped with the Day convolution symmetric monoidal structure. In particular, for $R$ a cyclotomic $\E_\infty$-ring, the map $\phi_R^0 : \tau_{\ge0}(R) \to \tau_{\ge 0}(R^{\t\Cp})$ is a map of $\cir$-equivariant $\E_\infty$-rings, $K_R$ is an $\cir$-equivariant module over $\tau_{\ge0}(R)$, and for $X$ an $R$-module and $n$ any integer, the map $\phi_X^n : \tau_{\ge n}(X) \to \tau_{\ge n}(X^{\t\Cp})$ is linear over $\phi_R^0$.
\end{remark}


\subsection{On $\ku_p^\triv$}
\label{j--ku}

Tate constructions are especially calculable in the setting of complex oriented ring spectra. In this subsection, we make some calculations concerning connective complex K-theory, from which we will then bootstrap in our analysis of $\j_p$ in the next subsection.

\cref{t--tCp--ku,t--tCp--pow,t--tCp--zetap} will be used again here, and sometimes the symbol $q'$ will be used in place of $q$, for reasons that will become clear below.

\begin{proposition}
  \label{j--ku--homotopy}
  We have the following identifications of graded rings:
  \begin{align*}
    &\pi_*(\ku_p^{\h\cir}) \iso \Z_p\pow{q-1}[\beta,t]/(\beta t - (q-1))
    &&\pi_*(\ku_p^{\t\cir}) \iso \Z_p\pow{q-1}[u^{\pm 1}] \\
    &\pi_*(\ku_p^{\t\Cp}) \iso \Z_p[\zeta_p][u^{\pm 1}]
    &&\pi_*((\ku_p^{\t\Cp})^{\h\cir}) \iso \Z_p\pow{q'-1}[u^{\pm 1}] \\
    &\pi_*(\tau_{\ge0}(\ku_p^{\t\Cp})) \iso \Z_p[\zeta_p][u] &&\pi_*((\tau_{\ge0}(\ku_p^{\t\Cp}))^{\h\cir}) \iso \Z_p\pow{q'-1}[u,t]/(ut - [p]_{q'})          
  \end{align*}
  Moreover, the map
  \[
    \pi_*(\phi_{\ku_p^\triv}^0) : \pi_*(\ku_p) \to \pi_*(\tau_{\ge0}(\ku_p^{\t\Cp}))
  \]
  sends $\beta \mapsto (\zeta_p-1)u$ and the map
  \[
    \smash{\pi_*((\phi_{\ku_p^\triv}^0)^{\h\cir}) : \pi_*(\ku_p^{\h\cir}) \to \pi_*((\tau_{\ge0}(\ku_p^{\t\Cp}))^{\h\cir})}
  \]
  sends $t \mapsto t$, $q \mapsto (q')^p$, and $\beta \mapsto (q'-1)u$.
\end{proposition}

\begin{proof}
  This is partially a repetition of and partially an elaboration on \cref{t--tCp--oriented}; it follows from the same facts cited in the proof there.
\end{proof}

\begin{proposition}
  \label{j--ku--shifted-map}
  The map
  \[
    \tau_{\ge2}(\ku_p) \otimes_{\ku_p} \tau_{\ge0}(\ku_p^{\t\Cp}) \to \tau_{\ge2}(\tau_{\ge2}(\ku_p)^{\t\Cp})
  \]
  induced by the linear structure of the map $\smash{\phi^2_{\tau_{\ge2}(\ku_p)^\triv} : \tau_{\ge2}(\ku_p) \to \tau_{\ge2}(\tau_{\ge2}(\ku_p)^{\t\Cp})}$ over the map $\smash{\phi^0_{\ku_p^\triv} : \ku_p \to \smash{\tau_{\ge0}(\ku_p^{\t\Cp})}}$ (see \cref{j--cyc--filtered-phi-monoidality}) is an equivalence.
\end{proposition}

\begin{proof}
  The map $\pi_*(\ku_p^{\t\Cp}) \to \pi_*(\Z_p^{\t\Cp})$ induced by the truncation map $\ku_p \to \Z_p$ identifies with the map $\Z_p[\zeta_p][t^{\pm1}] \to \F_p[t^{\pm1}]$ sending $t \mapsto t$ and $\zeta_p \mapsto 1$. From this and the fiber sequence $\tau_{\ge2}(\ku_p) \to \ku_p \to \Z_p$, we deduce that the map $\smash{\pi_*(\tau_{\ge2}(\tau_{\ge2}(\ku_p)^{\t\Cp})) \to \pi_*(\tau_{\ge0}(\ku_p^{\t\Cp}))}$ induced by the map $\tau_{\ge2}(\ku_p) \to \ku_p$ is an injection, with image in degree $2n$ given as follows: it is $0$ for $n=0$, and it is $(\zeta_p-1)u^n \cdot \Z_p[\zeta_p]$ for $n \ge 1$. The claim follows from putting this together with the fact, from \cref{j--ku--homotopy}, that the map $\pi_*(\ku_p) \to \pi_*(\tau_{\ge0}(\ku_p^{\t\Cp}))$ induced by $\smash{\phi^0_{\ku_p^\triv}}$ identifies with the map $\Z_p[\beta] \to \Z_p[\zeta_p][u]$ sending $\beta \mapsto (\zeta_p-1)u$.
\end{proof}

\begin{notation}
  \label{j--ku--spherical-pow}
  Let $R$ be a $p$-complete $\E_\infty$-ring. For any symbol $x$, we let $R[x]$ denote the monoid $\E_\infty$-ring $R[\N]$ with identification $\pi_0(R[x]) \iso \pi_0(R)[x]$, and we let $R\pow{x-1}$ denote the $(p,x-1)$-completion of $R[x]$. We note that there is a natural identification of $\E_\infty$-$R$-algebras
  \[
    R\pow{x-1} \iso \lim_k R[\Z/p^k]
  \]
  (e.g. by \cite[Proposition 2.2.12]{lurie--elliptic-ii}); this induces an $R$-linear action of $\Z_p^\times$ on $R\pow{x-1}$, with the induced action of $u \in \Z_p^\times$ on $\pi_0(R\pow{x-1}) \iso \pi_0(R)\pow{x-1}$ sending $x \mapsto x^u$.

  Below, we regard $R[q']$ as an $\E_\infty$-algebra over $R[q]$ via the map of $\E_\infty$-rings $R[\N] \to R[\N]$ induced by the map of monoids $\N \to \N$ given by multiplication by $p$, and we similarly regard $R\pow{q'-1}$ as an $\E_\infty$-algebra over $R\pow{q-1}$, compatible with $\Z_p^\times$-actions.
\end{notation}

\begin{proposition}[Lurie]
  \label{j--ku--KU-orientation}
  There is an equivalence of $\E_\infty$-rings with $\Z_p^\times$-action
  \[
    \smash{\KU_p^{\B\cir} \iso \KU_p\pow{q-1};}
  \]
  here the $\Z_p^\times$-action on the left is induced by that on $\KU_p$, and the $\Z_p^\times$-action on the right is the diagonal action determined by the action on $\KU_p$ and the $\KU_p$-linear action discussed in \cref{j--ku--spherical-pow}. Moreover, under this equivalence, the $\Z_p^\times$-equivariant map $\smash{\KU_p^{\B\cir} \to \KU_p^{\B\cir}}$ induced by the $p$-fold cover map $\cir \to \cir$ identifies with the map $\KU_p\pow{q-1} \to \KU_p\pow{q'-1}$.
\end{proposition}

\begin{proof}
  This follows from Lurie's reinterpretation of Snaith's theorem, saying that $\KU_p$ is the orientation classifier of the formal multiplicative group over $\S_p$ \cite[\textsection 6.5]{lurie--elliptic-ii}. (We note that the comparison map here arises from the fact that the Bott class $\beta \in \pi_2(\KU_p)$ is a ``strict unit'' class, i.e. lifts to a class in $\pi_2(\mathbb{G}_{\mathrm{m}}(\KU_p))$.)
\end{proof}

\begin{remark}
  \label{j--ku--KU-orientation-alt}
  The equivalence of \cref{j--ku--KU-orientation} is related via bar-cobar duality to an equivalence of $\E_\infty$-$\KU_p$-algebras $\smash{\KU_p^{\cir} \iso \KU_p[\cir]}$.
\end{remark}

\begin{construction}
  \label{j--ku--orientation}
  We construct a commutative diagram of $\E_\infty$-rings with $\Z_p^\times \times \cir$-action
  \[
    \begin{tikzcd}
      \S_p\pow{q-1} \ar[r] \ar[d, "\sigma", swap] &
      \S_p\pow{q'-1} \ar[d, "\sigma'"] \\
      \ku_p \ar[r, "\phi^0_{\ku_p^\triv}"] &
      \tau_{\ge0}(\ku_p^{\t\Cp})
    \end{tikzcd}
  \]
  in which the upper horizontal map is equipped with trivial $\cir$-equivariant structure.

  The composition
  \[
    \S_p\pow{q-1} \to \KU_p\pow{q-1} \iso \KU_p^{\B\cir} \iso \KU_p^{\h\cir},
  \]
  where the first equivalence is that of \cref{j--ku--KU-orientation}, corresponds to a map of $\E_\infty$-rings with $\Z_p^\times \times \cir$-action $\S_p\pow{q-1} \to \KU_p$, where the source is equipped with the trivial $\cir$-action. As the source is connective, this is furthermore equivalent to a map of $\E_\infty$-rings with $\Z_p^\times \times \cir$-action $\sigma : \S_p\pow{q-1} \to \ku_p$. The remainder of the commutative diagram can be constructed similarly, using the definition of $\phi^0_{\ku_p^\triv}$ in terms of the composition
  \[
    \smash{\ku_p \to \ku_p^{\B\Cp} \iso \ku_p^{\h\Cp} \lblto{\can} \ku_p^{\t\Cp},}
  \]
  the identification of the map $\smash{\ku_p^{\h\cir} \to (\ku_p^{\B\Cp})^{\h\cir}}$ with the map $\smash{\ku_p^{\B\cir} \to \ku_p^{\B\cir}}$ induced by the $p$-fold cover $\cir \to \cir$, and the identification from \cref{j--ku--KU-orientation} of the map $\smash{\KU_p^{\B\cir} \to \KU_p^{\B\cir}}$ induced by the $p$-fold cover $\cir \to \cir$ with the map $\KU_p\pow{q-1} \to \KU_p\pow{q'-1}$.
\end{construction}

\begin{notation}
  \label{j--ku--can-factorization}
  Let $\smash{\phi'_{\ku_p^\triv}} : \ku_p \otimes_{\S_p\pow{q-1}} \S_p\pow{q'-1} \to \tau_{\ge 0}(\ku^{\t\Cp})$ denote the map of $\E_\infty$-rings with $\Z_p^\times \times \cir$-action induced by the commutative diagram of \cref{j--ku--orientation} and let $\smash{K'_{\ku_p^\triv}}$ denote the cofiber of $\smash{\phi'_{\ku_p^\triv}}$.
\end{notation}

\begin{proposition}
  \label{j--ku--nil}
  The $\cir$-equivariant $\ku_p$-module $\smash{K'_{\ku_p^\triv}/(p,\beta)}$ is nilpotent.
\end{proposition}

\begin{proof}
  We will use the nilpotence criterion of \cref{j--nil--or}. By \cref{j--ku--homotopy}, the map obtained by applying $\smash{\pi_*((-)^{\h\cir})}$ to $\phi'_{\ku_p^\triv}$ identifies with the map of 
  graded $\Z_p\pow{q'-1}[t]$-algebras
  \[
    \Z_p\pow{q'-1}[\beta,t]/(\beta t - (q-1)) \to
    \Z_p\pow{q'-1}[u,t]/(ut - [p]_{q'})
  \]
  (where $q = (q')^p$) sending $\beta \mapsto (q'-1)u$. It follows that the map obtained by applying $\smash{(-)^{\h\cir}}$ to $\phi'_{\ku_p^\triv}/(p,\beta) \iso \phi'_{\ku_p^\triv} \otimes_{\ku_p} \F_p$ identifies with the map of $\F_p$-algebras
  \[
    \F_p[q',t]/(q-1) \to \F_p[q',u,t]/((q'-1)u, ut-[p]_{q'}).
  \]
  Calculating the cofiber of this map, we identify $\smash{(K'_{\ku_p^\triv}/(p,\beta))^{\h\cir}}$ with the $\F_p$-module
  \[
    \{u,u^2,\ldots\} \cdot \F_p[q']/(q'-1),
  \]
  with action of $t$ given by zero due to the relation $ut =  [p]_{q'} \equiv (q'-1)^{p-1} \pmod p$. Thus, the $\F_p^{\B\cir}$-module structure on $\smash{(K'_{\ku_p^\triv}/(p,\beta))^{\h\cir}}$ factors through the map $\F_p^{\B\cir} \to \F_p$ sending $t \mapsto 0$, proving in particular the desired nilpotence.
\end{proof}


\subsection{On $\j_p^\triv$}
\label{j--j}

In this subsection, we prove the result we are after concerning the map $\phi^0 : \j_p^\triv \to \sh(\j_p^\triv)$, \cref{i--nil--thm}. It will be technically convenient to perform our calculations with the following variant of $\j_p$ and then deduce the result for $\j_p$.\footnote{The convenience of this object (for different but closely related purposes) was indicated to us by Lurie.}

\begin{notation}
  \label{j--j--f}
  We define an $\E_\infty$-ring $\j_{p,0} := \tau_{\ge0}(\KU_p^{\h\Gamma_0})$. It has homotopy groups given by
  \[
    \pi_*(\j_{p,0}) \iso
    \begin{cases}
      \Z_p & \text{if}\ *=0 \\
      \Z_p/(pn)\Z_p & \text{if}\ * = 2n-1\ (n \in \Z_{\ge 1}) \\
      0 & \text{otherwise},
    \end{cases}
  \]
  and moreover the boundary map $\pi_2(\KU_p) \to \pi_1(\j_{p,0})$ is the canonical surjection $\Z_p \to \Z_p/p\Z_p$. Noting that the canonical map $\pi_2(\j_{p,0}/p) \to \pi_2(\KU_p/p)$ is an isomorphism, we let $\smash{\o\beta \in \pi_2(\j_{p,0}/p)}$ denote the unique preimage under this map of the reduction of the Bott class $\beta$. The class $\o\beta$ is classified by a map of $\j_{p,0}$-modules $\j_{p,0}/p[2] \to \j_{p,0}/p$, which we will also denote by $\o\beta$. For $n$ a positive integer, we set
  \[
    \j_{p,0}/(p,\o\beta^n) := \cofib(\o\beta^n : \j_{p,0}/p[2n] \to \j_{p,0}/p),
  \]
  and for $X$ a $\j_{p,0}$-module, we set $X/(p,\o\beta^n) := X \otimes_{\j_{p,0}} \j_{p,0}/(p,\o\beta^n)$.
\end{notation}

\begin{lemma}
  \label{j--j--modpbeta}
  For $X$ a $\j_{p,0}$-module, the spectrum $X/(p,\o\beta)$ naturally admits an $\F_p$-module structure.
\end{lemma}

\begin{proof}
  To begin, we observe that
  \[
    \pi_*(\j_{p,0}/(p,\o\beta)) \iso
    \begin{cases}
      \F_p & \text{if}\ * \in \{0,1\} \\
      0 & \text{otherwise}.
    \end{cases}
  \]
  We deduce that the quotient map $\j_{p,0} \to \j_{p,0}/(p,\o\beta)$ factors canonically through an equivalence
  \[
    \tau_{\le1}(\tau_{\le1}(\j_{p,0})/p) \isoto \j_{p,0}/(p,\o\beta)
  \]
  (here we have used that the action of $p$ on $\j_{p,0}/p$, and hence on $\j_{p,0}/(p,\o\beta)$, is zero, since $p$ is odd).

  Now, $\tau_{\le1}(\j_{p,0})$ is an $\E_\infty$-$\j_{p,0}$-algebra and receives a map of $\E_\infty$-rings from $\tau_{\le1}(\j_p) \iso \Z_p$. We may thus identify $\tau_{\le1}(\j_{p,0})/p$ with the $\E_\infty$-ring $\tau_{\le1}(\j_{p,0}) \otimes_{\Z_p} \F_p$. In combination with the previous paragraph, this produces an $\E_\infty$-$\j_{p,0}$-algebra structure on  $\j_{p,0}/(p,\o\beta)$ together with a map of $\E_\infty$-rings $\F_p \to \j_{p,0}/(p,\o\beta)$. Then $X/(p,\o\beta) \iso X \otimes_{\j_{p,0}} \j_{p,0}/(p,\o\beta)$ is naturally a $\j_{p,0}/(p,\o\beta)$-module, which restricts to an $\F_p$-module.
\end{proof}

\begin{notation}
  \label{j--j--adams}
  As in \cref{t--tCp--psi-ntn}, we will denote the action of an element $g \in \Z_p^\times$ on $\ku_p$ by $\psi^g$; we will denote its action on $\S_p\pow{q-1}$ and $\S_p\pow{q'-1}$ by $\Psi^g$.
\end{notation}

\begin{construction}
  \label{j--j--diagram}
  We construct a commutative diagram of $\cir$-equivariant $\j_{p,0}$-modules
  \[
    \begin{tikzcd}[column sep=large]
      \j_{p,0} \ar[rr, "\phi^0_{\j_{p,0}^\triv}"] \ar[d] &
      &
      \tau_{\ge 0}(\j_{p,0}^{\t\Cp}) \ar[d] \\
      \ku_p \ar[r] \ar[d, "\psi^{1+p}_\circ"] &
      \ku_p \otimes_{\S_p\pow{q-1}} \S_p\pow{q'-1} \ar[r, "\phi'_{\ku_p^\triv}"] \ar[d, "(\psi^{1+p} \otimes \Psi^{1+p})_\circ"] &
      \tau_{\ge0}(\ku_p^{\t\Cp}) \ar[d, "\psi^{1+p}_{\circ\circ}"] \\
      \tau_{\ge2}(\ku_p) \ar[r] &
      \tau_{\ge2}(\ku_p) \otimes_{\S_p\pow{q-1}} \S_p\pow{q'-1} \ar[r, "\phi'_{\tau_{\ge2}(\ku_p)^\triv}"] &
      \tau_{\ge2}(\tau_{\ge2}(\ku_p)^{\t\Cp})
    \end{tikzcd}
  \]
  in which the left and right columns are fiber sequences. (The top rectangle is in fact a pushout of $\E_\infty$-rings, but we will not prove this here.)

  The map $\psi^{1+p}_\circ$ is the unique map lifting the endomorphism $\psi^{1+p}-\id$ of $\ku_p$ along the canonical map $\tau_{\ge2}(\ku_p) \to \ku_p$. Both existence and uniqueness follow from the fiber sequence $\tau_{\ge2}(\ku_p) \to \ku_p \to \Z_p$, and the fact that $\psi^{1+p}$ induces the identity map on $\pi_0(\ku_p) \iso \Z_p$. The maps $\smash{(\psi^{1+p} \otimes \Psi^{1+p})_\circ}$ and $\smash{\psi^{1+p}_{\circ\circ}}$ are defined similarly, the former arising from the endomorphism $\psi^{1+p} \otimes \Psi^{1+p} - \id$ of $\ku_p \otimes_{\S_p\pow{q-1}} \S_p\pow{q'-1}$ and the latter arising from the map $\tau_{\ge0}(\ku_p^{\t\Cp}) \to \tau_{\ge0}(\tau_{\ge 2}(\ku_p)^{\t\Cp})$ induced by $\psi^{1+p}_\circ$.

  The map $\smash{\phi'_{\ku_p^\triv}}$ is as defined in 
  \cref{j--ku--can-factorization} and the map $\smash{\phi'_{\tau_{\ge2}(\ku_p)^\triv}}$ is defined similarly, i.e. by using the map $\smash{\phi^2_{\tau_{\ge2}(\ku_p)^\triv} : \tau_{\ge2}(\ku_p) \to \tau_{\ge2}(\tau_{\ge2}(\ku_p)^{\t\Cp})}$ and its linearity over $\phi^0_{\ku_p^\triv}$ (see \cref{j--cyc}).
\end{construction}

\begin{remark}
  \label{j--j--diagram-addendum}
  Concerning the maps $\smash{\phi'_{\ku_p^\triv}}$ and $\smash{\phi'_{\tau_{\ge2}(\ku_p)^\triv}}$ in \cref{j--j--diagram}: it follows from their definitions and \cref{j--ku--shifted-map} that $\smash{\phi'_{\tau_{\ge2}(\ku_p)^\triv}}$ may be identified with the two-fold suspension of $\phi'_{\ku_p^\triv}$.
\end{remark}

\begin{lemma}
  \label{j--j--nil-lemma}
  Let $K$ be the total cofiber of the commutative square of $\cir$-equivariant $\j_{p,0}$-modules
  \[
    \begin{tikzcd}
      \ku_p \ar[r] \ar[d, "\psi^{1+p}_\circ"] &
      \ku_p \otimes_{\S_p\pow{q-1}} \S_p\pow{q'-1} \ar[d, "(\psi^{1+p} \otimes \Psi^{1+p})_\circ"] \\
      \tau_{\ge2}(\ku_p) \ar[r] &
      \tau_{\ge2}(\ku_p) \otimes_{\S_p\pow{q-1}} \S_p\pow{q'-1}
    \end{tikzcd}
  \]
  contained in the commutative diagram of \cref{j--j--diagram}. Then $K/(p,\o\beta)$ is nilpotent.
\end{lemma}

\begin{proof}
  By \cref{j--j--modpbeta}, $K/(p,\o\beta)$ admits an $\cir$-equivariant $\F_p$-module structure. We may thus use the nilpotence criterion of \cref{j--nil--or}.

  By \cref{j--ku--homotopy}, we have an identification
  \[
    \pi_*(\ku_p^{\h\cir}) \iso \Z_p\pow{q-1}[\beta,t]/(\beta t - (q-1))
  \]
  where $t$ denotes the standard complex orientation; $\pi_*((\ku_p \otimes_{\S_p\pow{q-1}} \S_p\pow{q'-1})^{\h\cir})$ then identifies with the base change of this along $\Z_p\pow{q-1} \to \Z_p\pow{q'-1}$. The canonical map
  \[
    \smash{\pi_*(\tau_{\ge2}(\ku_p)^{\B\cir}) \to \pi_*(\ku_p^{\B\cir})}
  \]
  is an injection, with image in degree $2n$ given as follows: for $n \le 0$, it is $(q-1) t^n \cdot \Z_p\pow{q-1}$, and for $n \ge 1$, it is $\beta^n \cdot \Z_p\pow{q-1}$; the map
  \[
    \smash{\pi_*((\tau_{\ge2}(\ku_p) \otimes_{\S_p\pow{q-1}} \S_p\pow{q'-1})^{\B\cir}) \to \pi_*((\ku_p \otimes_{\S_p\pow{q-1}} \S_p\pow{q'-1})^{\B\cir})}
  \]
  has the same description, but tensored up along $\Z_p\pow{q-1} \to \Z_p\pow{q'-1}$.

  In terms of these identifications, applying $\smash{(-)^{\h\cir}}$ to the square in the statement and reducing modulo $(p,\o\beta)$ results in the commutative square of $\F_p$-modules
  \[
    \begin{tikzcd}
      \F_p[t] \ar[r] \ar[d, "\psi^{1+p}_\circ"] &
      \F_p[q',t]/(q-1) \ar[d, "(\psi^{1+p} \otimes \Psi^{1+p})_\circ"] \\
      \beta \cdot \F_p \oplus \left(\{t, t^2, \ldots\} \cdot \frac{(q-1)\F_p[q]}{(q-1)^2}\right) \ar[r] &
      \beta \cdot \F_p[q']/(q-1) \oplus \left(\{t, t^2, \ldots\} \cdot \frac{(q-1)\F_p[q']}{(q-1)^2}\right).
    \end{tikzcd}
  \]
  Let us calculate the behavior of the vertical maps. On $\smash{\pi_*(\ku_p^{\h\cir})}$, we have
  \[
    \psi^{1+p}(\beta) = (1+p)\beta \equiv \beta \pmod p, \qquad
    \psi^{1+p}(q) = q^{1+p}.
  \]
  Using the relation $\beta t = q-1$, we deduce that
  \[
    \psi^{1+p}(t) \equiv \frac{q^{1+p}-1}{q-1} t \equiv \left(\frac{q^{1+p}-q}{q-1}+1\right)t \equiv (q(q-1)^{p-1}+1)t \equiv t \pmod{p,(q-1)^2}
  \]
  (since $p-1 \ge 2$, as $p$ is odd). Combining this with the formula $\Psi^{1+p}(q') = (q')^{1+p} = q'q$, we find
  \begin{align*}
    (\psi^{1+p} \otimes \Psi^{1+p})_\circ(t^n(q')^m)
    &= \psi^{1+p}(t^n)\Psi^{1+p}((q')^m) - t^n(q')^m \\
    &\equiv t^n(q')^mq^m - t^n(q')^m \pmod{p,(q-1)^2} \\
    &\equiv t^n(q')^m(q^m-1) \pmod{p,(q-1)^2}.
  \end{align*}
  Note that $q^m-1$ is nonzero in $\frac{(q-1)\F_p[q']}{(q-1)^2}$ for $1 \le m \le p-1$. We may now compute the total cofiber of the above square of $\F_p$-modules, taking horizontal and then vertical cofibers; we obtain
  \[
    \smash{\{\beta q', \ldots, \beta (q')^{p-1}\} \cdot \F_p,}
  \]
  concentrated in a single degree, on which $t$ evidently acts by zero.
\end{proof}

\begin{proposition}
  \label{j--j--nil-result}
  The $\cir$-equivariant $\j_{p,0}$-module $\smash{K_{\j_{p,0}^\triv}/(p,\o\beta)}$ (\cref{j--cyc--filtered-phi}) is nilpotent.
\end{proposition}

\begin{proof}
  Using the commutative diagram of \cref{j--j--diagram}, the claim follows from combining \cref{j--ku--nil}, \cref{j--j--diagram-addendum}, and \cref{j--j--nil-lemma}.
\end{proof}

\begin{proof}[Proof of \cref{i--nil--thm}]
  In $\pi_*(\j_{p,0})$ we have $v_1 = (\o{\beta})^{p-1}$. It thus follows from \cref{j--j--nil-result} and a devissage that $K_{\j_{p,0}^\triv}/(p,v_1)$ is nilpotent. The claim follows from this and the fact that
  \[
    \smash{K_{\j_p^\triv} \iso K_{\j_{p,0}^\triv}^{\h\F_p^\times}},
  \]
  which is a retract of $\smash{K_{\j_{p,0}^\triv}}$ (here we have used that $p-1 = |\F_p^\times|$ acts invertibly on $\j_{p,0}$, $\smash{\tau_{\ge0}(\j_{p,0}^{\t\Cp})}$, and $K_{\j_{p,0}^\triv}$, as they are $p$-complete).
\end{proof}

\section{$\K(1)$-local $\TC$ and K-theory}
\label{k1k}

In this section, we discuss some implications of \cref{i--main--thhZp,i--nil--thm} on the behavior of topological cyclic homology and algebraic K-theory after $\K(1)$-localization. As described in \cref{i--k1k}, these implications comprise a height $1$ analogue of the height $0$ results of Antieau--Mathew--Morrow--Nikolaus \cite{ammn--beilinson}; the general results and arguments in \cref{k1k--cyc,k1k--ring} are straightforward adaptations of theirs. In \cref{k1k--eg}, we spell out some specific examples of these results, obtaining new calculations in $\K(1)$-local $\TC$ and K-theory.


\subsection{$\TC$ of cyclotomic spectra}
\label{k1k--cyc}

We first formulate a result for the topological cyclic homology of general bounded below cyclotomic spectra.

\begin{notation}
  \label{k1k--cyc--map}
  Recall from \cref{t--pf--map} that there is a unique map of cyclotomic $\E_\infty$-rings $\alpha : \j_p^\triv \to \pcpl{\THH(\Z)}$.
\end{notation}

\begin{lemma}
  \label{k1k--cyc--fib-tS1}
  For $X$ a bounded below $\cir$-equivariant spectrum, $\smash{\pcpl{((X \otimes \fib(\alpha))^{\t\cir})}} \iso 0$.
\end{lemma}

\begin{proof}
  Combine \cref{i--main--thhZp,i--nil--thm,j--nil--mod-p-v1}.
\end{proof}

\begin{theorem}
  \label{k1k--cyc--main-sq}
  For $X$ a bounded below cyclotomic spectrum, there is a natural map of spectra $\pcpl{\TC(X \otimes \pcpl{\THH(\Z)})} \to \pcpl{((X \otimes \j_p^\triv)^{\t\cir})}$ making the square
  \begin{equation}
    \label{k1k--cyc--TC-tS1-map--diagram}
    \begin{tikzcd}
      \pcpl{\TC(X \otimes \j_p^\triv)} \ar[r] \ar[d] &
      \pcpl{\TC(X \otimes \pcpl{\THH(\Z)})} \ar[d] \\
      \pcpl{((X \otimes \j_p^\triv)^{\h\cir})} \ar[r] &
      \pcpl{((X \otimes \j_p^\triv)^{\t\cir})}
    \end{tikzcd}
  \end{equation}
  commute, the upper horizontal map being induced by $\alpha$ and the left vertical and lower horizontal maps being the canonical ones. Moreover, upon $\K(1)$-localization, this square becomes cartesian and the map $\pcpl{\TC(X)} \to \pcpl{\TC(X \otimes \j_p^\triv)}$ induced by the unit map $\S \to \j_p^\triv$ becomes an equivalence.
\end{theorem}

\begin{proof}
  The map and commutative square are obtained immediately from the commutative diagram
  \begin{equation}
    \label{k1k--cyc--main-sq--construction-diagram}
    \begin{tikzcd}
      \pcpl{\TC(X \otimes \j_p^\triv)} \ar[r, "\alpha"] \ar[d] &
      \pcpl{\TC(X \otimes \pcpl{\THH(\Z)})} \ar[d] \\
      \pcpl{((X \otimes \j_p^\triv)^{\h\cir})} \ar[r, "\alpha"] \ar[d, "\can", swap] &
      \pcpl{((X \otimes \pcpl{\THH(\Z)})^{\h\cir})} \ar[d, "\can"] \\
      \pcpl{((X \otimes \j_p^\triv)^{\t\cir})} \ar[r, "\alpha"] &
      \pcpl{((X \otimes \pcpl{\THH(\Z)})^{\t\cir})}
    \end{tikzcd}
  \end{equation}
  and the fact that the lowest horizontal map is an equivalence (\cref{k1k--cyc--fib-tS1}).

  The cartesianness claim follows from two observations about the diagram \cref{k1k--cyc--main-sq--construction-diagram}:
  \begin{itemize}[leftmargin=*]
  \item The upper square is cartesian. This follows from considering the variant of \cref{k1k--cyc--main-sq--construction-diagram} in which the $\can$ maps are replaced by $\can - \phi$; there the columns are fiber sequences, and the lowest horizontal map remains an equivalence.
  \item The map $\can : (X \otimes \pcpl{\THH(\Z)})^{\h\cir} \to (X \otimes \pcpl{\THH(\Z)})^{\t\cir}$ is a $\K(1)$-local equivalence. Indeed, the fiber identifies with a shift of $(X \otimes \pcpl{\THH(\Z)})_{\h\cir}$, which is a colimit of $\Z$-modules, hence vanishes $\K(1)$-locally.
  \end{itemize}

  That the map $\pcpl{\TC(X)} \to \pcpl{\TC(X \otimes \j_p^\triv)}$ is a $\K(1)$-local equivalence follows from the canonical map $\pcpl{\TC(X)} \otimes \j_p \to \pcpl{\TC(X \otimes \j_p^\triv)}$ being a $p$-complete equivalence \cite[Remark 2.4]{ammn--beilinson} and the unit map $\S \to \j_p$ being a $\K(1)$-local equivalence.
\end{proof}

\begin{corollary}
  \label{k1k--cyc--main-seq}
  For $X$ a bounded below cyclotomic spectrum, there is a natural fiber sequence of spectra
  \[
    \L_{\K(1)}(X_{\h\cir})[1] \to \L_{\K(1)}\TC(X) \to \L_{\K(1)}\TC(X \otimes \pcpl{\THH(\Z)})
  \]
  in which the second map is induced by the unit of $\pcpl{\THH(\Z)}$.
\end{corollary}


\subsection{$\TC$ and K-theory of ring spectra}
\label{k1k--ring}

We now apply \cref{k1k--cyc--main-sq} to the study of  $\K(1)$-localized $\TC$ and K-theory of ring spectra. We begin with the following result for $\j_p$-algebras, parallel to the result \cite[Theorem 2.12]{ammn--beilinson} for $\Z$-algebras.

\begin{theorem}
  \label{k1k--ring--jmod}
  For $R$ a connective $\E_1$-$\j_p$-algebra, there is a natural commutative square of spectra
  \[
    \begin{tikzcd}
      \pcpl{\TC(R)} \ar[r] \ar[d] &
      \pcpl{\TC(R \otimes_\S \Z)} \ar[d] \\
      \pcpl{\TC^-(R/\j_p)} \ar[r] &
      \pcpl{\TP(R/\j_p)}
    \end{tikzcd}
  \]
  in which all maps but the right vertical one are the canonical ones. Moreover, this square becomes cartesian upon $\K(1)$-localization.
\end{theorem}

\begin{proof}
  We combine \cref{k1k--cyc--main-sq} (for $X = \THH(R)$) with the commutative square
  \[
    \begin{tikzcd}
      \smash{\pcpl{((\THH(R) \otimes \j_p^\triv)^{\h\cir})}} \ar[r, "\can"] \ar[d] &
      \smash{\pcpl{((\THH(R) \otimes \j_p^\triv)^{\t\cir})}} \ar[d] \\
      \pcpl{\TC^-(R/\j_p)} \ar[r, "\can"] &
      \pcpl{\TP(R/\j_p)}
    \end{tikzcd}    
  \]
  induced by the canonical $\cir$-equivariant map $\THH(R) \otimes \j_p^\triv \to \THH(R/\j_p)$, which becomes cartesian after $\K(1)$-localization. This last claim follows from considering the map induced on horizontal fibers, as the aforementioned $\cir$-equivariant map is a $\K(1)$-local equivalence, and hence so too is the map obtained from this by applying $(-)_{\h\cir}$.
\end{proof}

Next, we prove our result for general connective $\E_1$-rings, \cref{i--k1k--main}. The proof will rely on the following prior result on the $\K(1)$-local K-theory of $\Z$-algebras.

\begin{theorem}[Bhatt--Clausen--Mathew \citebare{bhatt-clausen-mathew--k1k}; Land--Meier--Mathew--Tamme \citebare{lmmt--purity}]
  \label{k1k--ring--truncating}
  For $A$ a connective $\E_1$-$\Z$-algebra, the following canonical maps are equivalences:
  \[
    \L_{\K(1)}\K(A) \to \L_{\K(1)}\K(\pi_0(A)) \to \L_{\K(1)}\K(\pi_0(A)[1/p]).
  \]
\end{theorem}

\begin{proof}
  See \cite[Theorem 1.1]{bhatt-clausen-mathew--k1k} and \cite[Corollary 4.23]{lmmt--purity}.
\end{proof}

\begin{remark}
  \label{k1k--ring--acyclic}
  In fact, \cite[Corollary 4.23]{lmmt--purity} is more general than \cref{k1k--ring--truncating}: there the statement is proved for $A$ any connective, $\K(1)$-acyclic $\E_1$-ring. As alluded to in \cref{i--k1k}, this more general statement can alternatively be deduced from \cref{i--k1k--main} (under our assumption that $p$ is odd), but we are certainly using the statement for $\E_1$-$\Z$-algebras as an input to our proof of the latter.
\end{remark}

\begin{lemma}
  \label{k1k--ring--K1hS1tS1}
  For $X$ an $\cir$-equivariant spectrum, there is a natural fiber square of spectra
  \[
    \begin{tikzcd}
      \L_{\K(1)}((X \otimes \j_p^\triv)^{\h\cir}) \ar[r] \ar[d] &
      \L_{\K(1)}((X \otimes \j_p^\triv)^{\t\cir}) \ar[d] \\
      (\L_{\K(1)}X)^{\h\cir} \ar[r] &
      \pcpl{((\L_{\K(1)}X)^{\t\cir})}
    \end{tikzcd}
  \]
  in which the horizontal maps are the canonical ones.
\end{lemma}

\begin{proof}
  The square is induced by the localization map $X \otimes \j_p^{\triv} \to \L_{\K(1)}(X \otimes \j_p^\triv) \iso \L_{\K(1)}X$, and it is cartesian because the map on horizontal fibers is an equivalence.
\end{proof}

\begin{proof}[Proof of \cref{i--k1k--main}]
  The claim follows from considering the commutative diagram
  \[
    \fontsize{8.5}{12}\selectfont
    \begin{tikzcd}[column sep=small,row sep=3.5ex]
      \L_{\K(1)}\K(R) \ar[r] \ar[d] &
      \L_{\K(1)}\K(R \otimes_\S \Z) \ar[r] \ar[d] &
      \L_{\K(1)}\K(\pi_0(R)) \ar[r] \ar[d] &
      \L_{\K(1)}\K(\pi_0(R)[1/p]) \\
      \L_{\K(1)}\TC(R) \ar[r] \ar[d] &
      \L_{\K(1)}\TC(R \otimes_\S \Z) \ar[r] \ar[d] &
      \L_{\K(1)}\TC(\pi_0(R)) \\
      \L_{\K(1)}((\THH(R) \otimes \j_p^\triv)^{\h\cir}) \ar[r] \ar[d] &
      \L_{\K(1)}((\THH(R) \otimes \j_p^\triv)^{\t\cir}) \ar[d] \\
      \pcpl{\TC^-(\L_{\K(1)}R)} \ar[r] &
      \pcpl{\TP(\L_{\K(1)}R)}
    \end{tikzcd}
  \]
  which can be described as follows:
  \begin{itemize}[leftmargin=*]
  \item In the first two rows, the horizontal maps are induced by the canonical maps of $\E_1$-rings (noting that $\pi_0(R \otimes_\S \Z) \iso \pi_0(R)$) and the vertical maps are given by the cyclotomic trace. By the Dundas--Goodwillie--McCarthy theorem \cite{dundas-goodwillie-mccarthy}, the two squares formed by these rows are cartesian. By \cref{k1k--ring--truncating}, the second and third maps in the first row are equivalences, and by cartesianness it follows that the second map in the second row is also an equivalence.
  \item The square formed by the second and third rows is the fiber square of \cref{k1k--cyc--main-sq} (applied to $X = \THH(R)$).
  \item The square formed by the third and fourth rows is the fiber square of \cref{k1k--ring--K1hS1tS1} (applied to $X = \THH(R)$), noting that we have a natural equivalence $\L_{\K(1)}\THH(R) \iso \pcpl{\THH(\L_{\K(1)}R)}$ by virtue of $\K(1)$-localization being $p$-completely smashing (\cref{i--conv}\cref{i--conv--K1}). \qedhere
  \end{itemize}
\end{proof}


\subsection{Examples}
\label{k1k--eg}

The general results above give us tools for accessing the $\K(1)$-local $\TC$ and K-theory of a ring spectrum $R$ even when $R$ is not $\K(1)$-acyclic. Here we demonstrate this with a few specific calculations, which together prove in particular \cref{i--k1k--eg}.

\begin{remark}
  \label{k1k--eg--K-TC}
  For any connective $\E_1$-ring $R$ such that $\pi_0(R)$ is commutative and $p$-complete, the $\K(1)$-localized cyclotomic trace map $\L_{\K(1)}\K(R) \to \L_{\K(1)}\TC(R)$ is an equivalence. Indeed, the Dundas--Goodwillie--McCarthy theorem reduces the claim to the case where $R$ is discrete, in which case it follows from the rigidity theorem of Clausen--Mathew--Morrow \cite{cmm--rigidity}; cf. \cite[Proof of Theorem 2.17]{bhatt-clausen-mathew--k1k}. This applies in particular for $R \in \{\S_p,\j_p,\ku_p,\Z_p\}$, examples to be further discussed below.
\end{remark}

\begin{example}
  \label{k1k--eg--S}
  Applying \cref{i--k1k--main} in the case $R = \S$, we obtain a fiber sequence
  \[
    \pcpl{\J_p[\B\cir]}[1] \to \L_{\K(1)}\TC(\S) \to \L_{\K(1)}\TC(\Z);
  \]
  for the fiber term here, we have used the canonical equivalence
  \[
    \pcpl{\THH(\L_{\K(1)}\S)} = \pcpl{\THH(\J_p)} \iso \J_p^\triv,
  \]
  coming from the fact that $\J_p$ is an idempotent algebra in $\pcpl\Spt$.
  
  This fiber sequence was in fact present implicitly in the discussion of \cref{z}; let us make it completely explicit. Comparing its construction with the proof of \cref{z--ol--k1}, we see that it fits into a commutative diagram
  \[
    \begin{tikzcd}
      \pcpl{\J_p[\B\cir]} \ar[r, "\id-\nu"] \ar[d] &
      \pcpl{\J_p[\B\cir]} \ar[r] \ar[d, "\Nm"] &
      \cofib(\id-\nu : \pcpl{\J_p[\B\cir]} \to \pcpl{\J_p[\B\cir]}) \ar[d, "\Nm"] \\
      0 \ar[r] \ar[d] &
      \fib(\id-[p]^* : \J_p^{\B\cir} \to \J_p^{\B\cir})  \ar[r, "\id"] \ar[d] &
      \fib(\id-[p]^* : \J_p^{\B\cir} \to \J_p^{\B\cir}) \ar[d] \\
      \pcpl{\J_p[\B\cir]}[1] \ar[r] &
      \L_{\K(1)}\TC(\S) \ar[r] &
      \L_{\K(1)}\TC(\Z),
    \end{tikzcd}
  \]
  in which each row and column is a fiber sequence. The results of \cref{z} (namely \cref{z--ol--k1,z--tr--J-description,z--tr--fiber-description}) identify this commutative diagram with the following one:
  \[
    \fontsize{9.5}{12}\selectfont
    \begin{tikzcd}[row sep=large, column sep=large]
      \J_p \oplus \pcpl{\big(\bigoplus_{n \in \Z_{\ge0}} \KU_p\big)} \ar[r, "{\begin{psmallmatrix}0 \amp 0 \\ \theta \amp \id - \sigma\end{psmallmatrix}}"] \ar[d] &
      \J_p \oplus \pcpl{\big(\bigoplus_{n \in \Z_{\ge0}} \KU_p\big)} \ar[r, "{\begin{psmallmatrix}\id \amp 0 \\ \theta \amp (\pi)_{n \in \Z_{\ge0}}\end{psmallmatrix}}"] \ar[d, "{\begin{psmallmatrix}0 \amp 0 \\ 0 \amp (\tau)_{n \in \Z_{\ge0}}\end{psmallmatrix}}"] &
      \J_p \oplus \KU_p/\J_p \ar[d, "{\begin{psmallmatrix}0 \amp 0 \\ 0 \amp \tau'\end{psmallmatrix}}"] \\
      0 \ar[r] \ar[d] &
      \J_p \oplus \J_p[-1] \ar[r, "\id"] \ar[d] &
      \J_p \oplus \J_p[-1] \ar[d] \\
      \J_p[1] \oplus \pcpl{\big(\bigoplus_{n \in \Z_{\ge0}} \KU_p[1]\big)} \ar[r, "{\begin{psmallmatrix}0 \amp 0 \\ 0 \amp 0 \\ 0 \amp \id \\ \theta' \amp 0\end{psmallmatrix}}"] &
      \J_p \oplus \J_p[1] \oplus \pcpl{\big(\bigoplus_{n \in \Z_{\ge0}} \KU_p[1]\big)} \oplus \Y' \ar[r, "{\begin{psmallmatrix}\id \amp 0 \amp 0 \amp 0 \\ 0 \amp \id \amp 0 \amp 0 \\ 0 \amp 0 \amp 0 \amp \pi'\end{psmallmatrix}}"] &
      \J_p \oplus \J_p[1] \oplus \X'.
    \end{tikzcd}
  \]
  Here $\sigma : \bigoplus_{n \in \Z_{\ge0}} \KU_p \to \bigoplus_{n \in \Z_{\ge0}} \KU_p$ is the shift operator and $\theta : \J_p \to \bigoplus_{n \in \Z_{\ge0}}$ is the unit map into the zeroth factor, as in the proof of \cref{z--ol--k1}; $\tau : \KU_p \to \J_p[-1]$ is as in \cref{z--ol--tau} and $\Y' = \cofib(\tau)$ as in \cref{z--tr--fiber-description}; $\tau' : \KU_p/\J_p \to \J_p[-1]$ and $\X' = \cofib(\tau')$ are as in the proof of \cref{z--ol--k1}; and $\pi : \KU_p \to \KU_p/\J_p$ is the tautological map and $\pi' : \Y' \to \X'$ the map induced by $\pi$.

  From the above we may identify the boundary map $\L_{\K(1)}\TC(\Z) \to \pcpl{\J_p[\B\cir]}[2]$ of our fiber sequence: it is given by the composition
  \[
    \L_{\K(1)}\TC(\Z) \to \X' \to \KU_p/\J_p[1] \to \J_p[2] \to \pcpl{\J_p[\B\cir]}[2],
  \]
  the first map being the projection, the second and third being the tautological boundary maps, and the fourth induced by the basepoint of $\B\cir$.
\end{example}

\begin{example}
  \label{k1k--eg--j}
  The unit map $\S \to \j_p$ induces equivalences upon applying $\pcpl{\pi_0(-)}$ and $\L_{\K(1)}(-)$, so it follows from \cref{i--k1k--main} that it also induces an equivalence upon applying $\L_{\K(1)}\TC(-)$.
\end{example}

\begin{example}
  \label{k1k--eg--jzeta}
  Following Levy \cite{levy--K1}, we write $\j_{p,\zeta} := \ku_p^{\h\Gamma}$. As a spectrum, this is equivalent to $\tau_{\ge -1}(\J_p)$. In particular, it is $(-1)$-connective, from which it follows that $\THH(\j_{p,\zeta})$ is $(-1)$-connective (e.g. by considering the skeletal filtration on the cyclic bar construction). We may thus apply \cref{k1k--cyc--main-seq}, and we obtain a fiber sequence
  \[
    \pcpl{\J_p[\B\cir]}[1] \to \L_{\K(1)}\TC(\j_{p,\zeta}) \to \L_{\K(1)}\TC(\j_{p,\zeta} \otimes_\S \Z).
  \]
  We claim that the following maps induce equivalences on $\K(1)$-local $\TC$:
  \[
    \j_{p,\zeta} \otimes_\S \Z \iso (\ku_p \otimes_\S \Z)^{\h\Gamma} \to \Z_p^{\h\Gamma}, \qquad
    \Z_p \to \Z_p^{\h\Gamma};
  \]
  here the first map is induced by the truncation map $\ku_p \otimes_\S \Z \to \Z_p$, and $\Gamma$ acts trivially on $\Z_p$, which gives the second map. Note that this claim implies that the connective cover map $\j_p \to \j_{p,\zeta}$ induces an equivalence on $\K(1)$-local $\TC$, by comparison of the associated fiber sequences. Moreover, we deduce from \cite[Theorem A]{levy--K1} that the maps
  \[
    \L_{\K(1)}\K(\j_p) \to \L_{\K(1)}\K(\j_{p,\zeta}) \to \L_{\K(1)}\K(\J_p)
  \]
  are equivalences.

  Let us justify the claim:
  \begin{enumerate}
  \item For the map $(\ku_p \otimes_\S \Z)^{\h\Gamma} \to \Z_p^{\h\Gamma}$, we apply \cite[Theorem B]{levy--K1}, noting that $\L_{\K(1)}\TC(-)$ is truncating on connective $\E_1$-$\Z$-algebras \cite{bhatt-clausen-mathew--k1k,lmmt--purity,mathew--K1TR}.
  \item For the map $\Z_p \to \Z_p^{\h\Gamma}$, the pullback square
    \[
      \begin{tikzcd}
        \TC(\Z_p^{\h\Gamma}) \ar[r] \ar[d]
        & \TC(\Z_p) \ar[d] \\
        \TC(\Z_p) \ar[r] &
        \TC(\Z_p[x])
      \end{tikzcd}
    \]
    of Land--Tamme \cite[Theorem 4.1]{land-tamme--pushouts} reduces us to checking that the map of rings $\Z_p \to \Z_p[x]$ induces an equivalence on $\K(1)$-local $\TC$. This follows from purity equivalence $\L_{\K(1)}\TC(R) \iso \L_{\K(1)}\K(R[1/p])$ for these rings $R$ \cite{bhatt-clausen-mathew--k1k,lmmt--purity}.
  \end{enumerate}
\end{example}

\begin{example}
  \label{k1k--eg--ku}
  The unit map $\S \to \ku_p$ induces a commutative diagram
  \[
    \begin{tikzcd}
      \pcpl{\J_p[\B\cir]}[1] \ar[r] \ar[d] &
      \L_{\K(1)}\TC(\S) \ar[r] \ar[d] &
      \L_{\K(1)}\TC(\Z) \ar[d, "\sim"] \\
      \pcpl{\KU_p[\B\cir]}[1] \ar[r] &
      \L_{\K(1)}\TC(\ku_p) \ar[r] &
      \L_{\K(1)}\TC(\Z_p),
    \end{tikzcd}
  \]
  the rows being the fiber sequences supplied by \cref{i--k1k--main}. The upper row is as in \cref{k1k--eg--S}, and the fiber term in the lower row comes from the equivalence
  \[
    \pcpl{\THH(\L_{\K(1)}\ku_p)} \iso \pcpl{\THH(\KU_p)} \iso \KU_p^\triv,
  \]
  which holds because $\KU_p$ is a $0$-cotruncated object in $\pcpl{\CAlg}$, i.e. $\Map_{\CAlg}(\KU_p,A)$ is discrete for any $p$-complete $\E_\infty$-ring $A$ (see \cite[Theorem 5.0.2]{lurie--elliptic-ii}).

  The description of the boundary map for the upper row in \cref{k1k--eg--S} shows that the boundary map of the lower row is canonically null, i.e. the lower row canonically splits, giving an equivalence
  \[
    \L_{\K(1)} \TC(\ku_p) \iso \L_{\K(1)}\TC(\Z) \oplus \pcpl{\KU_p[\B\cir]}[1].
  \]
\end{example}

\begin{remark}
  \label{k1k--eg--orientable}
  Let $R$ be a complex orientable, connective $\E_1$-ring such that $\pcpl{\pi_0(R)} \iso \Z_p$ (e.g. $R = \MU$). By the same logic as in \cref{k1k--eg--ku}, combined with \cref{z--tr--k1-orientation}, the fiber sequence
  \[
    \pcpl{(\THH(\L_{\K(1)}R)_{\h\cir})}[1] \to \L_{\K(1)}\TC(R) \to \L_{\K(1)}\TC(\pi_0(R))
  \]
  of \cref{i--k1k--main} admits a splitting.
\end{remark}

\begin{example}
  \label{k1k--eg--KU}
  Consider the $\K(1)$-localized localization fiber sequence
  \[
    \L_{\K(1)}\K(\Z_p) \to \L_{\K(1)}\K(\ku_p) \to \L_{\K(1)}\K(\KU_p)
  \]
  of Blumberg--Mandell \cite{blumberg-mandell--localization}. Here the first map is the ($\K(1)$-localized) transfer map in K-theory, induced by restriction along $\ku_p \to \Z_p$. We claim that this first map is canonically null, giving an equivalence
  \[
    \L_{\K(1)}\K(\KU_p) \iso \L_{\K(1)}\K(\ku_p) \oplus \L_{\K(1)}\K(\Z_p)[1].
  \]
  Indeed, from \cref{k1k--eg--ku} we know that the map $\L_{\K(1)}\K(\ku_p) \to \L_{\K(1)}\K(\Z_p)$ canonically splits, so it suffices to see that the composition
  \[
    \L_{\K(1)}\K(\ku_p) \to \L_{\K(1)}\K(\Z_p) \to \L_{\K(1)}\K(\ku_p)
  \]
  is canonically null. This is true even before $\K(1)$-localization: it is given by multiplication by the class $[\Z_p] \in \K(\ku_p)$, which identifies with zero by virtue of the cofiber sequence of $\ku_p$-modules
  \[
    \ku_p[2] \lblto{\beta} \ku_p \to \Z_p.
  \]
\end{example}

\section{Noncommutative crystalline--de Rham comparison}
\label{c}

In this section, we will prove \cref{i--pv--thm}, most of which, we note again, was proved earlier and in a different manner by Petrov--Vologodsky \cite{petrov-vologodsky--TP}. We begin  in \cref{c--nil} with some further general discussion of nilpotence, continuing from \cref{j--nil}. The proof of the theorem is then given in \cref{c--pf}.


\subsection{Nilpotence continued}
\label{c--nil}

\begin{definition}
  \label{c--tate-category}
  Let $R$ be an $\cir$-equivariant $\E_\infty$-ring. Let $\smash{\cat{I_R} \subseteq \Mod_R(\Spt^{\B\cir})}$ denote the full subcategory spanned by those $\cir$-equivariant $R$-modules $X$ such that $X/(p,v_1)$ is nilpotent. We define
  \[
    \smash{\cpl{(\Mod_R^{\t\cir})_{(p,v_1)}} := \Mod_R(\Spt^{\B\cir})/\cat{I_R}}
  \]
  to be the Verdier quotient; since $\cat{I_R}$ is a thick tensor ideal, this and the associated quotient functor carry canonical symmetric monoidal structures.
\end{definition}

\begin{lemma}
  \label{c--tate-factorization}
  Let $R$ be an $\cir$-equivariant $\E_\infty$-ring. Then the lax symmetric monoidal functor
  \[
    \cpl{((-)^{\t\cir})}_{(p,v_1)} : \Mod_R(\Spt^{\B\cir}) \to \Spt
  \]
  factors uniquely through the symmetric monoidal quotient functor $\smash{\Mod_R(\Spt^{\B\cir}) \to \cpl{(\Mod^{\t\cir}_R)_{(p,v_1)}}}$.
\end{lemma}

\begin{proof}
  This follows from the definition of $\cpl{(\Mod_R^{\t\cir})_{(p,v_1)}}$ and \cref{j--nil--mod-p-v1}.
\end{proof}

\begin{construction}
  \label{c--tate-adjunction}
  Let $R \to R'$ be a map of $\cir$-equivariant $\E_\infty$-rings. By \cref{j--nil--adjunction}, the base change functor $\smash{\Mod_R(\Spt^{\B\cir}) \to \Mod_{R'}(\Spt^{\B\cir})}$ carries $\cat{I}_R$ into $\cat{I}_{R'}$ and the forgetful functor $\smash{\Mod_{R'}(\Spt^{\B\cir}) \to \Mod_R(\Spt^{\B\cir})}$ carries $\cat{I}_{R'}$ into $\cat{I}_R$. It follows that there is an induced adjunction
  \begin{equation}
    \label{c--tate-adjunction--adj}
    \cpl{(\Mod^{\t\cir}_R)_{(p,v_1)}} \fromto \cpl{(\Mod^{\t\cir}_{R'})_{(p,v_1)}}.
  \end{equation}
\end{construction}

\begin{lemma}
  \label{c--tate-equivalence}
  Let $\gamma : R \to R'$ be a map of $\cir$-equivariant $\E_\infty$-rings. Suppose that $\fib(\gamma) \in \cat{I}_R$. Then the adjunction \cref{c--tate-adjunction--adj} is an equivalence.
\end{lemma}

\begin{proof}
  For $X$ an $\cir$-equivariant $R$-module, the fiber of the unit map $X \to R' \otimes_R X$ is equivalent to $\fib(\gamma) \otimes_R X$, which by our hypothesis lies in $\cat{I}_R$. This implies that the unit transformation of \cref{c--tate-adjunction--adj} is an equivalence.

  We now address the counit transformation. For $X'$ an $\cir$-equivariant $R'$-module, the counit map $R' \otimes_R X' \to X'$ may be rewritten as $(R' \otimes_R R') \otimes_{R'} X' \to R' \otimes_{R'} X'$. Thus, similarly to the previous paragraph, it suffices to check that the fiber of the map $R' \otimes_R R' \to R'$ lies in $\cat{I}_{R'}$, in other words that the counit transformation of \cref{c--tate-adjunction--adj} is an equivalence on the image of $R'$ in $\smash{\cpl{(\Mod^{\t\cir}_{R'})_{(p,v_1)}}}$. But this object is in the image of the left adjoint functor, so the claim follows from the previous paragraph.
\end{proof}


\subsection{The proof}
\label{c--pf}

Recall from \cref{t--pf--map,t--pf--map-comparison} that there are unique maps of cyclotomic $\E_\infty$-rings $\alpha : \j_p^\triv \to \pcpl{\THH(\Z)}$ and $\o\alpha : \Z_p^\triv \to \THH(\F_p)$.

\begin{proposition}
  \label{c--commutative-diagram}
  Consider the diagram of $\cir$-equivariant $\E_\infty$-$\j_p$-algebras
  \[
    \begin{tikzcd}
      \j_p^\triv \ar[r, "\alpha"] \ar[d] &
      \pcpl{\THH(\Z)} \ar[d] \ar[dl] \\
      \Z_p^\triv \ar[r, "\o\alpha"] &
      \THH(\F_p),
    \end{tikzcd}
  \]
  where the left hand vertical map and diagonal map are the truncation maps and the right hand vertical map is induced by the reduction map $\Z \to \F_p$. After applying the quotient functor
  \[
    \smash{\Mod_{\j_p^\triv}(\Spt^{\B\cir}) \to \cpl{(\Mod_{\j_p^\triv}^{\t\cir})_{(p,v_1)}}},
  \]
  the maps $\alpha$ and $\o{\alpha}$ become equivalences and this diagram commutes.
\end{proposition}

\begin{proof}
  That $\alpha$ is inverted by the quotient functor follows from \cref{i--main--thhZp,i--nil--thm}; and for $\o\alpha$ it follows from the analogous result that $\fib(\o\alpha)/p$ is nilpotent (see \cite[Construction 2.6 and Lemma 2.7]{ammn--beilinson}). The upper triangle evidently commutes in $\smash{\CAlg^{\B\cir}}$, and, as discussed in \cref{t--pf--map-comparison}, so too does the outer square. It follows immediately then that the image of the lower triangle under the quotient functor also commutes.
\end{proof}

\begin{corollary}
  \label{c--general-result}
  For $M$ an $\cir$-equivariant $\pcpl{\THH(\Z)}$-module, there is a natural equivalence
  \[
    \pcpl{((M \otimes_{\pcpl{\THH(\Z)}} \Z_p^\triv)^{\t\cir})} \iso
    \pcpl{((M \otimes_{\pcpl{\THH(\Z)}} \THH(\F_p))^{\t\cir})},
  \]
  which is lax symmetric monoidal and which in the case $M = \pcpl{\THH(\Z)}$ recovers the equivalence $\smash{\Z_p^{\t\cir} \iso \TP(\F_p)}$ induced by the map $\o\alpha$.
\end{corollary}

\begin{proof}
  Combining \cref{c--tate-equivalence,c--commutative-diagram}, we see that $\o\alpha$ defines an equivalence between the images of $\Z_p^\triv$ and $\THH(\F_p)$ in $\smash{\CAlg(\cpl{(\Mod^{\t\cir}_{\pcpl{\THH(\Z)}})}_{(p,v_1)})}$. The claim then follows from \cref{c--tate-factorization,j--nil--mod-p-v1}.
\end{proof}

\begin{proof}[Proof of \cref{i--pv--thm}]
  We apply \cref{c--general-result} to $M = \pcpl{\THH(\cat{C})}$.
\end{proof}

\bigskip

\printbibliography

\end{document}